\documentclass[11pt,reqno]{amsart}

\usepackage{amsmath, amsthm, amsopn, amssymb, enumerate, tikz, mathtools}
\usepackage[english]{babel}
\usetikzlibrary{arrows,automata}
\usetikzlibrary{decorations,decorations.pathmorphing}

\usepackage{verbatim}
\usepackage{tikz}

\tikzstyle randomWalkPathStyle=[thin,opacity=0.5]

\newcounter{x}
\newcounter{c}
\newcounter{lastY}

\newcommand{\drawRandomWalk}[6]
{
    \begin{scope}

	\tikzstyle{every node}+=[circle, fill=black]
	\setcounter{c}{0}
	\setcounter{x}{#1}

        \foreach \y in #2 {

		\ifnum\value{c} > 0
			\draw[gray,randomWalkPathStyle] (\value{x} - 1,\value{lastY}) -- (\value{x},\y);
			\ifnum #3 < 99 \draw[gray,dashed,randomWalkPathStyle] (\value{x} - 1,#3) -- (\value{x},#3); \fi
			\ifnum \value{c} > 1
				\node at (\value{x} - 1,\value{lastY}) [fill=#6] {};
			\else
				\node at (\value{x} - 1,\value{lastY}) [fill=#4] {};
			\fi
		\fi

		\setcounter{lastY}{\y}
		\stepcounter{x}
		\stepcounter{c}
        }

	\node at (\value{x} - 1,\value{lastY}) [fill=#5] {};

    \end{scope}
}

\newtheorem{thm}{Theorem}[section]
\newtheorem{lemma}[thm]{Lemma}
\newtheorem{prop}[thm]{Proposition}
\newtheorem{cor}[thm]{Corollary}
\newtheorem{claim}[thm]{Claim}
\theoremstyle{definition}
\newtheorem*{remark}{Remark}
\newtheorem*{conj}{Conjecture}

\renewcommand{\subsubsection}[1] {\smallskip \noindent {\bf #1.}}

\newcommand{\cP}{\mathcal{P}}
\newcommand{\cC}{\mathcal{C}}
\newcommand{\cB}{\mathcal{B}}

\newcommand{\cD}{\mathcal{D}}
\newcommand{\cI}{\mathcal{I}}
\newcommand{\N}{\mathbb{N}}
\newcommand{\Z}{\mathbb{Z}}
\newcommand{\R}{\mathbb{R}}
\newcommand{\E}{\mathbb{E}}
\newcommand{\Var}{\text{Var}}
\renewcommand{\Pr}{\mathbb{P}}
\DeclareMathOperator\Hom{Hom}
\DeclareMathOperator\Rng{Rng}
\DeclareMathOperator\per{per}
\DeclareMathOperator\lcm{lcm}
\def\eqd{\,{\buildrel d \over =}\,}

\title{Random Walk with Long-Range Constraints}
\date{\today}

\author{Ron Peled}
\address{Ron Peled\hfill\break
Tel Aviv University\\
School of Mathematical Sciences\\
Tel Aviv, 69978, Israel.}
\email{peledron@post.tau.ac.il}
\urladdr{http://www.math.tau.ac.il/~peledron}

\author{Yinon Spinka}
\address{Yinon Spinka\hfill\break
Tel Aviv University\\
School of Mathematical Sciences\\
Tel Aviv, 69978, Israel.}
\email{yinonspi@post.tau.ac.il}
\urladdr{http://www.math.tau.ac.il/~yinonspi}

\thanks{Research of R.P. and Y.S. supported by an ISF grant and an IRG grant.}
\begin{document}

\begin{abstract}
We consider a model of a random height function with long-range
constraints on a discrete segment. This model was suggested by
Benjamini, Yadin and Yehudayoff and is a generalization of simple
random walk. The random function is uniformly sampled from all graph
homomorphisms from the graph $P_{n,d}$ to the integers $\Z$, where
the graph $P_{n,d}$ is the discrete segment $\{0,1,\ldots, n\}$ with
edges between vertices of different parity whose distance is at most
$2d+1$. Such a graph homomorphism can be viewed as a height function
whose values change by exactly one along edges of the graph
$P_{n,d}$. We also consider a similarly defined model on the
discrete torus.

Benjamini, Yadin and Yehudayoff conjectured that this model
undergoes a phase transition from a delocalized to a localized phase
when $d$ grows beyond a threshold $c\log n$. We establish this
conjecture with the precise threshold $\log_2 n$. Our results
provide information on the typical range and variance of the height
function for every given pair of $n$ and $d$, including the critical
case when $d-\log_2 n$ tends to a constant.

In addition, we identify the local limit of the model, when $d$ is
constant and $n$ tends to infinity, as an explicitly defined Markov
chain.
\end{abstract}

\maketitle

\section{Introduction}
\label{sec:introduction}

Given two graphs $G$ and $H$, a {\em graph homomorphism} from $G$ to
$H$ is a function $f\colon V(G) \to V(H)$ such that if $x$ and $y$
are neighbors in $G$, then $f(x)$ and $f(y)$ are neighbors in $H$. A
graph homomorphism from a graph $G$ to $\mathbb{Z}$ is then a map
from the vertex set of $G$ to the integers, that maps adjacent
vertices to adjacent integers. For a given vertex $v_0 \in G$, we
denote by $\Hom(G, v_0)$ the set of all homomorphisms from $G$ to
$\mathbb{Z}$, which map $v_0$ to $0$. Precisely,
\[ \Hom(G,v_0) := \big\{ f \colon V(G) \to \Z ~|~ f(v_0)=0,\, |f(x)-f(y)|=1 \text{ when }(x,y) \in E(G)\big\} .\]

The set $\Hom(G,v_0)$ is non-empty and finite when $G$ is finite,
bipartite and connected. Benjamini, H{\"a}ggstr{\"o}m and Mossel
\cite{Benjamini2000} initiated the study of random
$\Z$-homomorphisms, that is, uniformly chosen elements of
$\Hom(G,v_0)$. Special cases of this model include the simple random
walk, when $G=\{0,1,\ldots, n\}$ with nearest-neighbor connections,
the random walk bridge, when $G$ is a cycle, and the branching
random walk, when $G$ is a tree. The model is sometimes referred to
as a $G$-indexed random walk.
The behavior of typical $\Z$-homomorphisms is poorly understood for general graphs $G$.
Beyond simple and branching random walks, results are available mainly for the hypercube \cite{Kahn,Galvin}, high-dimensional cubic lattices \cite{peled2010high} and expander
and tree graphs \cite{Benjamini2000,peled2012lipschitz}. In particular, the case when
$G=\Z_{2n}^2$, a two-dimensional discrete torus, appears completely
open. This case is related to the 6-vertex, square-ice and
antiferromagnetic 3-state Potts models of statistical physics (see \cite{peled2010high}).

Benjamini, Yadin and Yehudayoff \cite{Benjamini} suggested the study
of this model when $G= T_{n,d}$ is a certain one-dimensional graph
with long-range edges, defined below. In this work we study the
properties of the model on this graph, as well as its close
relative, the graph $P_{n,d}$. Specifically, let $P_{n,d}$, for
$n,d\ge 1$, be the graph defined by
\begin{equation}
\label{eq:def-graph-P_n_d}
\begin{aligned}
  V(P_{n,d}) &:= \{0,1,\ldots, n\},\\
  E(P_{n,d}) &:= \{(i,j)\ |\ |i-j|\in\{1,3,\ldots, 2d+1\}\}.
\end{aligned}
\end{equation}
Thus, a uniformly chosen random function $f$ from $\Hom(P_{n,d}, 0)$
is a simple random walk conditioned on satisfying $|f(i) - f(j)|=1$
whenever $i,j$ have different parity and are at distance at most
$2d+1$. Figure~\ref{fig:samples} shows a typical sample from
$\Hom(P_{n,d},0)$. Similarly, let $T_{n,d}$, $n\ge 1$ even and $d
\ge 1$, be the graph defined by
\begin{equation}
\label{eq:def-graph-T_n_d}
\begin{aligned}
  V(T_{n,d}) &:= \{0,1,\ldots, n-1\} , \\
  E(T_{n,d}) &:= \big\{(i,j) ~|~ \min\{|i-j|, n-|i-j|\} \in \{1,3,\ldots, 2d+1\} \big\} .
\end{aligned}
\end{equation}
Thus, a uniformly chosen random function $f$ in $\Hom(T_{n,d},0)$ is
a simple random walk \emph{bridge} conditioned on satisfying $|f(i)
- f(j)|=1$ whenever $i,j$ have different parity and are at distance
at most $2d+1$ on the cycle.

In the rest of the paper we abbreviate $\Z$-homomorphisms to
homomorphisms. We shall loosely refer to homomorphisms on $P_{n,d}$
as being on the line, and to homomorphisms on $T_{n,d}$ as being on
the torus.

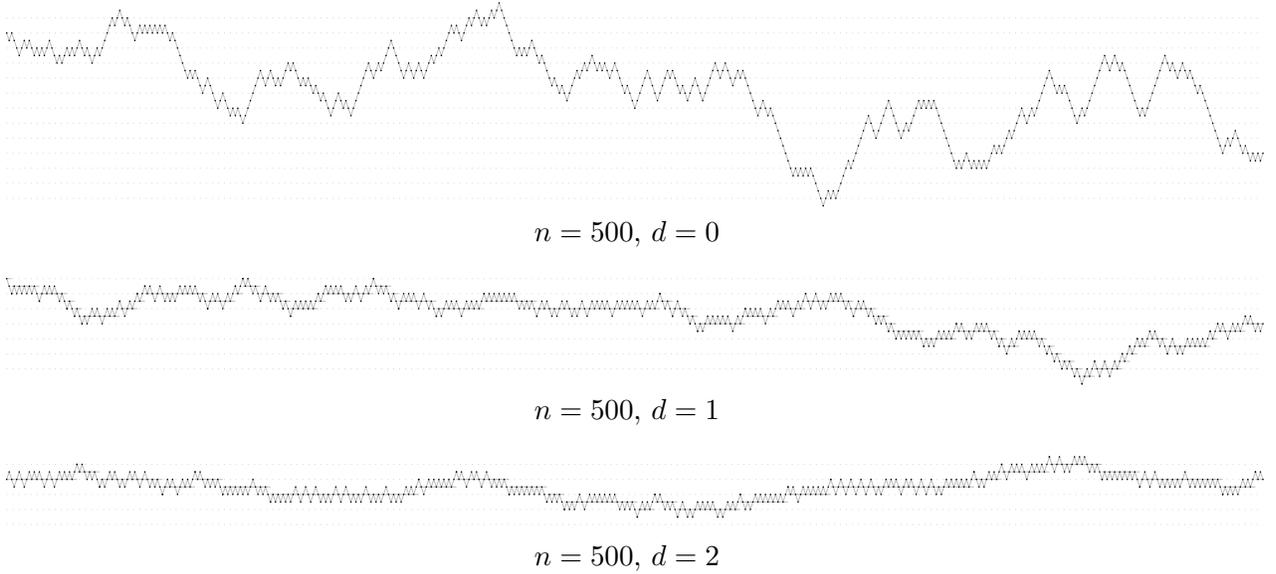
\begin{figure}[!t]
\vspace{8pt}
\centering
\begin{tikzpicture}

    \pgftransformcm{0.0334}{0}{0}{0.1}{\pgfpoint{0cm}{0cm}}
    \tikzstyle{every node}=[minimum size=0.02cm,inner sep=0]
    \draw[line width=0.1mm,gray!40,dotted](0,2)--(500,2);
    \draw[line width=0.1mm,gray!40,dotted](0,0)--(500,0);
    \draw[line width=0.1mm,gray!40,dotted](0,-2)--(500,-2);
    \draw[line width=0.1mm,gray!40,dotted](0,-4)--(500,-4);
    \draw[line width=0.1mm,gray!40,dotted](0,-6)--(500,-6);
    \draw[line width=0.1mm,gray!40,dotted](0,-8)--(500,-8);
    \draw[line width=0.1mm,gray!40,dotted](0,-10)--(500,-10);
    \draw[line width=0.1mm,gray!40,dotted](0,-12)--(500,-12);
    \draw[line width=0.1mm,gray!40,dotted](0,-14)--(500,-14);
    \draw[line width=0.1mm,gray!40,dotted](0,-16)--(500,-16);
    \draw[line width=0.1mm,gray!40,dotted](0,-18)--(500,-18);
    \draw[line width=0.1mm,gray!40,dotted](0,-20)--(500,-20);
    \draw[line width=0.1mm,gray!40,dotted](0,-22)--(500,-22);
    \drawRandomWalk{0}{{0,-1,0,-1,-2,-3,-2,-1,-2,-1,-2,-3,-2,-3,-2,-3,-2,-1,-2,-3,-4,-3,-4,-3,-2,-3,-2,-3,-2,-1,-2,-3,-2,-3,-4,-3,-2,-3,-2,-1,0,1,2,1,2,3,2,1,2,1,0,-1,0,1,0,1,0,1,0,1,0,1,0,1,0,-1,0,-1,-2,-3,-4,-5,-6,-5,-6,-5,-6,-7,-8,-7,-6,-7,-8,-9,-10,-9,-8,-9,-10,-11,-10,-11,-10,-11,-12,-11,-10,-9,-8,-7,-6,-5,-6,-7,-6,-5,-6,-7,-6,-7,-6,-5,-4,-5,-4,-5,-6,-7,-6,-7,-6,-7,-8,-7,-8,-9,-8,-9,-10,-11,-10,-9,-8,-9,-10,-9,-10,-11,-10,-9,-8,-7,-6,-5,-4,-5,-6,-5,-4,-5,-4,-3,-2,-1,-2,-3,-4,-5,-6,-5,-4,-5,-6,-5,-4,-5,-6,-5,-4,-3,-4,-3,-2,-3,-2,-1,0,1,0,1,0,-1,0,1,2,1,2,3,2,1,2,1,2,3,2,3,4,3,2,1,0,-1,-2,-3,-2,-3,-2,-3,-2,-1,-2,-3,-4,-3,-4,-5,-6,-7,-6,-7,-8,-7,-8,-9,-8,-7,-6,-5,-6,-5,-4,-5,-4,-3,-4,-5,-4,-5,-4,-5,-6,-5,-4,-5,-6,-7,-8,-7,-8,-9,-10,-9,-8,-7,-6,-5,-6,-7,-8,-9,-8,-7,-6,-5,-6,-5,-6,-7,-8,-7,-8,-9,-8,-7,-6,-7,-8,-9,-8,-7,-6,-5,-4,-5,-6,-5,-4,-5,-6,-7,-6,-5,-6,-5,-6,-7,-8,-9,-10,-11,-10,-9,-10,-11,-10,-11,-12,-13,-14,-15,-16,-17,-18,-19,-18,-19,-18,-19,-18,-19,-18,-19,-20,-21,-22,-23,-22,-21,-22,-21,-22,-21,-20,-19,-18,-17,-18,-17,-16,-15,-14,-13,-12,-11,-12,-13,-14,-13,-12,-11,-10,-9,-10,-11,-12,-13,-14,-13,-12,-13,-12,-11,-10,-9,-10,-9,-10,-9,-10,-9,-10,-11,-12,-13,-14,-15,-16,-17,-18,-17,-18,-17,-16,-17,-18,-17,-18,-17,-18,-17,-18,-17,-16,-15,-16,-15,-16,-15,-14,-13,-14,-13,-12,-11,-10,-11,-12,-11,-10,-11,-10,-9,-8,-7,-6,-5,-6,-7,-8,-7,-8,-7,-8,-9,-10,-11,-12,-11,-10,-11,-10,-9,-8,-7,-6,-5,-4,-3,-4,-5,-4,-3,-4,-5,-4,-5,-6,-7,-8,-9,-10,-11,-10,-11,-10,-9,-8,-7,-6,-5,-4,-3,-4,-5,-4,-5,-6,-5,-4,-5,-6,-5,-6,-7,-8,-7,-8,-9,-10,-11,-12,-13,-14,-15,-16,-15,-14,-15,-14,-13,-14,-15,-16,-15,-16,-17,-16,-17,-16,-17,-16}}{99}{}{}{};

\end{tikzpicture}
$n=500$, $d=0$

\makebox[0cm]{\vspace{5pt}}

\begin{tikzpicture}

    \pgftransformcm{0.0334}{0}{0}{0.1}{\pgfpoint{0cm}{0cm}}
    \tikzstyle{every node}=[minimum size=0.02cm,inner sep=0]
    \draw[line width=0.1mm,gray!40,dotted](0,0)--(500,0);
    \draw[line width=0.1mm,gray!40,dotted](0,-2)--(500,-2);
    \draw[line width=0.1mm,gray!40,dotted](0,-4)--(500,-4);
    \draw[line width=0.1mm,gray!40,dotted](0,-6)--(500,-6);
    \draw[line width=0.1mm,gray!40,dotted](0,-8)--(500,-8);
    \draw[line width=0.1mm,gray!40,dotted](0,-10)--(500,-10);
    \draw[line width=0.1mm,gray!40,dotted](0,-12)--(500,-12);
    \drawRandomWalk{0}{{0,-1,-2}}{0}{}{}{};
    \drawRandomWalk{2}{{-2,-1,-2,-1,-2,-1,-2,-1,-2,-1,-2,-3}}{-1}{}{}{};
    \drawRandomWalk{13}{{-3,-2,-1,-2,-1,-2,-1,-2,-3,-2,-3,-4}}{-2}{}{}{};
    \drawRandomWalk{24}{{-4,-3,-4,-5}}{-3}{}{}{};
    \drawRandomWalk{27}{{-5,-4,-5,-6}}{-4}{}{}{};
    \drawRandomWalk{30}{{-6,-5,-6,-5,-4,-5,-4,-5,-6,-5,-4,-5,-4,-5,-4,-3}}{-5}{}{}{};
    \drawRandomWalk{45}{{-3,-4,-5,-4,-3,-4,-3,-2}}{-4}{}{}{};
    \drawRandomWalk{52}{{-2,-3,-2,-1}}{-3}{}{}{};
    \drawRandomWalk{55}{{-1,-2,-1,-2,-3,-2,-1,-2,-3,-2,-3,-2,-3,-2,-1,-2,-1,-2,-1,-2,-1,-2,-3,-2,-3,-4}}{-2}{}{}{};
    \drawRandomWalk{80}{{-4,-3,-2,-3,-2,-3,-4,-3,-2,-3,-2,-1}}{-3}{}{}{};
    \drawRandomWalk{91}{{-1,-2,-1,0}}{-2}{}{}{};
    \drawRandomWalk{94}{{0,-1,0,-1,-2,-1,-2,-3}}{-1}{}{}{};
    \drawRandomWalk{101}{{-3,-2,-1,-2,-3,-2,-3,-2,-3,-4}}{-2}{}{}{};
    \drawRandomWalk{110}{{-4,-3,-4,-5}}{-3}{}{}{};
    \drawRandomWalk{113}{{-5,-4,-3,-4,-3,-4,-3,-4,-3,-4,-3,-2}}{-4}{}{}{};
    \drawRandomWalk{124}{{-2,-3,-2,-1}}{-3}{}{}{};
    \drawRandomWalk{127}{{-1,-2,-1,-2,-1,-2,-1,-2,-3,-2,-3,-2,-1,-2,-3,-2,-1,-2,-1,0}}{-2}{}{}{};
    \drawRandomWalk{146}{{0,-1,-2,-1,-2,-1,-2,-3}}{-1}{}{}{};
    \drawRandomWalk{153}{{-3,-2,-3,-4}}{-2}{}{}{};
    \drawRandomWalk{156}{{-4,-3,-2,-3,-2,-3,-2,-3,-4,-3,-2,-3,-4,-3,-4,-5}}{-3}{}{}{};
    \drawRandomWalk{171}{{-5,-4,-5,-4,-3,-4,-3,-4,-3,-4,-5,-4,-3,-4,-3,-4,-3,-4,-3,-2}}{-4}{}{}{};
    \drawRandomWalk{190}{{-2,-3,-2,-3,-2,-3,-2,-3,-2,-3,-2,-3,-2,-3,-4,-3,-4,-3,-4,-3,-4,-5}}{-3}{}{}{};
    \drawRandomWalk{211}{{-5,-4,-3,-4,-3,-4,-5,-4,-5,-4,-3,-4,-3,-4,-5,-4,-5,-4,-3,-4,-5,-4,-3,-4,-3,-4,-3,-4,-3,-4,-5,-4,-3,-4,-3,-4,-3,-4,-3,-4,-3,-4,-5,-4,-3,-4,-3,-4,-3,-2}}{-4}{}{}{};
    \drawRandomWalk{260}{{-2,-3,-4,-3,-4,-5}}{-3}{}{}{};
    \drawRandomWalk{265}{{-5,-4,-3,-4,-5,-4,-5,-6}}{-4}{}{}{};
    \drawRandomWalk{272}{{-6,-5,-6,-7}}{-5}{}{}{};
    \drawRandomWalk{275}{{-7,-6,-7,-6,-5,-6,-5,-6,-5,-6,-5,-6,-5,-6,-7,-6,-5,-6,-5,-4}}{-6}{}{}{};
    \drawRandomWalk{294}{{-4,-5,-4,-5,-4,-5,-4,-5,-6,-5,-4,-5,-4,-3}}{-5}{}{}{};
    \drawRandomWalk{307}{{-3,-4,-3,-4,-3,-4,-5,-4,-3,-4,-3,-2}}{-4}{}{}{};
    \drawRandomWalk{318}{{-2,-3,-4,-3,-2,-3,-4,-3,-4,-3,-2,-3,-2,-3,-2,-3,-4,-3,-4,-5}}{-3}{}{}{};
    \drawRandomWalk{337}{{-5,-4,-3,-4,-3,-4,-5,-4,-5,-6}}{-4}{}{}{};
    \drawRandomWalk{346}{{-6,-5,-6,-5,-6,-7}}{-5}{}{}{};
    \drawRandomWalk{351}{{-7,-6,-7,-8}}{-6}{}{}{};
    \drawRandomWalk{354}{{-8,-7,-8,-7,-8,-7,-8,-7,-8,-7,-8,-9}}{-7}{}{}{};
    \drawRandomWalk{365}{{-9,-8,-9,-8,-9,-8,-7,-8,-7,-8,-7,-8,-7,-6}}{-8}{}{}{};
    \drawRandomWalk{378}{{-6,-7,-6,-7,-8,-7,-8,-7,-6,-7,-6,-7,-6,-7,-8,-7,-8,-9}}{-7}{}{}{};
    \drawRandomWalk{395}{{-9,-8,-9,-10}}{-8}{}{}{};
    \drawRandomWalk{398}{{-10,-9,-8,-9,-8,-7}}{-9}{}{}{};
    \drawRandomWalk{403}{{-7,-8,-7,-8,-7,-8,-7,-8,-9,-8,-9,-10}}{-8}{}{}{};
    \drawRandomWalk{414}{{-10,-9,-10,-11}}{-9}{}{}{};
    \drawRandomWalk{417}{{-11,-10,-11,-12}}{-10}{}{}{};
    \drawRandomWalk{420}{{-12,-11,-12,-11,-12,-13}}{-11}{}{}{};
    \drawRandomWalk{425}{{-13,-12,-13,-14}}{-12}{}{}{};
    \drawRandomWalk{428}{{-14,-13,-12,-13,-12,-11}}{-13}{}{}{};
    \drawRandomWalk{433}{{-11,-12,-13,-12,-11,-12,-13,-12,-11,-12,-11,-10}}{-12}{}{}{};
    \drawRandomWalk{444}{{-10,-11,-10,-9}}{-11}{}{}{};
    \drawRandomWalk{447}{{-9,-10,-9,-8}}{-10}{}{}{};
    \drawRandomWalk{450}{{-8,-9,-8,-9,-8,-7}}{-9}{}{}{};
    \drawRandomWalk{455}{{-7,-8,-7,-8,-9,-8,-9,-10}}{-8}{}{}{};
    \drawRandomWalk{462}{{-10,-9,-8,-9,-10,-9,-10,-9,-8,-9,-8,-9,-8,-9,-8,-9,-8,-7}}{-9}{}{}{};
    \drawRandomWalk{479}{{-7,-8,-7,-6}}{-8}{}{}{};
    \drawRandomWalk{482}{{-6,-7,-6,-7,-8,-7,-6,-7,-6,-7,-6,-5}}{-7}{}{}{};
    \drawRandomWalk{493}{{-5,-6,-5,-6,-7,-6,-7,-6}}{-6}{}{}{};

\end{tikzpicture}
$n=500$, $d=1$

\makebox[0cm]{\vspace{5pt}}

\begin{tikzpicture}

    \pgftransformcm{0.0334}{0}{0}{0.1}{\pgfpoint{0cm}{0cm}}
    \tikzstyle{every node}=[minimum size=0.02cm,inner sep=0]
    \draw[line width=0.1mm,gray!40,dotted](0,2)--(500,2);
    \draw[line width=0.1mm,gray!40,dotted](0,0)--(500,0);
    \draw[line width=0.1mm,gray!40,dotted](0,-2)--(500,-2);
    \draw[line width=0.1mm,gray!40,dotted](0,-4)--(500,-4);
    \draw[line width=0.1mm,gray!40,dotted](0,-6)--(500,-6);
    \drawRandomWalk{0}{{0,1,0,-1,0,1,0,-1,0,1,0,1,0,1,0,-1,0,1,0,-1,0,1,0,1,0,1,0,1,2}}{0}{}{}{};
    \drawRandomWalk{28}{{2,1,2,1,0,1,0,1,0,-1}}{1}{}{}{};
    \drawRandomWalk{37}{{-1,0,-1,0,1,0,1,0,-1,0,-1,0,1,0,1,0,-1,0,1,0,-1,0,-1,0,-1,-2}}{0}{}{}{};
    \drawRandomWalk{62}{{-2,-1,0,-1,0,-1,-2,-1,0,-1,0,-1,0,1}}{-1}{}{}{};
    \drawRandomWalk{75}{{1,0,1,0,-1,0,-1,0,-1,0,-1,-2}}{0}{}{}{};
    \drawRandomWalk{86}{{-2,-1,-2,-1,-2,-1,-2,-1,-2,-1,-2,-1,0,-1,-2,-1,-2,-1,-2,-3}}{-1}{}{}{};
    \drawRandomWalk{105}{{-3,-2,-3,-2,-3,-2,-3,-2,-3,-2,-1,-2,-3,-2,-1,-2,-3,-2,-1,-2,-3,-2,-3,-2,-3,-2,-1,-2,-3,-2,-1,-2,-3,-2,-3,-2,-3,-2,-1,-2,-3,-2,-3,-2,-3,-2,-1,-2,-3,-2,-3,-2,-3,-2,-1,-2,-1,-2,-1,0}}{-2}{}{}{};
    \drawRandomWalk{164}{{0,-1,-2,-1,0,-1,0,-1,0,-1,0,-1,0,-1,0,1}}{-1}{}{}{};
    \drawRandomWalk{179}{{1,0,1,0,-1,0,1,0,1,0,-1,0,1,0,-1,0,-1,0,-1,0,-1,-2}}{0}{}{}{};
    \drawRandomWalk{200}{{-2,-1,-2,-1,-2,-1,-2,-1,-2,-1,-2,-1,-2,-1,-2,-3}}{-1}{}{}{};
    \drawRandomWalk{215}{{-3,-2,-3,-2,-3,-2,-3,-4}}{-2}{}{}{};
    \drawRandomWalk{222}{{-4,-3,-4,-3,-4,-3,-2,-3,-4,-3,-2,-3,-2,-3,-2,-3,-2,-3,-2,-3,-2,-3,-4,-3,-4,-3,-4,-3,-4,-5}}{-3}{}{}{};
    \drawRandomWalk{251}{{-5,-4,-3,-4,-3,-4,-3,-2}}{-4}{}{}{};
    \drawRandomWalk{258}{{-2,-3,-2,-3,-4,-3,-4,-3,-4,-5}}{-3}{}{}{};
    \drawRandomWalk{267}{{-5,-4,-3,-4,-5,-4,-5,-4,-3,-4,-3,-4,-3,-4,-3,-4,-5,-4,-5,-4,-3,-4,-3,-4,-3,-2}}{-4}{}{}{};
    \drawRandomWalk{292}{{-2,-3,-4,-3,-4,-3,-2,-3,-2,-3,-2,-3,-2,-3,-2,-3,-2,-3,-2,-1}}{-3}{}{}{};
    \drawRandomWalk{311}{{-1,-2,-1,-2,-3,-2,-1,-2,-1,-2,-1,-2,-1,-2,-1,-2,-1,0}}{-2}{}{}{};
    \drawRandomWalk{328}{{0,-1,-2,-1,0,-1,-2,-1,0,-1,-2,-1,0,-1,-2,-1,0,-1,-2,-1,-2,-1,0,-1,0,-1,0,-1,-2,-1,0,-1,0,-1,-2,-1,0,-1,-2,-1,0,-1,-2,-1,-2,-1,0,-1,0,-1,0,-1,0,-1,0,-1,0,1}}{-1}{}{}{};
    \drawRandomWalk{385}{{1,0,-1,0,-1,0,1,0,1,0,1,2}}{0}{}{}{};
    \drawRandomWalk{396}{{2,1,0,1,2,1,2,1,2,1,0,1,2,1,2,1,2,1,2,3}}{1}{}{}{};
    \drawRandomWalk{415}{{3,2,1,2,3,2,1,2,1,2,3,2,3,2,3,2,1,2,1,2,1,0}}{2}{}{}{};
    \drawRandomWalk{436}{{0,1,0,1,0,1,0,1,0,1,0,1,0,1,0,-1}}{1}{}{}{};
    \drawRandomWalk{451}{{-1,0,1,0,-1,0,1,0,-1,0,-1,0,-1,0,-1,0,-1,0,1,0,-1,0,-1,0,-1,0,-1,0,-1,0,-1,0,-1,-2}}{0}{}{}{};
    \drawRandomWalk{484}{{-2,-1,-2,-1,-2,-1,-2,-1,0,-1,0,-1,0,1}}{-1}{}{}{};
    \drawRandomWalk{497}{{1,0,1,0}}{0}{}{}{};

\end{tikzpicture}
\raisebox{-0.5ex}{$n=500$, $d=2$}

\caption{Uniformly sampled homomorphisms in $\Hom(P_{n,d},0)$. The case $d=0$ is just a simple random walk. The simulation uses a Metropolis
algorithm (see, e.g., \cite[Chapter~3]{Peres2009markov}) and coupling from the past \cite{ProppWilson}.}
\label{fig:samples}
\vspace{8pt}
\end{figure}

Our main objects of study are the size of the range of a typical
homomorphism on $P_{n,d}$ or $T_{n,d}$ and the variance of the
homomorphism at given vertices. For a graph $G$, the {\em range} of a function $f \colon V(G) \to \Z$ is defined as
\[ \Rng(f) := \{ f(v) ~|~ v \in V(G) \} .\]
Benjamini, Yadin and Yehudayoff made the following conjecture.

\noindent {\bf Conjecture (\cite{Benjamini}).} There exist constants
$b,c>0$ such that if $f_{n,d}$ is uniformly sampled from
$\Hom(T_{n,d},0)$,
\begin{enumerate}
  \item If $d(n)-c\log n \to -\infty$ as $n \to \infty$ then for any positive integer $r$, we have
  \[ \Pr(|\Rng(f_{n,d(n)})| \leq r) \xrightarrow[n \to \infty]{} 0 . \]
  \item If $d(n)-c\log n \to \infty$ as $n \to \infty$ then
  \[ \Pr(|\Rng(f_{n,d(n)})| \leq b) \xrightarrow[n \to \infty]{} 1 . \]
\end{enumerate}
Our work establishes this conjecture with the precise constants
$b=3$ and $c= 1/\log 2$, both on $T_{n,d}$ and $P_{n,d}$. In
addition, we discover that in the subcritical regime, when $d(n) -
\log_2 n\to-\infty$, the size of the range of a typical homomorphism
is of order $\sqrt{n2^{-d}}$ and the variance of the homomorphism at
vertex $k$ is of order $k2^{-d}$. Moreover, we explore the behavior
in the critical regime, when $d(n) - \log_2 n\to\mu\in\R$, and find
that in this case, the size of the range is a tight random variable
whose distribution is closely related to the Poisson distribution.

Our results may be intuitively understood as follows. Let
$f\in\Hom(P_{n,d})$. It is not difficult to verify that if $f(i+m) -
f(i) \geq 3$ then $m \geq 2d+3$. Figure \ref{fig:single-jump-segment} shows such an event. Moreover, if $m=2d+3$ and this event
occurs, then necessarily
\[
  \big(f(i+1) - f(i), f(i+2) - f(i),\ldots, f(i+2d+3) - f(i)\big) =
  (1,2,1,2,\ldots, 1,2,1,2,3).
\]
However, if this sequence of values is possible for $f$, then there
are at least $2^d$ other possible candidates of the form
\[
  (1,1+s_1,1,1+s_2,\ldots, 1,1+s_{d},1,2,1),\quad
  s_i\in\{-1,1\}.
\]
Thus, intuitively, the probability that the homomorphism changes its
height by $3$ on any given small segment is about $2^{-d}$.
Therefore, when $n2^{-d} \to 0$, we will not have any such segment,
so that the size of the range of the homomorphism will be bounded by
$3$. Conversely, when $n2^{-d} \to \infty$, the expected number of
segments with an upward or downward movement of size $3$ will be
roughly $n2^{-d}$. Since the direction of these movements should be
only mildly correlated, we expect the size of the resulting range
to be of order $\sqrt{n2^{-d}}$. Our work makes these ideas
precise.

\begin{figure}[!t]
\centering
\vspace{8pt}
\begin{tikzpicture}
    \pgftransformcm{0.45}{0}{0}{0.55}{\pgfpoint{0cm}{0cm}}

    \draw[line width=0.1mm,gray,dotted](4,-1.8)--(4,2);
    \draw[line width=0.1mm,gray,dotted](13,-1.8)--(13,2);
    \node at (4,-1.8) [below] {$i$};
    \node at (13,-1.8) [below] {$k$};

    \drawRandomWalk{0}{{-1,0}}{99}{gray!25}{gray!25}{};
    \drawRandomWalk{1}{{0,1,0}}{0}{gray!25}{gray!25}{gray!25};
    \drawRandomWalk{3}{{0,-1,0}}{0}{gray!25}{}{orange};
    \drawRandomWalk{5}{{0,1,0,1,0,1,0,1,2}}{0}{}{red}{};
    \drawRandomWalk{13}{{2,1}}{1}{red}{gray!25}{};
    \drawRandomWalk{14}{{1,2,1}}{1}{gray!25}{gray!25}{gray!25};
    \drawRandomWalk{16}{{1,0,1}}{1}{gray!25}{gray!25}{gray!25};
\end{tikzpicture}
\caption{A homomorphism jumps from some value $t$ at vertex $i$ to $t+3$ at vertex $k$. The minimal length of such a segment is $k-i = 2d+3$. In order for this jump to occur, the values at the $d+1$ vertices, $k-1, k-3, \dots, k-2d-1$, are forced to be $t+2$. Here $d=3$.}
\label{fig:single-jump-segment}
\end{figure}
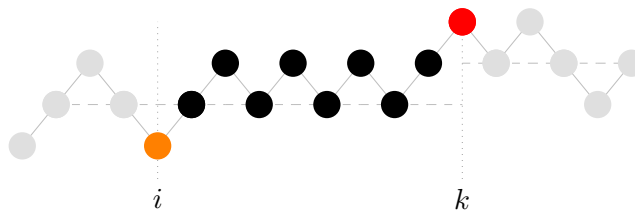

\section{Main Results}
\label{sec:main-results}

\subsection{Homomorphisms on the line}
\label{sec:main-results-line}

In this section we present results on homomorphisms on the graph
$P_{n,d}$, which was defined in \eqref{eq:def-graph-P_n_d}.
Throughout this section, $f_{n,d}$ denotes a uniformly chosen
homomorphism in $\Hom(P_{n,d},0)$.

We state results regarding the size of the range of a typical
homomorphism. As a homomorphism must change its value by exactly one
along edges, the range is always of size at least $2$. In fact, the
range is exactly $2$ only for two particular homomorphisms, and at
least $3$ otherwise. We shall show that the size of the range is $3$
plus a term of order $\sqrt{n2^{-d}}$. Hence, we distinguish between
three regimes, $n2^{-d} \to \infty$, $n2^{-d} \to 0$ and $n2^{-d}
\to \lambda \in (0,\infty)$, termed the subcritical regime, the
supercritical regime and the critical regime, respectively.

\smallskip
\noindent
{\bf The supercritical regime.}
The supercritical regime is when $d(n) - \log_2 n \to \infty$ (i.e. $n 2^{-d(n)} \to 0$) as $n \to \infty$. In this case, the large number of constraints prevents a typical homomorphism from growing. In fact, we show that, with high probability, it will take on only $3$ values.
\begin{thm}
\label{thm:line-supercritical} For any positive integers $n$, $d$
and $r$, we have
\[ \Pr\big( |\Rng(f_{n,d})| \geq 3 + r\big) \leq \binom{n}{r} 2^{-dr} \quad \text{and} \quad \Pr\big( |\Rng(f_{n,d})| < 3\big) \leq 2^{1-n/2} .\]
Thus, if $d(n) - \log_2 n \to \infty$ as $n \to \infty$ then
\[ \Pr\big(|\Rng(f_{n,d(n)})| = 3\big) \xrightarrow[n \to \infty]{} 1 .\]
\end{thm}
The following corollary gives more precise information about the
structure of a typical homomorphism in the supercritical regime.
Denote by $V_i := \{ 2k+i ~|~ 0 \leq 2k+i \leq n \}$, $i=0,1$, the
even and odd vertices, respectively, and denote by $\Omega_0$ and
$\Omega_1$ the set of homomorphisms which are constant on $V_0$ and
$V_1$, respectively. Note that for each $i \in \{0,1\}$, conditioned
on $f \in \Omega_i$, the random vector $(f(x) - f(i) ~|~ x \in
V_{1-i})$ consists of independent uniform signs.
\begin{cor}
\label{cor:line-supercritical}
If $d(n) - \log_2 n \to \infty$ as $n \to \infty$ then
\begin{align*}
\Pr(\Omega_0 \cup \Omega_1) &= \Pr\big(|\Rng(f_{n,d(n)})| \leq
3\big) \xrightarrow[n \to \infty]{} 1 \quad\text{ and } \\
\Pr(\Omega_0 \cap \Omega_1) &= \Pr\big(|\Rng(f_{n,d(n)})| < 3\big)
\xrightarrow[n \to \infty]{} 0 .
\end{align*}
Moreover, if $n$ tends to infinity through odd numbers then
$\Pr(\Omega_0) \to 1/2$, and if $n$ tends to infinity through even
numbers then $\Pr(\Omega_0) \to 1/3$.
\end{cor}

The corollary implies that a typical homomorphism in the
supercritical regime has one of three possible structures. For odd
values of $n$, with probability $1/2 - o(1)$, the homomorphism takes
the value $0$ on all the even vertices, with probability $1/4 -
o(1)$, it takes the value $1$ on all the odd vertices, and, with
probability $1/4 - o(1)$, it takes the value $-1$ on all the odd
vertices. For even values of $n$, the probability of each of these
three options is $1/3 - o(1)$. The dependence on the parity of $n$
arises from the difference in the number of even and odd vertices in
each case. For odd values of $n$, $|V_0|=|V_1|$, whereas for even
values of $n$, $|V_0|=|V_1|+1$.

\smallskip
\noindent {\bf The subcritical regime.} The subcritical regime is
when $d(n) - \log_2 n \to -\infty$ (i.e. $n 2^{-d(n)} \to \infty$)
as $n \to \infty$. Here, the relatively small number of constraints
allows a typical homomorphism to grow.

\begin{thm}
\label{thm:line-subcritical-range} There exist absolute constants
$C,c>0$ such that for any positive integers $n$ and $d$, we have
\[ 3 + \big\lfloor c \sqrt{n 2^{-d}} \big\rfloor - 2^{1-n/2} \leq \E\big[|\Rng(f_{n,d})|\big] \leq 3 + C \sqrt{n 2^{-d}} .\]
Moreover, for any $\epsilon > 0$ there exists a $\delta > 0$ such that for any positive integers $n$ and $d$, we have
\begin{equation}
\label{eq:thm-line-subcritical-range-large-whp}
\Pr\left(|\Rng(f_{n,d})| < 3 + \big\lfloor \delta \sqrt{n2^{-d}} \big\rfloor \right) \leq \epsilon + 2^{1-n/2} .
\end{equation}
In particular, if $d(n) - \log_2 n \to -\infty$ as $n \to \infty$ then for any positive integer $r$, we have
\[ \Pr(|\Rng(f_{n,d(n)})| \leq r) \xrightarrow[n \to \infty]{} 0 .\]
\end{thm}

The next theorem quantifies the rate of growth of the variance of
the homomorphism.
\begin{thm}
\label{thm:line-subcritical-endpoint} There exist absolute constants
$C,c>0$ such that for any positive integers $n$ and $d$, we have
\[ \max\{ck 2^{-d},1\} \leq \Var(f_{n,d}(k)) \leq C k 2^{-d} + 4, \quad 1\le k\le n.\]
\end{thm}

\smallskip
\noindent {\bf The critical regime.} The critical regime is when
$d(n) - \log_2 n \to - \log_2 \lambda$ (i.e. $n 2^{-d(n)} \to
\lambda$) as $n \to \infty$, for some $\lambda \in (0,\infty)$. In
this case, the balance between the number of constraints at each
vertex and the amount of time available leads to an interesting
limiting behavior. Perhaps surprisingly, it turns out that the
parity of $n$ induces an effect which does not disappear in the
limit.

Denote by $\mu^{\text{even}}(\lambda)$ the distribution of a Poisson($\lambda$) variable conditioned to be even, and denote by $\mu^{\text{odd}}(\lambda)$ the distribution of a Poisson($\lambda$) variable conditioned to be odd. Define the {\em parity-biased Poisson distribution} with parameters $\lambda$ and $\alpha$ to be the following convex combination of $\mu^{\text{even}}(\lambda)$ and $\mu^{\text{odd}}(\lambda)$,
\begin{equation}
\label{eq:parity-biased-poisson}
\mu(\lambda,\alpha) := \frac{\alpha}{\alpha + \tanh(\lambda)} \cdot \mu^{\text{even}}(\lambda) + \frac{\tanh(\lambda)}{\alpha + \tanh(\lambda)} \cdot \mu^{\text{odd}}(\lambda) .
\end{equation}
One may check that
\begin{equation}
\label{eq:parity-biased-poisson-prob}
\mu(\lambda,\alpha)(r) = Z(\lambda,\alpha)^{-1} \cdot \alpha(r) \cdot \frac{\lambda^r}{r!}, \quad r \geq 0 ,
\end{equation}
where $\alpha(r)=\alpha$ if $r$ is even and $\alpha(r)=1$ if $r$ is odd and where $Z(\lambda,\alpha)$ is a normalizing constant. In particular, we see that the Poisson($\lambda$) distribution is obtained as $\mu(\lambda,1)$.

Let $(S_i ~|~ i=0,1,\dots )$ denote a simple random walk, and let
\begin{equation*}
N^{\pm}(\lambda) \sim \mu\left(\lambda/(2\sqrt{2}),
(3/(2\sqrt{2}))^{\pm 1}\right)
\end{equation*}
be independent of $(S_i ~|~ i \geq 0)$. Then $S_{N^{+}(\lambda)}$
and $S_{N^{-}(\lambda)}$ are simple random walks stopped at
independent random times. For a positive integer $k$, denote
$\Rng(S_k) := \{ S_i ~|~ 0 \leq i \leq k \}$.

\begin{thm}
\label{thm:line-critical-range} If $n 2^{-d(n)} \to \lambda$ as $n
\to \infty$, for some $\lambda \in (0,\infty)$, then
\[ |\Rng(f_{n,d(n)})| \xrightarrow[\substack{n \to \infty \\ n \text{ even}}]{(d)} ~ |\Rng(S_{N^+(\lambda)})| +
2\quad\text{ and }\quad|\Rng(f_{n,d(n)})| \xrightarrow[\substack{n \to \infty \\
n \text{ odd}}]{(d)} ~ |\Rng(S_{N^-(\lambda)})| + 2 .\]
\end{thm}
In fact, as the proof shows, one may couple a critical homomorphism
to a simple random walk run for $N^+$ or $N^-$ steps, according to
the parity of $n$.

\subsection{Homomorphisms on the torus}
\label{sec:main-results-torus}

In this section we present results for homomorphisms on the graph $T_{n,d}$, which was defined in \eqref{eq:def-graph-T_n_d}. Throughout this section, $n$ is even and $f_{n,d}$ denotes a uniformly chosen homomorphism in $\Hom(T_{n,d},0)$.

\smallskip
\noindent
{\bf The supercritical regime.}
The supercritical regime is when $d(n) - \log_2 n \to \infty$ (i.e. $n 2^{-d(n)} \to 0$) as $n \to \infty$. Similarly to the case on the line, the large number of constraints cause a typical homomorphism to take on only $3$ values.

\begin{thm}
\label{thm:torus-supercritical}
For any positive even integer $n$ and any positive integers $d$ and $r$, we have
\[ \Pr\big(|\Rng(f_{n,d})| \geq 3 + r\big) \leq \binom{n}{r}^2 2^{-(2d-1)r} \quad \text{and} \quad \Pr\big( |\Rng(f_{n,d})| < 3\big) \leq 2^{1-n/2} \]
Thus, if $d(n) - \log_2 n \to \infty$ as $n \to \infty$ then
\[ \Pr\big(|\Rng(f_{n,d(n)})| = 3\big) \xrightarrow[n \to \infty]{} 1 .\]
\end{thm}

Similarly to the case of the line, the following corollary gives
more precise information about the structure of a typical
homomorphism in the supercritical regime. Denote by $V_i := \{ 2k+i
~|~ 0 \leq k < n/2 \}$, $i=0,1$, the even and odd vertices of
$T_{n,d}$, respectively, and denote by $\Omega_0$ and $\Omega_1$ the
set of homomorphisms which are constant on $V_0$ and $V_1$,
respectively. Note that for each $i \in \{0,1\}$, conditioned on $f
\in \Omega_i$, the random vector $(f(x) - f(i) ~|~ x \in V_{1-i})$
consists of independent uniform signs.

\begin{cor}
\label{cor:torus-supercritical}
If $d(n) - \log_2 n \to \infty$ as $n \to \infty$ then
\[ \Pr(\Omega_0 \cup \Omega_1) = \Pr\big(|\Rng(f_{n,d(n)})| \leq 3\big) \xrightarrow[n \to \infty]{} 1 \quad\text{ and }\quad
\Pr(\Omega_0) = \Pr(\Omega_1) \xrightarrow[n \to \infty]{} 1/2. \]
\end{cor}
Thus, a typical homomorphism in the supercritical regime is constant
on either the even or odd vertices of $T_{n,d}$, with the two cases
being equally likely. The effect induced by the parity of $n$ in
Corollary~\ref{cor:line-supercritical} does not appear here, as $n$
is always assumed to be even in the case of the torus.

\smallskip
\noindent {\bf The subcritical regime.} The subcritical regime is
when $d(n) - \log_2 n \to - \infty$ (i.e. $n 2^{-d(n)} \to \infty$)
as $n \to \infty$. As before, the relatively small number of
constraints allows a typical homomorphism to grow.

\begin{thm}
\label{thm:torus-subcritical-range} There exist absolute constants
$C,c>0$ such that for any positive even integer $n$ and any positive integer $d$, we have
\[ 3 + \big\lfloor c \sqrt{n 2^{-d}} \big\rfloor - 2^{1-n/2} \leq \E\big[|\Rng(f_{n,d})|\big] \leq C(\sqrt{n 2^{-d}}+1) .\]
Moreover, for any $\epsilon > 0$ there exists a $\delta > 0$ such that for any positive even integer $n$ and any positive integer $d$, we have
\[ \Pr\left(|\Rng(f_{n,d})| < 3 + \big\lfloor \delta \sqrt{n 2^{-d}} \big\rfloor \right) \leq \epsilon + 2^{1-n/2} .\]
In particular, if $d(n) - \log_2 n \to -\infty$ as $n \to \infty$ then for any positive integer $r$, we have
\[ \Pr\big(|\Rng(f_{n,d(n)})| \leq r\big) \xrightarrow[n \to \infty]{} 0 .\]
\end{thm}

\smallskip
\noindent {\bf The critical regime.} The critical regime is when
$d(n) - \log_2 n \to - \log_2 \lambda$ (i.e. $n 2^{-d(n)} \to
\lambda$) as $n \to \infty$, for some $\lambda \in (0,\infty)$. As
for the line, this choice of parameters leads to an interesting
limiting behavior. In this case, the random homomorphism behaves
similarly to a simple random walk \emph{bridge} of length $2N$,
where $N$ is an independent random variable whose distribution is a
type of biased Poisson distribution. The distribution of $N$ here is
biased differently from the case of the line. Specifically, $N$ has
the distribution of a Poisson random variable with parameter
\begin{equation}\label{eq:lambda_prime_def}
  \lambda' := \frac{\lambda}{4\sqrt{2}} ,
\end{equation}
conditioned to be equal to another such independent Poisson random
variable.

Denote by $\nu(\lambda')$ the distribution of $X$ conditioned on $X=Y$, where $X$ and $Y$ are independent Poisson($\lambda'$) random variables. One may check that
\begin{equation}
\label{eq:equality-biased-poisson}
\nu(\lambda')(k) = Z(\lambda')^{-1} \cdot \frac{(\lambda')^{2k}}{(k!)^2} , \quad k \geq 0 ,
\end{equation}
where $Z(\lambda')$ is a normalizing constant.

For a positive even integer $k$, let $(B^k_i ~|~ 0 \leq i \leq k )$
denote a simple random walk bridge of length $k$ (that is, a simple
random walk conditioned on $B^k_k=0$), and let $N(\lambda') \sim
\nu(\lambda')$ be an independent random variable. Thus,
$B^{2N(\lambda')}$ is obtained by first sampling $N(\lambda')$ and
then sampling a simple random walk bridge of length
$2N(\lambda')$. For a positive even integer $k$, denote $\Rng(B^k)
:= \{ B^k_i ~|~ 0 \leq i \leq k \}$.

\begin{thm}
  \label{thm:torus-critical-range}
If $n 2^{-d(n)} \to \lambda$ as $n \to \infty$, for some $\lambda
\in (0,\infty)$, then
\[ |\Rng(f_{n,d(n)})| \xrightarrow[n \to \infty]{(d)} \big|\Rng\big(B^{2N(\lambda')}\big)\big| +
2,\] where $\lambda'$ is defined by \eqref{eq:lambda_prime_def}.
\end{thm}

This theorem is closely related to Theorem
\ref{thm:line-critical-range}. On the line, the range of a
homomorphism in the critical regime is determined by a simple random
walk whose length is a parity-biased Poisson random variable. Note
that if we condition a simple random walk with a Poisson($\mu$)
number of steps to end at its initial value, then the number of
steps it takes has distribution $\nu(\mu/2)$. To see this, observe
that the number of positive and negative steps of the random walk
are independent Poisson($\mu/2$) random variables and we are
conditioning that these variables are equal. The same phenomenon
continues to hold if we start with a simple random walk taking a
parity-biased Poisson($\mu$,$\alpha$) number of steps. Indeed, the
number of steps must be even in order for the walk to end at its
initial value, and a parity-biased Poisson($\mu$,$\alpha$)
conditioned to be even is the same as a Poisson($\mu$) conditioned
to be even.

\subsection{Local limits on the line}
\label{sec:main-results-local-limit}

In this section we present results for homomorphisms on the graph
$P_{n,d}$, which was defined in \eqref{eq:def-graph-P_n_d}, when $d$
is constant and $n$ tends to infinity.

Our first result gives an approximate count of the number of
homomorphisms.
\begin{thm}
  \label{thm:local-limit-hom-count}
For any positive integer $d$ there exists a constant $C(d)>0$ such that
\[ |\Hom(P_{n,d},0)| = C(d)\lambda(d)^{n/2}(1+o(1)) \quad \text{ as } n \to \infty, \]
where $\lambda(d)$ is the unique positive solution to the equation
\[ \lambda^{d - 1/2} (\lambda - 2) = 1 . \]
\end{thm}

\begin{remark}
The constant $\lambda(d)$ above satisfies
\[ \lambda(d) = 2 + 2^{-d+1/2}(1-o(1)) \quad \text{ as } d \to \infty .\]
\end{remark}

Our next result concerns the local limit of the homomorphism. This
local limit lives on $P_{\infty,d}$, the limiting graph of
$P_{n,d}$. Precisely, $P_{\infty,d}$, for $d\ge 1$, is the graph
defined by
\begin{equation}
\label{eq:def-graph-P_infty_d}
\begin{aligned}
  V(P_{\infty,d}) &:= \{0,1,2,\ldots\},\\
  E(P_{\infty,d}) &:= \big\{(i,j) ~|~ |i-j|\in\{1,3,\ldots, 2d+1\}\big\}.
\end{aligned}
\end{equation}
For a function $g$ defined on a domain $\Omega$ and a set
$A\subseteq \Omega$, we write $g|_A$ for the restriction of $g$ to
$A$.
\begin{thm}[Local Limit]
  \label{thm:local-limit}
For any constant $d \geq 1$, there exists a distribution
$\mu_{\infty,d}$ on $\Hom(P_{\infty,d},0)$ such that the uniform
distribution on $\Hom(P_{n,d},0)$ converges to $\mu_{\infty,d}$ as
$n \to \infty$, in the following sense. Let $f_{n,d}$ be a uniformly
chosen homomorphism in $\Hom(P_{n,d},0)$ and let $f_{\infty,d}$ be
sampled from $\mu_{\infty,d}$. Then,
\[ \Pr(f_{n,d}|_{\{0,1,\dots,r\}} = f) \xrightarrow[n \to \infty]{} \Pr(f_{\infty,d}|_{\{0,1,\dots,r\}} = f) \quad \text{for any } r \geq 0 \text{ and } f \in \Z^{\{0,1,\dots,r\}} . \]
\end{thm}

\begin{remark}
The random homomorphism $f_{\infty,d}$ is described by an explicit
Markov chain on $2d+2$ states, as shown in Figure
\ref{fig:local-limit-markov-chain}, through a process which decodes
infinite words on the alphabet $\{a,b,A,B\}$ into homomorphisms on
$P_{\infty,d}$. See Section~\ref{sec:local-limits} for details.
\end{remark}

\noindent{\bf Policy on constants:} In the rest of the paper we
employ the following policy on constants. We write $C,c,C',c'$ for
positive absolute constants, whose values may change from line to
line. Specifically, the values of $C,C'$ may increase and the values
of $c,c'$ may decrease from line to line.

\section{Preliminaries}
\label{sec:preliminaries}

We gather here a number of general tools which we require for our
proof.

\begin{lemma}
\label{lem:main-tool}
Let $E$ and $F$ be events in a discrete probability space and let $T\colon E \to \cP(F)$ be a mapping, where $\cP(A)$ denotes the collection of all subsets of $A$. For $f \in F$, define
\[ N(f) := \left\{ e \in E ~|~ f \in T(e) \right\} .\]
If for some $p,q > 0$, we have
\begin{equation}
\label{eq:main-tool-assumptions}
\begin{aligned}
\Pr(T(e)) \geq \Pr(e) \cdot p , &\quad e \in E, \\ |N(f)| \leq q, &\quad f \in F,
\end{aligned}
\end{equation}
then
\[ \Pr(E) \leq \Pr(F) \cdot \frac{q}{p} .\]
\end{lemma}
\begin{proof}
It is a simple matter to verify that
\[ \sum_{e \in E} \Pr(T(e))
 = \sum_{e \in E} \sum_{f \in F} \Pr(f) \mathbf{1}_{T(e)}(f)
 = \sum_{f \in F} \sum_{e \in E} \Pr(f) \mathbf{1}_{T(e)}(f)
 = \sum_{f \in F} \Pr(f) \cdot |N(f)| .\]
The result now follows by the assumptions in \eqref{eq:main-tool-assumptions}.
\end{proof}

\begin{remark}
The opposite inequalities in \eqref{eq:main-tool-assumptions} would yield the analogous result.
Namely, if $\Pr(T(e)) \leq p\Pr(e)$ for all $e \in E$ and $|N(f)| \geq q$ for all $f \in F$, then $p\Pr(E) \geq q\Pr(F)$.
Note that when applying this lemma for the uniform distribution, the assumptions in \eqref{eq:main-tool-assumptions} become $|T(e)| \geq p$ for all $e \in E$ and $|N(f)| \leq q$ for all $f \in F$, while the conclusion remains the same.
\end{remark}

\begin{lemma}
\label{lem:ratio-prob-tool}
Let $X$ be a non-negative, integer-valued random variable.
Assume that, for some positive integer $n$ and some $a>1$, we have $\Pr(X = k) \geq a \cdot \Pr(X = k-1)$, for all $1 \leq k \leq n$. Then
\[ \Pr(X < n) \leq 1/a .\]
\end{lemma}
\begin{proof}
It is easy to verify (by induction) that
\[ \Pr(X=n) \geq a^k \cdot \Pr(X=n-k) , \quad 1 \leq k \leq n . \]
Thus,
\[ \Pr(X < n) = \sum_{k=0}^{n-1} \Pr(X=k) = \sum_{k=1}^{n} \Pr(X=n-k) \leq \Pr(X=n) \cdot \sum_{k=1}^{n} a^{-k} \leq \frac{\Pr(X=n)}{a-1} .\]
Therefore,
\[ 1 \geq \Pr(X \leq n) = \Pr(X < n) + \Pr(X = n) \geq \Pr(X < n) (1 + (a-1)) = a \cdot \Pr(X < n) . \qedhere \]
\end{proof}

We will use a theorem by Benjamini, H\"{a}ggstr\"{o}m and Mossel
\cite{Benjamini2000} to transfer results from the line to the torus.
This is an FKG inequality for the measure induced on non-negative
homomorphisms by taking pointwise absolute value.

Given a set $V$, we equip $\Z^V$ with the usual pointwise partial order $\preceq$. A function $\phi \colon \Z^V \to \R$ is said to be {\em increasing} if $\phi(f) \leq \phi(g)$ whenever $f \preceq g$.

\begin{thm}[FKG inequality for absolute values~{\cite[Proposition~2.3]{Benjamini2000}}]
\label{thm:fkg}
Let $G$ be a finite, bipartite and connected graph, let $v_0 \in V(G)$ and let $f$ be a uniformly chosen homomorphism in $\Hom(G,v_0)$. Then, for any two increasing functions $\phi,\psi \colon \Hom(G,v_0) \to \R$, we have
\[ \E\big[\phi(|f|) \cdot \psi(|f|)\big] \geq \E\big[\phi(|f|)\big] \cdot \E\big[\psi(|f|)\big] \]
where $|f|$ is the non-negative homomorphism obtained from $f$ by taking pointwise absolute value.
\end{thm}

Consider the event $Q$ that a homomorphism on $P_{n,d}$ is in fact a
valid homomorphism on $T_{n,d}$ (by identifying the vertex $n$ with
the vertex $0$). If we could write $\mathbf{1}_Q(f) = \psi(|f|)$ for
some function $\psi$ then we may be able to use the above theorem to
transfer results from the line to the torus by conditioning on $Q$.
However, it is not the case that $\mathbf{1}_Q$ is a function of the
absolute value of the homomorphism, and so we cannot apply Theorem
\ref{thm:fkg} directly. Instead, we make use of Theorem
\ref{thm:fkg} in order to prove a similar proposition specialized
for our purposes. See Proposition \ref{prop:fkg-line-torus} in
Section \ref{sec:torus} for more details.

The following result of Erd{\H o}s is useful for analyzing
homomorphism on the line.
\begin{thm}[{\cite[Theorem~1]{Erdos1945}}]
\label{thm:random-sign-walk}
Let $n$ be a positive integer, let $a_1,\dots,a_n \in \R$ satisfy $|a_i| \geq 1$ for $1 \leq i \leq n$ and let $\epsilon_1,\dots,\epsilon_n \sim U(\{-1,1\})$ be random independent signs. Denote
\[ S := \epsilon_1 a_1 + \cdots + \epsilon_n a_n .\]
Then, for any integer $r>0$ and any $a \in \R$, we have
\[ \Pr(|S - a| < r) \leq r \cdot \binom{n}{\lfloor n/2 \rfloor} \cdot 2^{-n} \leq \frac{C r}{\sqrt{n}} .\]
\end{thm}

The next proposition, which is a consequence of the previous result,
is useful for analyzing homomorphisms on the torus.
\begin{prop}
\label{prop:sampling-without-replacement2}
Let $n$ be a positive integer, let $a_1,\dots,a_n \in \R$ satisfy $|a_i| \geq 1$ for $1 \leq i \leq n$ and let $\pi$ be a uniformly chosen permutation of $\{1,2,\dots,n\}$. Denote
\[ S_i := a_{\pi(1)} + \cdots + a_{\pi(i)}, \quad 1 \leq i \leq n . \]
Then, for any integer $r>0$, we have
\[ \Pr\left(\max_{1 \leq i \leq n} |S_i| < r\right) \leq \frac{C r}{\sqrt{n}} .\]
\end{prop}
\begin{proof}
Let $\epsilon_1,\dots,\epsilon_n \sim U(\{-1,1\})$ be uniform independent signs. Denote
\begin{align*}
a &:= a_1 + \cdots + a_n , \\
S' &:= \sum_{i=1}^{n} \frac{\epsilon_i + 1}{2} a_i = \frac{\epsilon_1 a_1 + \cdots + \epsilon_n a_n}{2} + \frac{a}{2} .
\end{align*}
Let $T \sim \text{Bin}(n,1/2)$ be independent of $\pi$ and observe that
\[ S_T \eqd S' ,\]
an observation which was pointed out to us by Gady Kozma. Therefore, by Theorem \ref{thm:random-sign-walk},
\[ \Pr\left(\max_{1 \leq i \leq n} |S_i| < r\right) \leq \Pr\left(|S_T| < r\right) = \Pr\left(|S'| < r\right) = \Pr\left(\left|\sum_{i=1}^n \epsilon_i a_i + a\right| < 2r\right) \leq \frac{C r}{\sqrt{n}} . \qedhere \]
\end{proof}

The next lemma presents a simple result on limits of distributions.

\begin{lemma}
\label{lem:distribution-ratio-convergence}
Let $X_{\infty},X_1,X_2,\dots$ be non-negative, integer-valued random variables. Assume that $\Pr(X_{\infty}=k)>0$ for all integers $k \geq 0$. If the family $\{X_1,X_2,\dots\}$ is tight, and
\begin{equation}
\label{eq:distribution-ratio-convergence}
\lim_{n \to \infty} \frac{\Pr(X_n=k)}{\Pr(X_n = k -1)} = \frac{\Pr(X_{\infty}=k)}{\Pr(X_{\infty}=k -1)} , \quad k \geq 1,
\end{equation}
then
\[ X_n \xrightarrow[n \to \infty]{(d)} X_{\infty} .\]
\end{lemma}
\begin{proof}
Let $\epsilon > 0$ and, using the tightness assumption, choose an
integer $M$ such that $\Pr(X_n
> M) \leq \epsilon$ for all $n \in \N \cup \{\infty\}$. Then
\[ 1- \epsilon \leq \sum_{k=0}^M \Pr(X_n=k) \leq 1 , \quad n \in \N \cup \{\infty\} .\]
Therefore, by the assumption \eqref{eq:distribution-ratio-convergence},
\[ \limsup_{n \to \infty} \frac{1- \epsilon}{\Pr(X_n=0)} \leq \lim_{n \to \infty} \sum_{k=0}^M \frac{\Pr(X_n=k)}{\Pr(X_n=0)} = \sum_{k=0}^M \frac{\Pr(X_{\infty}=k)}{\Pr(X_{\infty}=0)} \leq \frac{1}{\Pr(X_{\infty}=0)} .\]
Similarly, we have
\[ \liminf_{n \to \infty} \frac{1}{\Pr(X_n=0)} \geq \lim_{n \to \infty} \sum_{k=0}^M \frac{\Pr(X_n=k)}{\Pr(X_n=0)} = \sum_{k=0}^M \frac{\Pr(X_{\infty}=k)}{\Pr(X_{\infty}=0)} \geq \frac{1 - \epsilon}{\Pr(X_{\infty}=0)} .\]
Since $\epsilon$ is arbitrary, we conclude that $\Pr(X_n=0) \to \Pr(X_{\infty}=0)$ as $n \to \infty$, which in turn gives that $\Pr(X_n=k) \to \Pr(X_{\infty}=k)$ as $n \to \infty$ for all $k$, by \eqref{eq:distribution-ratio-convergence}.
\end{proof}

\section{Homomorphisms on the Line}
\label{sec:line}

In this section we will prove the theorems regarding homomorphisms
on the line which were stated in Section
\ref{sec:main-results-line}. As was pointed out in the introduction,
it seems unlikely that a homomorphism jumps from some value $t$ to
$t \pm 3$ on any given small segment. Figure
\ref{fig:single-jump-segment} illustrates a section of a
homomorphism for which such a jump occurs. The main idea in our
proofs is to identify the vertices at which these jumps occur, as
they determine the large scale behavior of the homomorphism. That
is, the values of the homomorphism at the jumps contain the global
information necessary to determine the range and the variance. To
this end, we first define the notion of the (local) average height
of a homomorphism at a vertex (this is illustrated by the horizontal
dashed line in Figure \ref{fig:rough-movements}). The average height
at a vertex is determined by finding the closest past time at which
$3$ different values appeared consecutively and taking the midpoint
to be the average height. For vertices which no such time exists (as
is the case for the $0$ vertex), we set the average height to $0$.
One can think of the average height as a process beginning at $0$
that ``lazily follows'' the homomorphism, only to ensure that it is
never at a distance greater than $1$. With this notion in hand, we
define a jump as a change in the average height. Of course any such
jump has an associated sign or direction, which is determined by
whether the average height increased or decreased.

We later show that the probability of a jump occurring at a given
vertex (greater than $2d$) is no more than $2^{-d}$. This will show
that in the supercritical regime, with high probability, there will
not be any jumps (after vertex $2d$). That is, the average height
does not change after time $2d$. A moment's reflection reveals that
this means that the homomorphism takes on at most $3$ different
values (not $4$, as it may initially seem).

We do not give a lower bound for the probability of a jump occurring
at a given vertex. Instead, we only show that the typical number of
jumps is of order $n2^{-d}$, the jumps are approximately
equidistributed on the line and that, moreover, the directions of
these jumps are weakly correlated. Of course, if the directions of
these jumps were truly independent, then the values of the
homomorphism at the jumps would constitute a simple random walk. We
will show that, at least in terms of the maximum/range of the
homomorphism, the behavior is very similar to that of a simple
random walk. This will show that the range is of order
$\sqrt{n2^{-d}}$ and that the variance at a vertex $k$ is of order
$k2^{-d}$.

In the analysis of these so-called jumps, we encounter a minor
complication due to the fact that jumps in the same direction can
``clump'' together. Of course jumps cannot occur consecutively in
the sense of two consecutive vertices on the line. So then what is
the minimal distance between two jumps? The answer is twofold. The
minimal distance between two jumps with different directions is
$2d+3$, while two jumps in the same direction can already occur at
distance $2d+1$. This phenomenon will pop up again and again in our
analysis. For example, its manifestation is evident in the Markov
chain describing the local limit in Section~\ref{sec:local-limits}
(see Figure \ref{fig:local-limit-markov-chain}).

One meaning of this phenomenon is that if we condition on the event that a jump occurs at two given vertices, say $k_1$ and $k_2$, $k_1<k_2$, the directions of these jumps are non-negatively correlated. However, conditioning also on the event that a jump does not occur just after the first of these jumps (i.e. at $k_1+2d+1$), their directions become independent. This leads us to consider ``chains'' of jumps. A chain is just a sequence of minimal-distance same-direction jumps. Now, if we condition on the event that there are chains of given lengths (and not longer) at any number of given vertices, the directions of all these chains will be independent. This will allow us to reduce some of the analysis to a case of independent variables.

\subsection{Definitions}

We consider the graph $P_{n,d}$ whose vertex set is $\{0,1,...,n\}$ and whose edges are $(k,m)$ for $|k - m| = 1,3,...,2d+1$. Throughout this section, $\Hom(P_{n,d}):=\Hom(P_{n,d},0)$, $f$ is a uniformly sampled homomorphism from $\Hom(P_{n,d})$, the probability space is the uniform distribution on the set $\Hom(P_{n,d})$, and events are subsets of $\Hom(P_{n,d})$.

We define $h(k)$, the {\em (local) average height} at vertex $k$, inductively as follows. Set $h(0):=0$. For $1 \leq k \leq n$, define
\begin{align*}
h(k) &:= \begin{cases}
    h(k-1)  & \text{if } | f(k) - h(k-1) | \leq 1  \\
    f(k-1)  & \text{otherwise}
\end{cases} , \\
\Delta(k) &:= h(k) - h(k-1) .
\end{align*}
Define the event
\[ A_k := \{ \Delta(k) \neq 0 \} .\]
When $A_k$ occurs, we say that a {\em jump} occurred at vertex $k$ (see Figure \ref{fig:rough-movements}). Let
\begin{equation}\label{eq:S_line_def}
 S := \left\{ 1 \leq k \leq n ~ | ~  \Delta(k) \neq 0 \right\}
\end{equation}
be the positions of the jumps, and denote by
\begin{equation}\label{eq:R_line_def}
 R := |S \setminus\{1,\dots,2d+1\}| = \sum_{k=2d+2}^n \mathbf{1}_{A_k}
\end{equation}
the number of jumps after vertex $2d+1$. Recall that if a jump
occurs at vertex $k$, then the minimal possible value of $k'>k$ at
which another jump can occur is $k+2d+1$. Let $C_{k,t}$ be the event
that there is a chain of $t$ minimal-distance jumps ending at vertex
$k$. That is, for $t \geq 1$ and $(t-1)(2d+1) < k \leq n$, we define
\[ C_{k,t} := A_k \cap A_{k-2d-1} \cap \dots \cap A_{k-(t-1)(2d+1)} .\]

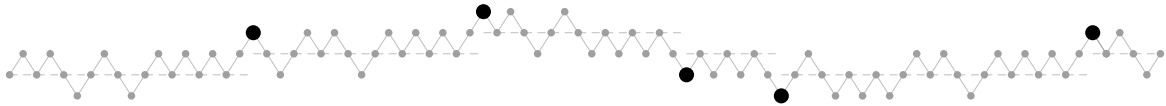
\begin{figure}[!t]
\centering
\vspace{5pt}
\begin{tikzpicture}

    \pgftransformcm{0.18}{0}{0}{0.28}{\pgfpoint{0cm}{0cm}}
    \tikzstyle bigNode=[circle,fill=black,minimum size=0.2cm,inner sep=0]
    \tikzstyle{every node}=[minimum size=0.1cm,inner sep=0]

    \drawRandomWalk{0}{{0,1,0,1,0,-1,0,1,0,-1,0,1,0,1,0,1,0,1,2}}{0}{gray!75}{gray!75}{gray!75};

    \drawRandomWalk{18}{{2,1,0,1,2,1,2,1,0,1,2,1,2,1,2,1,2,3}}{1}{gray!75}{gray!75}{gray!75};
    \node at (18,2) [bigNode] {};

    \drawRandomWalk{35}{{3,2,3,2,1,2,3,2,1,2,1,2,1,2,1,0}}{2}{gray!75}{gray!75}{gray!75};
    \node at (35,3) [bigNode] {};

    \drawRandomWalk{50}{{0,1,0,1,0,1,0,-1}}{1}{gray!75}{gray!75}{gray!75};
    \node at (50,0) [bigNode] {};

    \drawRandomWalk{57}{{-1,0,1,0,-1,0,-1,0,-1,0,1,0,1,0,-1,0,1,0,1,0,1,0,1,2}}{0}{gray!75}{gray!75}{gray!75};
    \drawRandomWalk{80}{{2,1}}{99}{gray!75}{gray!75}{gray!75};
    \node at (57,-1) [bigNode] {};

    \drawRandomWalk{80}{{2,1,2,1,0,1}}{1}{gray!75}{gray!75}{gray!75};
    \node at (80,2) [bigNode] {};

\end{tikzpicture}
\caption{A homomorphism in $\Hom(P_{n,d})$. The big vertices denote the positions of the jumps. The dashed horizontal lines denote the average height. Here $d=3$.}
\label{fig:rough-movements}
\vspace{5pt}
\end{figure}

Let $I = \{s_1,\dots,s_t\} \subset \{1,2,\dots,n\}$. We say that $I$ is a {\em feasible jump structure} if $\{S = I\} \neq \emptyset$. Observe that $\{S=I\} \neq \emptyset$ if and only if $\Pr(S=I)>0$ if and only if when we reorder the $s_i$ to satisfy $s_1 < s_2 < \cdots < s_t$, we have
\begin{equation}
\label{eq:line-feasible-jump-structure}
s_1 \text{ is even and for } 2 \leq j \leq t,~s_j - s_{j-1} \text{ is odd and satisfies } s_j - s_{j-1} \geq 2d+1 .
\end{equation}
In addition, we say that a subset $I \subset \{1,\dots,n\}$ is a {\em feasible jump sub-structure} if it is a subset of a feasible jump structure, or equivalently, if $\{I \subset S\} \neq \emptyset$. For a feasible jump sub-structure $I$, the event $\{I \subset S\}$ can be uniquely written as $C_{k_1,t_1} \cap \dots \cap C_{k_m,t_m}$, where
\begin{equation}
\label{eq:line-chain-structure}
\begin{aligned}
&t_1+\cdots+t_m = |I| , \\
&k_j - k_{j-1} > (2d+1)t_j , \quad 2 \leq j \leq m .
\end{aligned}
\end{equation}
These conditions ensure that there is no overlap between the different chains, and moreover, that there is some gap between them (since otherwise they would merge into a larger chain). For such $I$, we define
\[ \cC(I) := \{ (k_j,t_j) ~|~ 1 \leq j \leq m \} ,\]
and refer to this as the {\em chain structure} of $I$.

\subsection{Main lemmas}

As the above definitions suggest, the notion of a jump at a given vertex plays an important role in our analysis. It turns out that the behavior of jumps at the first $2d+1$ vertices differs significantly from that of the other vertices. Hence, it will be a recurring theme throughout Section \ref{sec:line} that these cases are handled separately.

The first two lemmas concern the probability of jumps at given vertices. The first of which shows that jumps at the first $2d+1$ vertices are not unlikely, while the second shows that elsewhere jumps are unlikely.

\begin{lemma}
\label{lem:line-jump-at-start}
We have
\[ 1/4 \leq \Pr(A_2) \leq 1/2 \]
and
\[ 1/3 \leq \Pr(A_1 \cup \cdots \cup A_{2d+1}) \leq 2/3 .\]
\end{lemma}
\begin{proof}
Denote $J := A_1 \cup \cdots \cup A_{2d+1}$. We shall show that
\[ \Pr(A_2) \leq \Pr(J) \leq 2 \Pr(A_2) ,\]
\[ \Pr(A_2) \leq \Pr(J^c) \leq 2 \Pr(A_2) ,\]
from which the result easily follows. Note that, by \eqref{eq:line-feasible-jump-structure}, $A_k = \emptyset$ for $k=1,3,\dots,2d+1$, so that $J = A_2 \cup A_4 \cup \cdots \cup A_{2d}$. Clearly $\Pr(A_2) \leq \Pr(J)$, as $A_2 \subset J$.

We note the following useful observation. For a homomorphism $f \in \Hom(P_{n,d})$, we have
\begin{equation}
\label{eq:line-A2-condition}
f \in A_2 \iff f(2) = 2 f(1) \iff f(2) \neq 0 .
\end{equation}

We begin by showing that $\Pr(J \setminus A_2) \leq \Pr(A_2)$, from which it follows that $\Pr(J) \leq 2 \Pr(A_2)$. To this end it suffices to show an injective mapping from $J \setminus A_2$ to $A_2$. Consider the mapping $f \mapsto f_1$ from $J \setminus A_2$ to $A_2$ defined by
\[ f_1(k) :=
    \begin{cases}
        f(k) &\text{if } k \neq 2 \\
        2 f(1) &\text{if } k = 2
    \end{cases} , \quad 0 \leq k \leq n .
\]
One may check that if $f \in J \setminus A_2$ then $f_1 \in A_2$. Recalling \eqref{eq:line-A2-condition}, it is clear that this mapping is invertible, and so we have $\Pr(J \setminus A_2) \leq \Pr(A_2)$.

We now show that $\Pr(A_2) \leq \Pr(J^c)$. Define a mapping $f \mapsto f_2$ from $A_2$ to $J^c$ by
\[ f_2(k) :=
    \begin{cases}
        f(k+1) - f(1) &\text{if } 0 \leq k < n \\
        f(n-1) - f(1) &\text{if } k = n
    \end{cases} , \quad 0 \leq k \leq n .
\]
Again one may check that this mapping is well-defined (in fact, this mapping can be defined on the entire space). Since it is injective (recall \eqref{eq:line-A2-condition}), we obtain $\Pr(A_2) \leq \Pr(J^c)$.

Finally, we show that $\Pr(J^c) \leq 2 \Pr(A_2)$. Consider the mapping $T \colon J^c \to A_2$ defined by
\[ T(f)(k) :=
    \begin{cases}
        0 &\text{if } k = 0 \\
        f(k-1) + f(1) &\text{if } 1 \leq k \leq n
    \end{cases} , \quad 0 \leq k \leq n .
\]
To see that this mapping is well-defined, recall \eqref{eq:line-A2-condition}, and note that $f \in J^c$ implies that $f(0)=f(2)=\cdots=f(2d)=0$. This is not an injective mapping, however, it satisfies $|T^{-1}(g)| \leq 2$ for $g \in A_2$. Therefore, by Lemma \ref{lem:main-tool}, we have $\Pr(J^c) \leq 2 \Pr(A_2)$.
\end{proof}

The next lemma is concerned with the probability of jumps occurring
at given vertices after $2d+1$. It states that this probability is
exponentially small in $d$ times the number of jumps. The idea
behind the proof is to remove the jumps and replace the freed up
areas with segments of constant average height. This allows us to
gain entropy by setting the values at every other vertex in each
such segment to be the average height $\pm 1$. See Figure
\ref{fig:line-chain-removal}.

\begin{lemma}
\label{lem:line-jumps-at-positions}
For any $t \geq 1$ and for any $2d+1 < s_1 < \cdots < s_t \leq n$, we have
\[ \Pr(A_{s_1} \cap \dots \cap A_{s_t}) \leq 2^{-dt} .\]
\end{lemma}

\begin{proof}
If $I:=\{s_1,\dots,s_t\}$ is not a feasible jump sub-structure then there is nothing to prove. Otherwise, we consider the chain structure of $I$, $\cC(I) = \{ (k_1,t_1), \dots, (k_m,t_m) \}$, where we have ordered the elements so that the $k_j$ are increasing. Due to our assumption that $s_1>2d+1$, we have $k_1 > (2d+1)t_1$. We note that it is enough to prove that for all $1 \leq j \leq m$,
\[ \Pr(C_{k_j,t_j} ~|~ C_{k_1,t_1} \cap \dots \cap C_{k_{j-1},t_{j-1}}) \leq 2^{-d t_j} .\]
We prove something stronger. Let $1 \leq k \leq n$ and $t \geq 1$ be such that $k':=k - (2d+1)t - 1 \geq 0$. Then, for any $\xi \in \mathbb{Z}^{\{0,\dots,k'\}}$ such that $\Pr(f |_{\{0,\dots,k'\}} = \xi) > 0$, we have
\begin{equation}
\label{eq:line-chain-prob-given-initial-values}
\Pr\left(C_{k,t} ~|~ f |_{\{0,\dots,k'\}} = \xi \right) \leq 2^{-dt - \lfloor t/2 \rfloor} .
\end{equation}
In order to show this, we construct a mapping which removes this chain and replaces the freed up segment with a segment of constant average height (see Figure \ref{fig:line-chain-removal}). Formally, we proceed as follows. For $m \geq 1$ and $w \in \{-1,1\}^m$, denote
\[ \tilde{w} := (w_1, 0, w_2, 0, \dots, w_m, 0) . \]
Define a mapping
\[ T_{k,t} \colon C_{k,t} \times \{-1,1\}^{dt+\lfloor t/2 \rfloor} \to \Hom(P_{n,d}) \]
by
\begin{equation}
\label{eq:def-T-k-t}
T_{k,t}(f,w)(i) := \begin{cases}
    f(i)                        & \text{if } i \leq k' \\
    f(k') + \tilde{w}(i-k')     & \text{if } k' < i < k-1 \\
    f(i+\delta)-f(k-1)+f(k')    & \text{if } k-1 \leq i < n \\
    f(i-\delta)-f(k-1)+f(k')    & \text{if } i = n
\end{cases} ,
\end{equation}
where $\delta = 0$ if $t$ is even and $\delta = 1$ if $t$ is odd.

\begin{figure}[!t]
\centering
\begin{tabular}{ll}
\raisebox{6ex}{$f \in C_{k,t}$} &
\begin{tikzpicture}
    \pgftransformcm{0.45}{0}{0}{0.55}{\pgfpoint{0cm}{0cm}}

    \draw[line width=0.1mm,gray,dotted](5,-2)--(5,4);
    \draw[line width=0.1mm,gray,dotted](21,-2)--(21,4);

    \drawRandomWalk{1}{{0,-1,0}}{99}{gray!25}{gray!50}{gray!37};
    \drawRandomWalk{3}{{0,1,0}}{0}{gray!50}{orange}{gray!75};
    \drawRandomWalk{5}{{0,1,0,1,0,1,2}}{0}{orange}{}{};
    \drawRandomWalk{11}{{2,1,2,1,2,3}}{1}{}{}{};
    \drawRandomWalk{16}{{3,2,3,2,3,4}}{2}{}{red}{};
    \drawRandomWalk{21}{{4,3}}{3}{red}{gray!75}{};
    \drawRandomWalk{22}{{3,4,3,2}}{3}{gray!50}{gray!25}{gray!37};

\end{tikzpicture}
\\ [-0.5ex]
\raisebox{8ex}{$T_{k,t}(f,w)$} &
\begin{tikzpicture}
    \pgftransformcm{0.45}{0}{0}{0.55}{\pgfpoint{0cm}{0cm}}

    \draw[line width=0.1mm,gray,dotted](5,-1.7)--(5,2);
    \draw[line width=0.1mm,gray,dotted](21,-1.7)--(21,2);
    \node at (5,-1.7) [below] {$k'$};
    \node at (21,-1.7) [below] {$k$};

    \drawRandomWalk{1}{{0,-1,0}}{99}{gray!25}{gray!50}{gray!37};
    \drawRandomWalk{3}{{0,1,0}}{0}{gray!50}{orange}{gray!75};
    \drawRandomWalk{5}{{0,-1,0,1,0,1,0,-1,0,-1,0,1,0,-1,0,1}}{0}{orange}{red}{};
    \drawRandomWalk{20}{{1,0}}{0}{red}{gray!75}{};
    \drawRandomWalk{21}{{0,1,0,-1}}{0}{gray!50}{gray!25}{gray!37};
    \node at (25,0) [circle, fill=none] {}; 

\end{tikzpicture}
\end{tabular}
\caption{A section of a homomorphism $f$ in the event $C_{k,t}$. Removing the chain between $k'$ and $k$, and replacing it with fluctuations around the average height, gives the homomorphism $T_{k,t}(f,w)$. The dashed horizontal lines denote the average height. Here $d=2$, $t=3$ and $w=(-1,1,1,-1,-1,1,-1)$.}
\label{fig:line-chain-removal}
\end{figure}
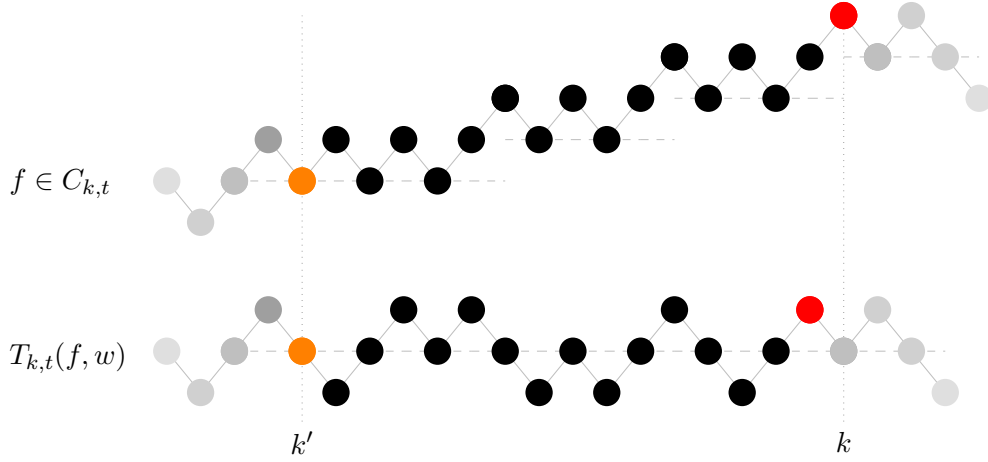

We now show that $T_{k,t}$ is well-defined, i.e. that $T_{k,t}(f,w) \in \Hom(P_{n,d})$.
For $0 \leq i,j \leq n$, denote
\[ \Delta_{i,j} := |T_{k,t}(f,w)(i)-T_{k,t}(f,w)(j)| .\]
For $0 \leq j \leq n$, define the event
\[ B_j := \big\{ f(i)=f(j) \text{ when } 0 \leq i \leq n \text{ satisfies } |i-j| \in \{2,4,\dots,2d\} \big\} . \]
We must show that $\Delta_{i,j}=1$ whenever $|i-j| \in \{1,3,\dots,2d+1\}$.
Let $0 \leq i,j \leq n$ satisfy $|i-j| \in \{1,3,\dots,2d+1\}$ and assume without loss of generality that $i<j$.
We shall further assume that $j<n$, the case $j=n$ being similar.
If $j \leq k'$, $i \geq k-1$ or $k'<i<j<k-1$ then $\Delta_{i,j}=1$ follows immediately from \eqref{eq:def-T-k-t} and the fact that $f \in \Hom(P_{n,d})$.
It remains to check the case when $i \leq k'$ and $j>k'$ and the case when $i<k-1$ and $j \geq k-1$.

We begin with the first case. Here, we have $\Delta_{i,j}=|f(i)-f(k')-\tilde{w}(j-k')|$.
Observe that $C_{k,t} \subset B_{k'}$.
Therefore, if $i$ has the same parity as $k'$ then $f(i)=f(k')$ and $|\tilde{w}(j-k')|=1$ since $j$ has the same parity as $k'$.
Otherwise, $i$ has the opposite parity of $k'$, and then $|f(i)-f(k')|=1$ and $\tilde{w}(j-k')=0$.
Thus, $\Delta_{i,j}=1$.

In the second case, we have $\Delta_{i,j}=|f(j+\delta)-f(k-1)-\tilde{w}(i-k')|$.
Note that $C_{k,t} \subset B_{k-1}$ and that $i-k'$ has the same parity as $j+\delta-k$.
One finds in a similar manner as in the first case that $|f(j+\delta)-f(k-1)|=1$ and $\tilde{w}(i-k')=0$ when $i$ has the same parity as $k'$, and that $f(j+\delta)=f(k-1)$ and $|\tilde{w}(i-k')|=1$ when $i$ has the opposite parity of $k'$. Hence, $\Delta_{i,j}=1$.

Observe that for any $f \in C_{k,t}$, necessarily,
\begin{align*}
&(f(k'),f(k'+1), \dots, f(k), f(k+1)) \\ & \qquad = (f(k'),\dots,f(k')) \pm (\underbrace{0,1,\dots,0,1,0}_{2d+1}, \underbrace{1,2,\dots,1,2,1}_{2d+1}, \dots, \underbrace{t-1,t,\dots,t-1,t,t-1}_{2d+1}, t,t+1,t) .
\end{align*}
Thus, it is easy to see that the mapping is injective. Moreover, the event $\{ f |_{\{0,\dots,k'\}} = \xi \}$ is clearly invariant under this mapping, so that
\[ \Pr\left(C_{k,t} \cap \{ f |_{\{0,\dots,k'\}} = \xi \}\right) \cdot 2^{dt+\lfloor t/2 \rfloor} \leq \Pr\left(f |_{\{0,\dots,k'\}} = \xi \right) , \]
proving \eqref{eq:line-chain-prob-given-initial-values}.
\end{proof}

\begin{remark}
The proof shows in fact that the probability of the event $A_{s_1} \cap \dots \cap A_{s_t}$ is bounded by $2^{-dt - (\lfloor t_1/2 \rfloor + \cdots + \lfloor t_r/2 \rfloor)}$, where $t_1,\dots,t_r$ are the lengths of the chains corresponding to $s_1,\dots,s_t$. With a small modification, the proof can be enhanced to give the bound $2^{-dt - \lfloor t/2 \rfloor}$, but we neither prove nor use this.
\end{remark}

Recall the definition of $R$ from \eqref{eq:R_line_def}. We would
like to obtain inequalities on the probability that $R$ is a given
value. We could do this in a similar manner to which the previous
lemma was proved. However, for variety, we prefer to employ a more
direct combinatorial technique. This approach also has the advantage
of introducing Lemma \ref{lem:line-bijection}, which gives a useful
description of the structure of the homomorphisms in
$\Hom(P_{n,d})$.

We decompose a homomorphism into two parts (see Figure
\ref{fig:line-bijection}). The first part constitutes the changes in
average height (the underlying walk) of the homomorphism, while the
second part constitutes the fluctuations around the average height
(the segments of constant average height). For a feasible jump
sub-structure $I$, define the {\em chain points} of $I$ by
\[ CP(I) := \bigcup_{(k,t) \in \cC(I) }\{ k - (2d+1)t - 1, \dots, k-1, k \} ,\]
and the {\em fluctuation points} of $I$ by
\[ FP(I) := \Big\{ 1 \leq k \leq n ~|~ \min\big\{ i > 0 ~|~ k - i \in CP(I) \cup \{-1\} \big\} \text{ is even} \Big\} .\]
That is, a point is a fluctuation point if its distance from the
chain to its left is positive and even. In particular, recalling the
definition of $S$ from \eqref{eq:S_line_def}, for any homomorphism
$f$ and any $k \in FP(S(f))$, $f$ is not at its average height at
$k$. Now, for a homomorphism $f$, define
\begin{align*}
X(f) &\in \{-1,1\}^{\cC(S(f))} \quad \text{and} \\
Y(f) &\in \{-1,1\}^{FP(S(f))}
\end{align*}
by
\begin{equation}
\begin{aligned}
X(f)_{(k,t)}    &:= f(k) - f(k-1) \quad \text{and} \\
Y(f)_{k}        &:= f(k) - f(k-1) .
\end{aligned}
\end{equation}

\begin{figure}[!t]
\centering
\vspace{5pt}

\newlength{\Dlen}
\newlength{\Dspc}
\Dlen=0.36cm
\Dspc=0.2cm

\newsavebox{\myLine}
\savebox{\myLine}{\tikz{\draw[line width=0.1mm,gray,dotted](0,0)--(0,1);}}

\newcommand{\sepLine}{\makebox[0cm]{\raisebox{-3.5ex}{\smash{{\usebox{\myLine}}}}}}

\newsavebox{\flucUp}
\newsavebox{\flucDown}
\newsavebox{\flucSpace}
\newsavebox{\flucSpaceH}
\savebox{\flucUp}{\makebox[\Dlen]{\small $+$}}
\savebox{\flucDown}{\makebox[\Dlen]{\small $-$}}
\savebox{\flucSpace}{\hspace{\Dlen}}
\savebox{\flucSpaceH}{\hspace{\Dspc}}

\[
\begin{array}{rl}
\raisebox{2.2ex}{$f$}
&
\begin{tikzpicture}

    \pgftransformcm{0.18}{0}{0}{0.28}{\pgfpoint{0cm}{0cm}}
    \tikzstyle bigNode=[draw=none,fill=none,minimum size=0.07cm,inner sep=0]
    \tikzstyle{every node}=[minimum size=0.07cm,inner sep=0]
    \tikzstyle randomWalkAxisStyle=[red,thin,opacity=0.6]

    \draw[line width=0.1mm,gray,dotted](10,-2)--(10,5);
    \draw[line width=0.1mm,gray,dotted](17,-2)--(17,5);
    \draw[line width=0.1mm,gray,dotted](25,-2)--(25,5);
    \draw[line width=0.1mm,gray,dotted](37,-2)--(37,5);
    \draw[line width=0.1mm,gray,dotted](41,-2)--(41,5);
    \draw[line width=0.1mm,gray,dotted](58,-2)--(58,5);
    \draw[line width=0.1mm,gray,dotted](66,-2)--(66,5);
    \draw[line width=0.1mm,gray,dotted](73,-2)--(73,5);

    \drawRandomWalk{0}{{0,1,0,-1,0,-1,0,1,0,1,0}}{0}{gray!75}{blue}{gray!75};
    \tikzstyle randomWalkPathStyle=[blue,thin,opacity=0.5]
    \drawRandomWalk{10}{{0,1,0,1,0,1,2,1}}{99}{blue}{blue}{blue};

    \tikzstyle randomWalkPathStyle=[gray,thin,opacity=0.5]
    \drawRandomWalk{17}{{1,0,1,2,1,2,1,0,1}}{1}{gray!75}{gray!75}{gray!75};
    \tikzstyle randomWalkPathStyle=[blue,thin,opacity=0.5]
    \drawRandomWalk{25}{{1,2,1,2,1,2,3,2,3,2,3,4,3}}{99}{blue}{blue}{blue};

    \tikzstyle randomWalkPathStyle=[gray,thin,opacity=0.5]
    \drawRandomWalk{37}{{3,4,3,2,3}}{3}{gray!75}{gray!75}{gray!75};
    \tikzstyle randomWalkPathStyle=[blue,thin,opacity=0.5]
    \drawRandomWalk{41}{{3,2,3,2,3,2,1,2,1,2,1,0,1,0,1,0,-1,0}}{99}{blue}{blue}{blue};

    \tikzstyle randomWalkPathStyle=[gray,thin,opacity=0.5]
    \drawRandomWalk{58}{{0,1,0,1,0,1,0,-1,0}}{0}{gray!75}{gray!75}{gray!75};
    \tikzstyle randomWalkPathStyle=[blue,thin,opacity=0.5]
    \drawRandomWalk{66}{{0,1,0,1,0,1,2,1}}{99}{blue}{blue}{blue};

    \tikzstyle randomWalkPathStyle=[gray,thin,opacity=0.5]
    \drawRandomWalk{73}{{1,0,1,2,1,0,1}}{1}{gray!75}{gray!75}{gray!75};

\end{tikzpicture}
\\
\underset{\raisebox{-1.75ex}{$X(f)$}}{Y(f)}
&
\hspace{0.035cm}
\usebox{\flucUp}
\usebox{\flucDown}
\usebox{\flucDown}
\usebox{\flucUp}
\usebox{\flucUp}
\sepLine
\raisebox{-3.5ex}{\makebox[1.26cm]{\small $+$}}
\sepLine
\usebox{\flucDown}
\usebox{\flucUp}
\usebox{\flucUp}
\usebox{\flucDown}
\sepLine
\raisebox{-3.5ex}{\makebox[2.16cm]{\small $+$}}
\sepLine
\usebox{\flucUp}
\usebox{\flucDown}
\sepLine
\raisebox{-3.5ex}{\makebox[3.06cm]{\small $-$}}
\sepLine
\usebox{\flucUp}
\usebox{\flucUp}
\usebox{\flucUp}
\usebox{\flucDown}
\sepLine
\raisebox{-3.5ex}{\makebox[1.26cm]{\small $+$}}
\sepLine
\usebox{\flucDown}
\usebox{\flucUp}
\usebox{\flucDown}

\end{array}
\]
\caption{A homomorphism is decomposed into two parts; chains (in
blue) and fluctuations (in gray). The chains, which consist of
consecutive jumps, contribute to the change in average height, while
the fluctuations do not. From the chains, we construct $X(f)$, which
contains the direction of each chain. From the fluctuations, we
construct $Y(f)$, which contains the direction of each fluctuation.
Given the positions of the jumps, $X(f)$ and $Y(f)$ precisely
contain the remaining information on the homomorphism. That is, for
any choice of $X$ and $Y$, there exists a unique homomorphism $f$
with $X(f)=X$ and $Y(f)=Y$. Here $d=2$.} \label{fig:line-bijection}
\vspace{5pt}
\end{figure}
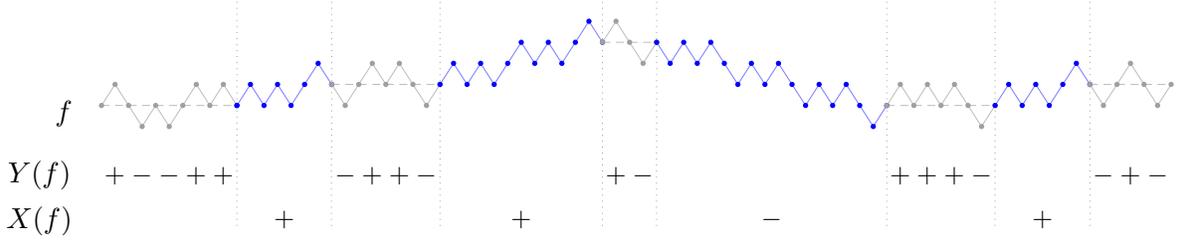

\begin{claim}
\label{cl:line-bijection-fluc-size}
For any feasible jump structure $I$, we have
\begin{equation}
\label{eq:line-bijection-fluc-size}
|FP(I)| = \max\left\{0,d+1- \frac{\min I}{2}\right\} + \left\lceil \frac{n-|I|}{2} \right\rceil - d|I| - |\cC(I)| .
\end{equation}
\end{claim}
\begin{proof}
Suppose that $\cC(I) = \{ (k_1,t_1), \dots, (k_m,t_m) \}$. Denote $t := |I| = t_1+\cdots+t_m$ and $s := \min I$. Then
\begin{align*}
|CP(I)| &= \sum_{j=1}^{m} \big( (2d+1)t_j + 2 \big) = (2d+1)t + 2m , \\
|CP(I) \cap \{0,1,\dots,n\}| &= (2d+1)t + 2m - \max\{0, 2d+2-s\} .
\end{align*}
Therefore, recalling that $s$ is even by
\eqref{eq:line-feasible-jump-structure},
\begin{align*}
|FP(I)|
  &= \left\lceil \frac{n - |CP(I) \cap \{0,1,\dots,n\}|}{2} \right\rceil \\
  &= \left\lceil \frac{n - (2d+1)t - 2m + \max\{0, 2d+2-s\}}{2} \right\rceil \\
  &= \left\lceil \frac{n - t}{2} \right\rceil - dt - m + \max\{0, d+1-s/2\} .
\qedhere
\end{align*}
\end{proof}

\begin{lemma}
\label{lem:line-bijection}
For any feasible jump structure $I$, the mapping $f \mapsto (X(f), Y(f))$ is a bijection between $\{S=I\}$ and $\{-1,1\}^{\cC(I)} \times \{-1,1\}^{FP(I)}$.
\end{lemma}
\begin{proof}

We shall describe the inverse mapping which maps a pair $(X,Y) \in \{-1,1\}^{\cC(I)} \times \{-1,1\}^{FP(I)}$ to the homomorphism $f_{X,Y} \in \{S=I\}$. For $0 \leq i \leq n$, let
\[ H(X,i) := \sum_{\substack{(k,t) \in \cC(I) \\ k<i}} t \cdot X_{(k,t)} ,\]
be the average height accumulated by chains ending before $i$. For
$(k,t) \in \cC(I)$, denote by $k'(k,t) := k-(2d+1)t-1$ the first
vertex of the chain and observe that $H(X,i) = H(X,k'(k,t))$ for all
$k'(k,t)\le i\le k$. Now, define
\[ f_{X,Y}(i) := \begin{cases}
    H(X,i) + Y_i                &\text{if } i \in FP(I) \\
    H(X,i)                  &\text{if } i+1 \in FP(I) \text{ or } i-1 \in FP(I) \\
    H(X,i) + X_{(k,t)} C_{i-k'(k,t)}    &\text{if } k'(k,t) \le i \leq k \text{ for some } (k,t) \in \cC(I)
\end{cases} , \]
where $C=(C_0,C_1,\dots)$ is the infinite sequence defined by
\[ C  := (\underbrace{0,1,\dots,0,1,0}_{2d+1}, \underbrace{1,2,\dots,1,2,1}_{2d+1}, \underbrace{2,3,\dots,2,3,2}_{2d+1}, \dots ) . \]
See Figure \ref{fig:line-bijection}. It is not difficult to check that $f_{X,Y} \in \Hom(P_{n,d})$ and that $S(f_{X,Y})=I$. It remains to check that $X(f_{X,Y})=X$, $Y(f_{X,Y})=Y$ and $f_{X(f),Y(f)}=f$. We omit the details.
\end{proof}

\begin{cor}
\label{cor:line-indep-chain-signs}
Conditioned on $S$, the following properties hold.
\begin{enumerate}
 \item $X$ is uniformly distributed over $\{-1,1\}^{\cC(S)}$.
 \item The random variables $\{ \Delta(s) \}_{s \in S}$ are independent uniform signs conditioned on $\Delta(s)=\Delta(s')$ whenever $s,s' \in S$ satisfy $|s-s'|=2d+1$.
 \item The difference in average height between two vertices $0 \leq k_0 < k_1 \leq n$ is a sum of independent variables, namely,
\begin{equation}
\label{eq:line-avg-height-diff}
h(k_1) - h(k_0) = \sum_{(j,t) \in \cC(S \cap \{k_0+1,\dots,k_1\})} t \cdot \Delta(j) .
\end{equation}
\end{enumerate}
\end{cor}
\begin{proof}
The first statement is an immediate consequence of Lemma \ref{lem:line-bijection}. The second statement is in turn a consequence of the first statement and of the definition of the chain structure $\cC(S)$. For the third statement, since
\[ h(k) = \sum_{i=1}^k \Delta(i) = \sum_{\substack{s \in S \\ s \leq k}} \Delta(s) , \quad 1 \leq k \leq n , \]
we see that
\[ h(k_1) - h(k_0) = \sum_{\substack{s \in S \\ k_0 < s \leq k_1}} \Delta(s) = \sum_{(j,t) \in \cC(S \cap \{k_0+1,\dots,k_1\})} t \cdot \Delta(j) . \qedhere \]
\end{proof}

\begin{cor}
\label{cor:line-jump-positions}
Conditioned on $|S|$ and $\min (S \cup \{2d+2\})$, $S$ is uniformly distributed over all feasible jump structures $I$ having $|I|=|S|$ and $\min (I \cup \{2d+2\})=\min (S \cup \{2d+2\})$.
\end{cor}
\begin{proof}
By Lemma \ref{lem:line-bijection} and Claim \ref{cl:line-bijection-fluc-size}, $\log_2 |\{S=I\}| = |\cC(I)| + |FP(I)|$ depends only on $|I|$ and $\min (I \cup \{2d+2\})$.
\end{proof}

For $r \geq 0$ and $1 \leq i \leq d+1$, denote by $c_i(r)$ the number of feasible jump structures $I$ having $|I \setminus \{1,\dots,2d+1\}|=r$ and $\min (I \cup \{2d+2\})=2i$ (recalling from \eqref{eq:line-feasible-jump-structure} that $\min I$ is even).
\begin{claim}
\label{cl:line-feasible-jump-structure-count}
For any non-negative integer $r$, we have
\begin{equation}
\label{eq:line-feasible-jump-structure-count}
\begin{aligned}
c_{d+1}(r) &= \binom{\lfloor (n-r-1)/2 \rfloor - (d-1)r}{r} \quad \text{and} \\
c_i(r) &= \binom{\lfloor (n-r)/2 \rfloor - (d-1)r - i}{r} , \quad 1 \leq i \leq d .
\end{aligned}
\end{equation}
\end{claim}
\begin{proof}
By considering the distance between two consecutive values in $I$
and recalling \eqref{eq:line-feasible-jump-structure}, we see that
the number of feasible jump structures $I$ having $|I|=r$ and $\min
I > 2d+1$ (where we set $\min \emptyset:=\infty$) is given by the
number of non-negative integer solutions to the equation
\[ x_1 + x_2 + \cdots + x_r \leq n \]
under the additional constraints that $x_1$ is even and at least $2d+2$ and, for $2 \leq j \leq r$, $x_j$ is odd and at least $2d+1$. Therefore, after substituting $x_1=2y_1+2d+2$ and $x_j=2y_j + 2d+1$ for $2 \leq j \leq r$, we see that $c_{d+1}(r)$ is equal to the number of non-negative integer solutions to the equation
\[ 2(y_1 + \cdots y_r) \leq n - (2d+1)r - 1 ,\]
from which the first result easily follows. Similarly, the number of feasible jump structures $I$ having $|I|=r+1$ and $\min I = 2i$ is given by the number of non-negative integer solutions to the equation
\[ 2i + x_1 + \cdots + x_r \leq n \]
under the additional constraint that, for $1 \leq j \leq r$, $x_j$ is odd and at least $2d+1$. Therefore, substituting $x_j=2y_j + 2d+1$ as before, we see that, for $1 \leq i \leq d$, $c_i(r)$ is equal to the number of non-negative integer solutions to the equation
\[ 2(y_1 + \cdots y_r) \leq n - (2d+1)r - 2i ,\]
from which the second result follows.
\end{proof}

The next lemma and its corollary give bounds on the distribution of $R$. Observe that, by Claim \ref{cl:line-feasible-jump-structure-count}, for any $1 \leq i \leq d+1$, $\Pr\big(R=r ~|~ \min (S \cup \{2d+2\})=2i\big)>0$ when $r$ satisfies $(2d+1)r + 2d \leq n$.
\begin{lemma}
\label{lem:line-ratio-ineq-for-R-cond}
For any positive integer $r$ such that $(2d+1)r+2d \leq n$, we have
\[ \frac{n - Crd}{4r 2^{d}} \leq \frac{\Pr\big(R=r ~|~ \min (S \cup \{2d+2\})=2i\big)}{\Pr\big(R=r-1 ~|~ \min (S \cup \{2d+2\})=2i\big)} \leq \frac{n}{2r 2^d} , \quad 1 \leq i \leq d+1 .\]
\end{lemma}
\begin{proof}
By Lemma \ref{lem:line-bijection}, Claim \ref{cl:line-bijection-fluc-size} and Claim \ref{cl:line-feasible-jump-structure-count}, we have
\[ \left| \{ R=r \} \cap \{ \min(S \cup \{2d+2\}) = 2i \} \right| = c_i(r) b_i(r) , \quad 1 \leq i \leq d+1 , \]
where $c_i(r)$ is given by
\eqref{eq:line-feasible-jump-structure-count} and
\begin{equation}
\label{eq:line-def-b-i}
\begin{aligned}
b_{d+1}(r) &= 2^{\lceil (n-r)/2 \rceil - dr} , \\
b_i(r) &= 2^{\lceil (n-r-1)/2 \rceil - dr - i + 1} , \quad 1 \leq i \leq d .
\end{aligned}
\end{equation}
It is easy to see that
\begin{equation}
\label{eq:line-ratio-ineq-for-R-cond-eq1}
2^{-d-1} \leq \frac{b_i(r)}{b_i(r-1)} \leq 2^{-d}, \quad 1 \leq i \leq d+1 ,
\end{equation}
and a computation shows that
\begin{equation}
\label{eq:line-ratio-ineq-for-R-cond-eq2}
\frac{n - Crd}{2r} \leq \frac{c_i(r)}{c_i(r-1)} \leq \frac{n}{2r} , \quad 1 \leq i \leq d+1 .
\end{equation}
We present this last computation for $i=d+1$. We have
\[ \frac{c_{d+1}(r)}{c_{d+1}(r-1)} = \frac{\lfloor (n-r-1)/2 \rfloor - (d-1)r}{r} \cdot \prod_{j=1}^{r-1} \frac{\lfloor (n-r-1)/2 \rfloor - (d-1)r - j}{\lfloor (n-r)/2 \rfloor - (d-1)(r-1) - j + 1} . \]
Since
\[ n/2 - Cdr \leq \lfloor (n-r-1)/2 \rfloor - (d-1)r \leq n/2 , \]
it suffices to show that the product above is at most $1$ and at least $1 - Cdr/(n-Cdr)$. Indeed, every element in the product is clearly at most $1$, and hence so is the product. For the other inequality, note that the last element in the product is the smallest, so that the product is at least
\begin{align*}
 \left( \frac{\lfloor (n-r-1)/2 \rfloor - dr + 1}{\lfloor (n-r)/2 \rfloor - d(r-1) + 1} \right)^{r-1}
  \geq \left( 1 - \frac{d +1}{n/2 - r/2 - dr + d + 1} \right)^r
  \geq 1 - \frac{4dr}{n - 4dr} .
\end{align*}
The statement now follows directly from \eqref{eq:line-ratio-ineq-for-R-cond-eq1} and \eqref{eq:line-ratio-ineq-for-R-cond-eq2}.
\end{proof}

\begin{cor}
\label{cor:line-ratio-ineq-for-R}
For any positive integer $r$ such that $(2d+1)r+2d \leq n$, we have
\[ \frac{n - Crd}{4r 2^{d}} \leq \frac{\Pr(R=r)}{\Pr(R=r-1)} \leq \frac{n}{2r 2^d} .\]
\end{cor}

\subsection{Proof of theorems}

We are now ready to prove the theorems stated in Section \ref{sec:main-results-line}.

\subsubsection{The supercritical regime}
We prove Theorem \ref{thm:line-supercritical} and Corollary \ref{cor:line-supercritical}.
By Lemma \ref{lem:line-jumps-at-positions}, we have
\[ \Pr(R \geq r) = \Pr\left( \bigcup_{2d+1<s_1<\cdots<s_r \leq n} A_{s_1} \cap \dots \cap A_{s_r} \right)
\leq \binom{n}{r} 2^{-dr} , \quad r \geq 1 . \]
One may easily check that $|\Rng(f)| \leq R + 3$, so that
\[ \Pr\big(|\Rng(f)| \geq 3 + r\big) \leq \Pr(R \geq r) \leq \binom{n}{r} 2^{-dr} , \quad r \geq 1 . \qedhere \]
Moreover, it is easy to describe all homomorphisms which take on at most $3$ values. Denote by $V_0$ and $V_1$ the even and odd vertices in $\{0,1,\dots,n\}$, respectively, and denote by $\Omega_0$ and $\Omega_1$ the set of homomorphisms which are constant on $V_0$ and $V_1$, respectively. Then it is clear that $\{ |\Rng(f)| \leq 3 \} = \Omega_0 \cup \Omega_1$, that $\{ |\Rng(f)|<3\} = \Omega_0 \cap \Omega_1$ and that $|\Omega_0 \cap \Omega_1| = 2$. Also, note that $|V_0| = \lfloor n/2 \rfloor + 1$ and $|V_1| = \lceil n/2 \rceil$, so that $|\Omega_0| = 2^{|V_1|} = 2^{\lceil n/2 \rceil}$ and $|\Omega_1| = 2^{|V_0|} = 2^{\lfloor n/2 \rfloor + 1}$. Therefore,
\[ \Pr\big(|\Rng(f)| < 3\big) = \Pr(\Omega_0 \cap \Omega_1) \leq \frac{|\Omega_0 \cap \Omega_1|}{|\Omega_0 \cup \Omega_1|} = \frac{2}{2^{\lceil n/2 \rceil} + 2^{\lfloor n/2 \rfloor + 1} - 2} \leq 2^{1-n/2} , \]
completing the proof of Theorem \ref{thm:line-supercritical}. To obtain Corollary \ref{cor:line-supercritical}, note that
\[ \frac{\Pr(\Omega_1)}{\Pr(\Omega_0)} = 2^{\lfloor n/2 \rfloor + 1 - \lceil n/2 \rceil} = \begin{cases}
    2 & \text{if } n \text{ is even} \\
    1 & \text{if } n \text{ is odd}
\end{cases} .\]
Hence, if $d - \log_2 n \to \infty$ as $n \to \infty$ then, since
$\Pr(\Omega_0 \cup \Omega_1)=\Pr(|\Rng(f)| \leq 3)=1-o(1)$ by
Theorem \ref{thm:line-supercritical}, we see that
$\Pr(\Omega_0)=1/2-o(1)$ if $n$ is odd and $\Pr(\Omega_0)=1/3-o(1)$
if $n$ is even.

\subsubsection{The subcritical regime}
Before proving the relevant theorems, we need a better understanding of the typical number of jumps.

\begin{lemma}
\label{lem:line-R-large-whp}
For any $\epsilon > 0$, we have
\[ \Pr \left( R < \lfloor \epsilon c n 2^{-d} \rfloor ~|~ \min (S \cup \{2d+2\}) = 2i \right) \leq \epsilon , \quad 1 \leq i \leq d+1 . \]
\end{lemma}
\begin{proof}
Let $0 < \epsilon < 1$ and $1 \leq i \leq d+1$. Lemma \ref{lem:line-ratio-ineq-for-R-cond} implies that if $c$ is small enough,
\[ \frac{\Pr(R=r ~|~ \min (S \cup \{2d+2\}) = 2i)}{\Pr(R=r-1 ~|~ \min (S \cup \{2d+2\}) = 2i)} \geq \frac{1}{\epsilon} , \quad 1 \leq r \leq \lfloor \epsilon c n 2^{-d} \rfloor . \]
Lemma \ref{lem:ratio-prob-tool} now yields the result.
\end{proof}

\begin{cor}
For any $\epsilon > 0$, if $n 2^{-d} \geq C/\epsilon$ then
\begin{equation}
\label{eq:line-R-large-whp}
\Pr \left( R < \epsilon c n 2^{-d} ~|~ \min (S \cup \{2d+2\}) = 2i \right) \leq \epsilon , \quad 1 \leq i \leq d+1 .
\end{equation}
Consequently, if $n 2^{-d} \geq C$ then
\begin{equation}
\label{eq:line-R-lower-expectation}
\E\big[R ~|~ \min (S \cup \{2d+2\}) = 2i \big] \geq c n 2^{-d} , \quad 1 \leq i \leq d+1 .
\end{equation}
\end{cor}
\begin{proof}
If $n 2^{-d} \geq C/\epsilon$ then $\epsilon c' n 2^{-d} \leq \lfloor \epsilon c n 2^{-d} \rfloor$, and hence, \eqref{eq:line-R-large-whp} follows from Lemma \ref{lem:line-R-large-whp}. To obtain \eqref{eq:line-R-lower-expectation}, substitute $\epsilon = 1/2$ in \eqref{eq:line-R-large-whp}.
\end{proof}

We shall also require a similar inequality for the number of jumps up to a given vertex. For $2d+1 \leq k \leq n$, define
\[ R(k) := |S \cap \{2d+2,2d+3,\dots,k\}| = \sum_{i=2d+2}^k \mathbf{1}_{A_i} .\]

\begin{lemma}
\label{lem:line-Rk-lower-expectation}
We have
\[ \E[R(k)] \geq c k 2^{-d} - 1/6 , \quad 1 \leq k \leq n .\]
\end{lemma}
\begin{proof}
First note that the statement is trivial when $n 2^{-d} < C$. Thus,
we may assume that $n 2^{-d} \geq C$. Denote $x_1 := \min S - 1$ and
denote by $x_2,\dots,x_r$ the distances between consecutive values
in $S$. By Corollary \ref{cor:line-jump-positions} and
\eqref{eq:line-feasible-jump-structure}, conditioned on the event
$\{R=r\}$ and on the event $E := \{ \min S > 2d \}$ (again, we set
$\min\emptyset=\infty$), $x_1,\dots,x_r$ are identically distributed
and satisfy $1+x_1 + \cdots + x_r \leq n$. Hence
\[ \E[x_i ~|~ R=r,E] \leq n/r , \quad 1 \leq i \leq r .\]
Since $R(k)<j$ if and only if $1+x_1+\cdots+x_j>k$, we have by Markov's inequality that
\[ \Pr(R(k)<j ~|~ R,E) = \Pr(x_1+\cdots+x_j \geq k ~|~ R,E) \leq \frac{1}{k} \E[x_1+\cdots+x_j ~|~ R,E] \leq \frac{jn}{kR} .\]
So
\[ \E[R(k) ~|~ R,E] \geq \Pr\big(R(k)\geq \lfloor kR/2n \rfloor ~|~ R,E\big) \cdot \lfloor kR/2n \rfloor \geq (1/2) \lfloor kR/2n \rfloor \geq kR/4n - 1/2 . \]
Hence, by the assumption that $n 2^{-d} \geq C$ and by \eqref{eq:line-R-lower-expectation},
\[ \E[R(k) ~|~ E] \geq k \E[R ~|~ E]/4n - 1/2 \geq c k 2^{-d} - 1/2 .\]
Finally, by Lemma \ref{lem:line-jump-at-start}, we have
\[ \E[R(k)] \geq \E[R(k) ~|~ E] \cdot \Pr(E) \geq c k 2^{-d} - 1/6 . \qedhere \]
\end{proof}

\begin{lemma}
\label{lem:line-lower-variance}
We have
\[ \Var(f(k)) \geq 1 , \quad 1 \leq k \leq n . \]
\end{lemma}
\begin{proof}
If $k$ is odd then $|f(k)| \geq 1$ and the result follows by the fact that $f(k)$ is symmetric. Henceforth, we assume that $k$ is even.

Consider the mapping $f \mapsto f_0$ from $A_2$ to $A_2^c$ defined by
\[ f_0(i) :=
    \begin{cases}
        0&\text{if } i = 0 \\
        f(i) - f(2) &\text{if } i \geq 1
    \end{cases} , \quad 0 \leq i \leq n .
\]
One may check that this mapping is indeed well-defined and that it is injective. Since $|f_0(k)|=2$ when $f(k)=0$, and since the mapping is injective, we have
\begin{align*}
\Pr(\{f(k)=0\} \cap A_2) & \leq \Pr(\{|f(k)|=2\} \cap A_2^c) \\
& \leq \Pr(\{ f(k) \neq 0\} \cap A_2^c) = 1 - \Pr(\{f(k)=0\} \cup A_2) .
\end{align*}
Therefore,
\[ \Pr(f(k)=0) + \Pr(A_2) = \Pr(\{f(k)=0\} \cup A_2) + \Pr(\{f(k)=0\} \cap A_2) \leq 1 .\]
Since, $\Pr(A_2) \geq 1/4$ by Lemma \ref{lem:line-jump-at-start}, we have
\[ \Pr(f(k)=0) \leq 3/4 .\]
Finally, since $f(k) \neq 0$ if and only if $|f(k)| \geq 2$, we have
\[ \Var(f(k)) = \E\left[f(k)^2\right] \geq 4 \cdot \Pr(|f(k)| \geq 2) = 4 \cdot \Pr(f(k) \neq 0) \geq 1 . \qedhere \]
\end{proof}

We are now ready to prove Theorem \ref{thm:line-subcritical-endpoint} and Theorem \ref{thm:line-subcritical-range}. In both proofs, we consider the following modified average height $h'$. For $1 \leq k \leq n$, define
\[ h'(k) := \begin{cases} 0 &\text{if } k \leq 2d+1 \\ h(k) - h(2d+1) &\text{otherwise} \end{cases} .\]
Recall that Corollary \ref{cor:line-indep-chain-signs} implies that, for any $2d+2 \leq k \leq n$, conditioned on $S$, $\{ \Delta(j) ~|~ (j,t) \in \cC(S \cap \{2d+2,\dots,k\}) \}$ are independent and
\begin{equation}
\label{eq:line-h-tag-formula}
h'(k) = \sum_{(j,t) \in \cC(S \cap \{2d+2,\dots,k\})} t \cdot \Delta(j) , \quad 1 \leq k \leq n .
\end{equation}

\begin{proof}[Proof of Theorem \ref{thm:line-subcritical-endpoint}]
By the above remark, we have
\begin{equation}
\label{eq:line-var-h-tag-k-1}
\Var(h'(k) ~|~ S) = \sum_{(j,t) \in \cC(S \cap \{2d+2,\dots,k\})} t^2 , \quad 1 \leq k \leq n .
\end{equation}
Notice that, conditioned on $S$, the expectation of $h'(k)$ is zero, so that by the law of total variance,
\begin{equation}
\label{eq:line-var-h-tag-k-2}
\Var(h'(k)) = \E[\Var(h'(k) ~|~ S)] , \quad 1 \leq k \leq n .
\end{equation}

To obtain an upper bound on $\Var(h'(k))$, we use Lemma \ref{lem:line-jumps-at-positions} to obtain
\[ \Pr\big((j,t) \in \cC(S \cap \{2d+2,\dots,k\})\big) \leq \Pr(C_{j,t}) \cdot \mathbf{1}_{\{(2d+1)t < j \leq k \}} \leq 2^{-dt} \cdot \mathbf{1}_{\{ j \leq k \}} , \]
for any $1 \leq j \leq n$ and $t \geq 1$. Therefore, by \eqref{eq:line-var-h-tag-k-1} and \eqref{eq:line-var-h-tag-k-2}, we have
\begin{equation}
\label{eq:line-upper-bound-variance-h-tag-k}
\Var(h'(k)) = \E\big[\Var(h'(k)  ~|~ S)\big] \leq \sum_{t=1}^{\infty} \sum_{j=1}^{n} t^2 2^{-dt} \cdot \mathbf{1}_{\{ j \leq k \}} = k \sum_{t=1}^{\infty} t^2 2^{-dt} \leq C k 2^{-d} .
\end{equation}
Finally, since $|f(k)-h(k)| \leq 1$ and $|h(k)-h'(k)| \leq 1$, we have $|f(k)| \leq |h'(k)| + 2$, which gives
\begin{align*}
\Var(f(k)) &= \E\left[f(k)^2\right] \leq \E\left[(|h'(k)| +
2)^2\right] = \E\left[h'(k)^2 + 4|h'(k)|  + 4 \right] \\
 &\leq 5 \cdot \E\left[h'(k)^2\right] + 4 = 5 \cdot \Var(h'(k)) + 4 \leq C k 2^{-d} + 4 .
\end{align*}

For the lower bound, we note that
\[ \sum_{(j,t) \in \cC(S \cap \{2d+2,\dots,k\})} t^2 \geq \sum_{(j,t) \in \cC(S \cap \{2d+2,\dots,k\})} t = |S \cap \{2d+2,\dots,k\}| = R(k) . \]
Therefore, by \eqref{eq:line-var-h-tag-k-1}, \eqref{eq:line-var-h-tag-k-2} and Lemma \ref{lem:line-Rk-lower-expectation}, we have
\[ \Var(h'(k)) =  \E[\Var(h'(k) ~|~ S)] \geq \E[R(k)] \geq c k 2^{-d} - 1/6 .\]
Since $|f(k)-h(k)| \leq 1$ and $|h(k)-h'(k)| \leq 1$, we have $|f(k)| \geq |h'(k)| - 2$. In particular, $|f(k)| \geq |h'(k)|/3$ when $|h'(k)| \geq 3$. Therefore,
\begin{align*}
 \Var(f(k)) = \E\left[f(k)^2\right] &\geq \E\left[(h'(k)/3)^2 \cdot \mathbf{1}_{\{|h'(k)| \geq 3\}} \right] \\
 &= \E\left[h'(k)^2\right]/9 - \E\left[h'(k)^2 \cdot \mathbf{1}_{\{|h'(k)| \leq 2\}} \right] / 9 \\
 &\geq \Var(h'(k))/9 - 4/9 \\
 &\geq c k 2^{-d} - 1/2 .
\end{align*}
Finally, together with Lemma \ref{lem:line-lower-variance}, we have
\[ \Var(f(k)) \geq \max\{1, c k 2^{-d} - 1/2 \} \geq \max\{1, c' k 2^{-d} \} . \qedhere \]
\end{proof}

\begin{proof}[Proof of upper bound in Theorem \ref{thm:line-subcritical-range}]

Denote $\cC(S \cap \{2d+2,\dots,n\}) = \{ (k_1,t_1), \dots, (k_m,t_m) \}$, ordering the elements so that the $k_j$ are increasing. Observe that for any $1 \leq j < m$ and any $k_j \leq k \leq k_{j+1}$, we have that $h(k)$ is between $h(k_j)$ and $h(k_{j+1})$. Therefore,
\[ \max_{1 \leq k \leq n} |h'(k)|  = \max_{1 \leq j \leq m} |h'(k_j)| .\]
In this notation, by \eqref{eq:line-h-tag-formula} we have
\[ h'(k_j) = \sum_{i=1}^j t_i \cdot \Delta(k_i) , \quad 1 \leq j \leq m , \]
where, conditioned on $S$, $\{ \Delta(k_j) ~|~ 1 \leq j \leq m \}$ are independent. Therefore, we may apply Kolmogorov's maximal inequality to the process $(h'(k_j) ~|~ 1 \leq j \leq m)$, conditioned on $S$, to obtain
\[ \Pr\left(\max_{1 \leq k \leq n} |h'(k)| \geq x ~|~ S \right) = \Pr\left(\max_{(k,t) \in \cC(S \cap \{2d+2,\dots,n\})} |h'(k)| \geq x ~|~ S \right) \leq \frac{\Var(h'(n) ~|~ S)}{x^2} .\]
Therefore, by \eqref{eq:line-upper-bound-variance-h-tag-k}, we have
\[ \Pr\left(\max_{1 \leq k \leq n} |h'(k)| \geq x \right) =
\E\left[\Pr\left(\max_{1 \leq k \leq n} |h'(k)| \geq x ~|~ S \right)\right] \leq
\frac{\Var(h'(n))}{x^2} \leq \frac{C n 2^{-d}}{x^2} .\]
From this we obtain
\[ \E\left[\max_{1 \leq k \leq n} |h'(k)|\right]
 = \sum_{x=1}^{\infty} \Pr\left(\max_{1 \leq k \leq n} |h'(k)| \geq x \right)
 \leq \sum_{x=1}^{\infty} \min\left\{1, \frac{C n 2^{-d}}{x^2} \right\}
 \leq C' \sqrt{n 2^{-d}}. \]
Finally, using the fact that
\[ |\Rng(f)| \leq 3 + \max_{1 \leq k \leq n} h'(k) - \min_{1 \leq k \leq n} h'(k) , \]
we obtain
\[ \E\big[|\Rng(f)|\big] \leq 3 + 2 \cdot \E\left[\max_{1 \leq k \leq n} |h'(k)|\right] \leq 3 + C \sqrt{n 2^{-d}} . \qedhere \]
\end{proof}

\begin{proof}[Proof of lower bound in Theorem \ref{thm:line-subcritical-range}]

We begin by showing that the range is large with high probability, when $n 2^{-d}$ is large enough. Fix $0 < \epsilon < 1$. Assume that $n 2^{-d} \geq C/\epsilon$. By \eqref{eq:line-R-large-whp}, there exists a $\delta_1>0$, depending only on $\epsilon$, such that
\begin{equation}
\label{eq:line-R-large-whp2}
\Pr \left( R < 2\delta_1 n 2^{-d} \right) \leq \epsilon/4 .
\end{equation}
This tells us that typically there are many jumps. We now show that typically there are many distinct chains as well. For $s \geq 1$, let
\[ M_s := \sum_{(2d+1)s < k \leq n} \mathbf{1}_{C_{k,s}} \]
be the number of sub-chains of length $s$. Then, as we shall now show,
\begin{equation}
\label{eq:line-M-R-C-relation}
|\cC(S)| \geq \frac{R-M_s}{s-1} , \quad s \geq 2 .
\end{equation}
Indeed, denoting $\cC(S \cap \{2d+2,\dots,n\}) = \{ (k_1,t_1), \dots, (k_m,t_m) \}$ and considering the contribution of each chain to $M_s$, we see that
\[ M_s = \sum_{i=1}^m \max\{0, t_i - s + 1\} \geq \sum_{i=1}^m (t_i - s + 1) = R - (s-1)m .\]
Noting that $|\cC(S)| \geq m$ now yields \eqref{eq:line-M-R-C-relation}. By Lemma \ref{lem:line-jumps-at-positions}, we have
\[ \E[M_s] = \sum_{(2d+1)s < k \leq n} \Pr(C_{k,s}) \leq n 2^{-ds} \leq n 2^{-d} 2^{1-s} , \quad s \geq 1 .\]
Taking $s=s_0$ large enough, we have by Markov's inequality,
\begin{equation}
\label{eq:line-M-small-whp}
\Pr\left(M_{s_0} \geq \delta_1 n 2^{-d}\right) \leq \frac{\E[M_{s_0}]}{\delta_1 n 2^{-d}} \leq \frac{2^{1-s_0}}{\delta_1} \leq \epsilon/4 .
\end{equation}
Therefore, by \eqref{eq:line-R-large-whp2}, \eqref{eq:line-M-R-C-relation} and \eqref{eq:line-M-small-whp}, we have for $\delta_2 := \delta_1/s_0$ that
\begin{equation}
\label{eq:line-many-chains-whp}
\Pr\left(|\cC(S)| \geq \delta_2 n 2^{-d}\right) \geq \Pr\left(R \geq 2\delta_1 n 2^{-d}, M_{s_0} \leq \delta_1 n 2^{-d}\right) \geq 1 - \epsilon/2 .
\end{equation}
Recalling from Corollary \ref{cor:line-indep-chain-signs} that, conditioned on $S$, $h(n)$ is the sum of $|\cC(S)|$ independent random variables, we may apply Theorem \ref{thm:random-sign-walk} to obtain
\[ \Pr\big(|h(n)| < r ~|~ S\big) \leq \frac{Cr}{\sqrt{|\cC(S)|}} , \quad r \in \N . \]
Therefore,
\begin{equation}
\label{eq:line-range-cond-on-chains}
\Pr\big(|\Rng(f)| < r ~|~ |\cC(S)|\big) \leq \Pr\big(|h(n)| < r ~|~ |\cC(S)| \big) \leq \frac{Cr}{\sqrt{|\cC(S)|}} , \quad r \in \N .
\end{equation}
Finally, by \eqref {eq:line-many-chains-whp} and \eqref{eq:line-range-cond-on-chains}, for any $\delta>0$, we have
\begin{align*}
\Pr\left(|\Rng(f)| < \big\lfloor \delta \sqrt{n2^{-d}} \big\rfloor \right)
&\leq \Pr\left(|\cC(S)| < \delta_2 n2^{-d}\right) + \Pr\left(|\Rng(f)| < \big\lfloor \delta \sqrt{n2^{-d}} \big\rfloor ~|~ |\cC(S)| \geq \delta_2 n2^{-d} \right) \\
&\leq \epsilon/2 + C \delta / \sqrt{\delta_2} .
\end{align*}
Therefore, there exists a $\delta>0$, depending only on $\epsilon$, such that if $\delta \sqrt{n 2^{-d}} \geq 1$ then
\[ \Pr\left(|\Rng(f)| < 3 + \big\lfloor \delta \sqrt{n2^{-d}} \big\rfloor \right) \leq \Pr\left(|\Rng(f)| < \big\lfloor 4 \delta \sqrt{n2^{-d}} \big\rfloor \right) \leq \epsilon ,\]
proving \eqref{eq:thm-line-subcritical-range-large-whp} when $n 2^{-d}$ is large enough. On the other hand, if $\delta \sqrt{n 2^{-d}} < 1$ then \eqref{eq:thm-line-subcritical-range-large-whp} follows immediately from Theorem \ref{thm:line-supercritical}, since
\[ \Pr\left(|\Rng(f)| < 3 + \big\lfloor \delta \sqrt{n2^{-d}} \big\rfloor \right) = \Pr\left(|\Rng(f)| < 3 \right) \leq 2^{1-n/2} . \]
It remains to show the lower bound on the expectation. Note that the statement is trivial when $n \leq 2$, and so we may assume that $n \geq 3$. By taking $\epsilon=1/4$ in \eqref{eq:thm-line-subcritical-range-large-whp}, noting that $|\Rng(f)| \geq 2$ and by Theorem \ref{thm:line-supercritical}, we conclude that
\begin{equation} \belowdisplayskip=-12pt
\label{eq:line-subcritical-range-expectation-lower-bound}
\begin{aligned}
\E\big[|\Rng(f)|\big] &= 3 + \E\big[(|\Rng(f)| - 3) \cdot \mathbf{1}_{\{|\Rng(f)| \geq 3\}}\big] + \E\big[(|\Rng(f)| - 3) \cdot \mathbf{1}_{\{|\Rng(f)| < 3\}}\big] \\
&= 3 + \E\big[(|\Rng(f)| - 3) \cdot \mathbf{1}_{\{|\Rng(f)| \geq 3\}}\big] - \Pr\left(|\Rng(f)| < 3\right) \\
&\geq 3 + (1 - 1/4 - 2^{-1/2}) \big\lfloor \delta \sqrt{n 2^{-d}} \big\rfloor - 2^{1-n/2} \\
&\geq 3 + \big\lfloor c \sqrt{n 2^{-d}} \big\rfloor - 2^{1-n/2} .
\end{aligned}
\end{equation}
\end{proof}

\subsubsection{The critical regime}
Here we prove Theorem \ref{thm:line-critical-range}. Denote $\lambda := \lim n 2^{-d}$ which exists and is a positive number by assumption. The proof of  Theorem \ref{thm:line-critical-range} consists of two parts. First, we show that $R$ converges to $N^{\pm}(\lambda)$ as $n$ tends to infinity through even or odd integers. Next, we show that in this regime the values at the jumps constitute a simple random walk and that this walk determines the range of the homomorphism.

By Lemma \ref{lem:line-jumps-at-positions}, we have
\[ \E[R] = \sum_{k=2d+2}^n \Pr(A_k) \leq n 2^{-d} = \lambda + o(1) .\]
Therefore, the expectation of $R$ is uniformly bounded as $n \to \infty$, and hence, Markov's inequality implies that $R$ is tight as $n \to \infty$.
Using notation as in the proof of Lemma \ref{lem:line-ratio-ineq-for-R-cond}, we have
\[ \left| \{ R=r \} \cap \{ \min(S \cup \{2d+2\}) = 2i \} \right| = c_i(r) b_i(r) , \quad 1 \leq i \leq d+1 . \]
A direct computation shows that for any constant $r \geq 0$, we have
\[ c_i(r) \sim \frac{n^r}{2^r r!} \quad \text{ as } n \to \infty , \] uniformly in $1 \leq i \leq d+1$, and
\[ 2^{\lceil (n-r)/2 \rceil} \sim 2^{(n-r)/2} \cdot \gamma(n-r) \quad \text{ as } n \to \infty , \]
where $\gamma(k):=1$ if $k$ is even and $\gamma(k):=\sqrt{2}$ if $k$ is odd. Denoting by $J := A_1 \cup \cdots \cup A_{2d+1}$ the event that a jump occurs prior to vertex $2d+2$, and recalling \eqref{eq:line-def-b-i}, we obtain
\begin{align*}
\left| \{R=r\} \cap J^c \right| &\sim \frac{n^r}{r!} \cdot 2^{n/2 - (d+3/2)r} \cdot \gamma(n-r) , \\
\left| \{R=r\} \cap A_{2i} \right|  &\sim \frac{n^r}{r!} \cdot 2^{n/2 - (d+3/2)r + 1/2 - i} \cdot \gamma(n-r-1) ,
\end{align*}
uniformly in $1 \leq i \leq d$. Therefore, denoting $\lambda' := \lambda / (2 \sqrt{2})$, we have
\[ |\{ R=r \}| \sim \frac{\lambda'^r}{r!} \cdot 2^{n/2} \cdot \left( \gamma(n-r) + \sqrt{2} \gamma(n-r-1) \right) ,\]
where we have used the fact that $\sum_{i=1}^d 2^{-i} \to 1$. Using the tightness of $R$, we see that
\[ \Pr(R=r) \sim Z(n)^{-1} \cdot \frac{\lambda'^r}{r!} \cdot \left( \gamma(n-r) + \sqrt{2} \gamma(n-r-1) \right) ,\]
where $Z(n)$ is a normalizing constant. Therefore, recalling the parity-biased Poisson distribution defined in \eqref{eq:parity-biased-poisson} and the equation \eqref{eq:parity-biased-poisson-prob}, we see that
\begin{equation}
\label{eq:line-critical-R-convergence}
R \xrightarrow[\substack{n \to \infty \\ n \text{ is even}}]{(d)} ~ \mu\big(\lambda', 3/(2\sqrt{2})\big) \quad\text{ and }\quad
R \xrightarrow[\substack{n \to \infty \\ n \text{ is odd}}]{(d)} ~ \mu\big(\lambda', 2\sqrt{2}/3\big) ,
\end{equation}
completing the first part of the proof.

We remark that it is also possible to obtain the limiting
distribution of $R$ conditioned on whether or not a jump occurred at
the first $2d+1$ vertices. We do not make use of this in our paper
but we note the final result. A further calculation gives the
following formula for the asymptotic probability of $J$,
\[ \left(\sqrt{2} \cdot \frac{\Pr(J^c)}{\Pr(J)}\right)^{(-1)^n} \xrightarrow[n \to \infty]{} \frac{\cosh(\lambda') + \sqrt{2} \sinh(\lambda')}{\sinh(\lambda') + \sqrt{2} \cosh(\lambda') } ,\]
and the following formula for the asymptotic distribution of $R$ given $\mathbf{1}_J$,
\[ \Pr(R=r ~|~ \mathbf{1}_J) - \mu\Big(\lambda', \sqrt{2}^{(-1)^{n+1+\mathbf{1}_J}}\Big)(r) \xrightarrow[n \to \infty]{} 0 .\]

We now proceed to analyze the range of a typical homomorphism in the critical regime. We begin by showing that the jumps are sparse enough so that it is unlikely to have chains of length greater than one. Let
\[ B := \bigcap_{k=2d+2}^n (A^c_k \cup A^c_{k-2d-1}) \]
be the event that there are no two minimal-distance jumps (i.e.
jumps at distance $2d+1$). We wish to show that $\Pr(B) = 1 - o(1)$.
Indeed, by considering the first $2d+1$ elements in the intersection
separately from the rest, Lemma \ref{lem:line-jumps-at-positions}
implies that $\Pr(B^c) \leq (2d+1) 2^{-d} + n 2 ^{-2d} = o(1)$ as
required. Notice that $\mathbf{1}_B$ is determined by $S$. Let
$\mathcal{I}$ denote the set of all feasible jump structures $I$
such that $\{S=I\} \subset B$. Observe that on the event $B$,
$\cC(S)=\{ (s,1) ~|~ s \in S \}$. Therefore, by Corollary
\ref{cor:line-indep-chain-signs}, for any $I \in \mathcal{I}$,
conditioned on $S=I$, $\{ \Delta(s) ~|~ s \in S \}$ are independent
uniform signs and, by \eqref{eq:line-h-tag-formula},
\[ h'(k) = \sum_{s \in S \cap \{ 2d+2,\dots,k \}} \Delta(s) .\]
In other words, for any $I \in \mathcal{I}$, conditioned on $S=I$,
$(h'(s) ~|~ 2d+1 < s \in S)$ is a simple random walk of length $R$
(without the leading zero). Since
\[ \{ h'(s) ~|~ 2d+1 < s \in S \} \cup \{0\} = \{ h'(k) ~|~ 2d+1 \leq k \leq
n \} , \]
for any $I \in \mathcal{I}$,
\begin{equation}
\label{eq:line-critical-range-h-tag-like-srw}
|\{ h'(k) ~|~ 2d+1 \leq k \leq n \}| \eqd |\Rng(S_R)| \quad \text{ conditioned on } S=I ,
\end{equation}
where $S_i$ is an independent simple random walk run for $i$ steps.
Define the event
\[ E := \{ |\{ f(k) ~|~ 0 \leq k \leq 2d \}|>2 \} .\]
It is not difficult to check that
\begin{equation}
\label{eq:line-critical-range-h-tag-equals-range-f}
|\Rng(f)| = 2 + |\{ h'(k) ~|~ 2d+1 \leq k \leq n \}| \quad \text{on } E .
\end{equation}
We now show that $\Pr(E)=1-o(1)$. Observe that Lemma \ref{lem:line-bijection} implies that
\[ \Pr(E^c ~|~ \min S \geq 4d+2) = 2 \cdot 2^{-d} . \]
Hence, by Lemma \ref{lem:line-jumps-at-positions}, and since $J \subset E$, we have
\begin{align*}
\Pr(E^c) &= \Pr(E^c \cap \{ \min S < 4d+2 \}) + \Pr(E^c \cap \{ \min S \geq 4d+2 \}) \\
 &\leq \Pr(2d+2 \leq \min S < 4d+2) + \Pr(E^c ~|~ \min S \geq 4d+2) \\
 &\leq 2d 2^{-d} + 2^{1-d} = o(1) .
\end{align*}
Finally, Theorem \ref{thm:line-critical-range} follows from \eqref{eq:line-critical-R-convergence}, \eqref{eq:line-critical-range-h-tag-like-srw}, \eqref{eq:line-critical-range-h-tag-equals-range-f} and the fact that $\Pr(E \cap B)=1-o(1)$.

\section{Homomorphisms on the Torus}
\label{sec:torus}

In this section we prove the theorems regarding homomorphisms on the
torus which were stated in Section \ref{sec:main-results-torus}. The
ideas and notions previously introduced in Section \ref{sec:line} to
handle the case of homomorphisms on the line will still prove to be effective on the
torus, although some of them will need to be adapted. For example,
the notions of average height, jumps and chains will still be used
and they are defined in an analogous manner. One thing which must
change, for instance, is how we use these notions and the events
that we condition on. Note that, if we condition on the lengths and
the positions of the chains, their signs will not be independent,
since they must add up correctly. This fact, which is inherently due
to the topology of the torus, makes the analysis slightly more
complex. Instead, we will show that, conditioned on the lengths and
the signs of the chains (but not on their positions), their relative
order is uniform. This will allow us to reduce some of the analysis
to a case of a uniformly chosen reordering of a sequence of numbers.
One aspect which is simpler for homomorphisms on the torus is that
there are no boundary effects, i.e., no need to consider the first
$2d+1$ vertices separately.

\subsection{Definitions}

We consider the graph $T_{n,d}$, $n$ even, whose vertex set is $V:=\{0,\dots,n-1\}$ and whose edges are defined by $i \sim j$ if and only if $\rho(i,j) = 1, 3,\dots,2d+1$, where we define the {\em distance} $\rho$ between $x$ and $y$ to be
\[ \rho(x,y) := \min\{ |x-y|, n - |x-y| \} , \quad x,y \in V .\]
We define also the {\em clockwise distance} $\rho^+$ from $x$ to $y$ to be
\[ \rho^+(x,y) := y-x + n \mathbf{1}_{\{x>y\}} = \min \{ k \geq 0 ~|~ x+k=y \bmod n \} , \quad x,y \in V .\]
Note that $\rho^+(x,y)+\rho^+(y,x)=n$ and that $\rho(x,y)=\min\{\rho^+(x,y),\rho^+(y,x)\}$ for any $x,y \in V$.

Throughout this section, $\Hom(T_{n,d}):=\Hom(T_{n,d},0)$, $f$ is a uniformly sampled homomorphism from $\Hom(T_{n,d})$, the probability space is the uniform distribution on the set $\Hom(T_{n,d})$, and events are subsets of $\Hom(T_{n,d})$. We also note that, in this section, addition and subtraction of elements in $V$ are always modulo $n$.

We would like to define the notion of the (local) average height of a homomorphism at a vertex $x \in V$. To do so, we ``look back'' just enough in order to define this in a meaningful way. Precisely, for $x \in V$, define the {\em average height} at $x$ as the unique number $h(x)$ satisfying
\[ \text{There exists a } k \geq 1 \text{ for which } \{ f(x - i) ~|~ i=0,1,\dots,k \} = \{ h(x) - 1, h(x), h(x) + 1 \} . \]
This is well defined for any homomorphism $f$ which takes on at
least $3$ values. There are two specific homomorphisms for which the
size of the range is $2$, and hence for which this is not well
defined. These are $f^{\text{flat}}_1$ and $f^{\text{flat}}_{-1}$,
where
\[ f^{\text{flat}}_i(x) := \begin{cases} 0 &\text{if } x \text{ is even} \\ i &\text{if } x \text{ is odd} \end{cases} , \quad x \in V, ~ i \in \{-1,1\} . \]
For these homomorphisms we define $h(x):=0$ for all $x \in V$. For $x \in V$, define
\begin{align*}
\Delta(x) &:= h(x) - h(x-1) , \\
A_x &:= \{ \Delta(x) \neq 0 \} .
\end{align*}
Observe that necessarily $\Delta(x)\in\{-1,0,1\}$. When $A_x$
occurs, we will say that a {\em jump} occurred at vertex $x$. For $x,y \in V$, denote by
\[ A_{x,y} := A_x \cap A_y \cap \{ \Delta(x) = - \Delta(y)\} \]
the event that there are jumps in opposite directions at $x$ and $y$. Denote by
\[ S^+ := \{ x \in V ~|~ \Delta(x) = 1 \} \quad \text{and} \quad S^- := \{ x \in V ~|~ \Delta(x) = - 1 \} \]
the sets of vertices at which a positive or negative jump occurred,
respectively. Let
\begin{equation}\label{eq:S_torus_def}
S := S^+ \cup S^-
\end{equation}
be the set of vertices at which we have a jump in either direction.
Notice that necessarily $|S^+|=|S^-|$, and define
\[ R := |S^+| = |S^-| = |S|/2 , \]
the number of jumps in a given direction. Notice that the clockwise
distance between jumps is at least $2d+1$, as for homomorphisms on
the line. For $x\in V$ and $t\ge 1$, let
\[ C_{x,t} := A_x \cap A_{x-2d-1} \cap \cdots \cap A_{x-(t-1)(2d-1)} \]
be the event that there is a chain of $t$ minimal-distance jumps ending at vertex $x$.

We say that a subset $I \subset V$ is a {\em feasible jump
structure} if $\{S = I\} \neq \emptyset$, i.e. if $\Pr(S=I)>0$. We
would like to describe this condition solely in terms of the
structure of $I$. To this end, write $I = \{s_1,\dots,s_t\}$, where
$0 \leq s_1<\cdots<s_t<n$ and let $s_0 := s_t$. Similarly to the
case of the line, see
condition~\eqref{eq:line-feasible-jump-structure}, the following
conditions are necessary for $I$ to be a feasible jump structure.
\begin{equation}
\label{eq:torus-feasible-jump-structure1}
\begin{aligned}
\rho^+(s_{j-1},s_j) \geq 2d+1 , & \quad 1 \leq j \leq t , \\
\rho^+(s_{j-1},s_j) \text{ is odd} , & \quad 1 \leq j \leq t .
\end{aligned}
\end{equation}
In contrast to the case of the line, these conditions alone are not sufficient for $I$ to be a feasible jump structure. This is due to the fact that the torus imposes a topological constraint. Namely, that at the end of the homomorphism the average height must ``return'' to its initial value. This additional condition, whose precise description \eqref{eq:torus-feasible-jump-structure2} we postpone to the next section, along with condition \eqref{eq:torus-feasible-jump-structure1}, is necessary and sufficient for $I$ to be a feasible jump structure.

In addition, we say that a subset $I \subset V$ is a {\em feasible jump sub-structure} if it is a subset of a feasible jump structure, or equivalently, if $\{I \subset S\} \neq \emptyset$. Notice that the definition implies that condition \eqref{eq:torus-feasible-jump-structure1} is necessary for $I$ to be a feasible jump sub-structure. For any $I \subset V$ satisfying \eqref{eq:torus-feasible-jump-structure1}, by considering the connected components of the subgraph of $T_{n,d}$ induced by $I$, one may see that the event $\{I \subset S\}$ can be uniquely written as $C_{k_1,t_1} \cap \dots \cap C_{k_m,t_m}$, where
\begin{equation}
\label{eq:torus-chain-structure}
\begin{aligned}
& 0 \leq k_1<\cdots<k_m<n ,\\
& t_1+\cdots+t_m = |I| ,\\
& \rho^+(k_{j-1},k_j) > (2d+1)t_j , \quad 1 \leq j \leq m ,
\end{aligned}
\end{equation}
and where we let $k_0 := k_m$ (see Figure \ref{fig:torus-chain-structure}). These conditions ensure that there is no overlap between the different chains, and moreover, that there is some gap between them (since otherwise they would merge into a larger chain). For a subset $I \subset V$ satisfying \eqref{eq:torus-feasible-jump-structure1}, we define
\[ \cC(I) := \{ (k_j, t_j) ~|~ 1 \leq j \leq m \} ,\]
and refer to this as the {\em chain structure} of $I$.

\definecolor{darkgreen}{rgb}{0.2,0.6,0.2}

\def\arrow[#1][#2][#3][#4]{
  (0:#1) arc (0:#3:#1) [rounded corners=0.5] --
  (#3:#1) -- ({#3+#4}:{#1/2+#2/2}) [rounded corners=0.5] --
  (#3:#2) arc (#3:0:#2) [rounded corners=1] --
  (-#4:{#1/2+#2/2}) -- (0:#1)
}

\def\flucline[#1][#2][#3]{
  (0:{#1/2+#2/2}) arc (0:#3:{#1/2+#2/2})
}

\def\jumpline[#1][#2][#3]{
  (0:#1) -- (#3:#2)
}

\def\chain[#1][#2][#3][#4][#5][#6][#7]{
    \begin{scope}[scale=#5]
        \draw[color=darkgreen, bottom color=green!90!black, top color=green!60] [rotate={#1*#3}] \arrow[#6][#7][{#1*#4*#2}][4];
    \end{scope}
}

\def\fluc[#1][#2][#3][#4][#5][#6]{
    \begin{scope}[scale=#4]
        \draw[color=red, decorate, decoration=snake] [rotate={#1*#2}] \flucline[#5][#6][{#1*#3}];
    \end{scope}
}

\def\jump[#1][#2][#3][#4][#5][#6]{
    \begin{scope}[scale=#4]
        \draw[color=blue] [rotate={#1*#2}] (#1:#6) arc (#1:{#1*#3}:#6);
        \draw[color=blue] [rotate={#1*#2}] (#1:#5) -- (#1:#6);
        \draw[color=blue] [rotate={#1*#2}] (0:#5) -- (#1:#5);
    \end{scope}
}

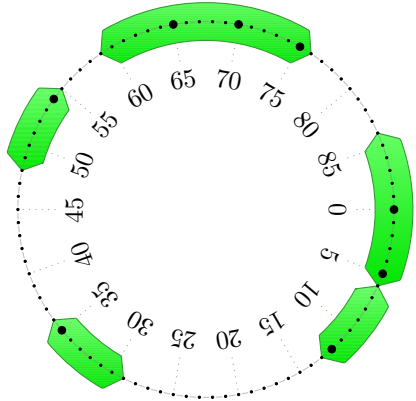
\begin{figure}[!t]
\centering
  \begin{tikzpicture}

    \def \d {5} //2d+1
    \def \n {90}
    \def \radius {2.5cm}
    \def \margin {1} 
    \def \tick {4}
    \def \r {2.5}
    \def \ra {0.9}
    \def \rb {1.1}

    \chain[\tick][\d][15][3][\r][\ra][\rb]
    \chain[\tick][\d][36][1][\r][\ra][\rb]
    \chain[\tick][\d][55][1][\r][\ra][\rb]
    \chain[\tick][\d][78][1][\r][\ra][\rb]
    \chain[\tick][\d][85][2][\r][\ra][\rb]

    \tikzstyle{every node}=[draw, circle, fill,minimum size=0.03cm,inner sep=0]
    \foreach \s in {1,...,\n}
    {
      \node at ({360/\n * (\s - 1)}:\radius) {};
      \draw[-, >=latex,thin,gray,opacity=0.5] ({360/\n * (\s - 1)+\margin}:\radius)
        arc ({360/\n * (\s - 1)+\margin}:{360/\n * (\s)-\margin}:\radius);
    }

    \foreach \s in {0,5,...,85}
    {
      \node[draw=none,fill=none, scale=0.9, rotate=(270-\s*\tick)] at ({\s * \tick * -1}:{\radius*0.7}) {\s};
      \draw[dotted,thin,gray!75] ({\s * \tick}:{\radius*0.8}) -- ({\s * \tick}:\radius);
    }

    \tikzstyle{every node}=[draw, circle, fill,minimum size=0.1cm,inner sep=0]
    \node[] at ({\tick * 15}:\radius) {};
    \node[] at ({\tick * 20}:\radius) {};
    \node[] at ({\tick * 25}:\radius) {};
    \node[] at ({\tick * 36}:\radius) {};
    \node[] at ({\tick * 55}:\radius) {};
    \node[] at ({\tick * 78}:\radius) {};
    \node[] at ({\tick * 85}:\radius) {};
    \node[] at ({\tick * 90}:\radius) {};

  \end{tikzpicture}
\caption{Given a subset $I \subset V$ satisfying \eqref{eq:torus-feasible-jump-structure1}, we construct $\cC(I)$, the chain structure of $I$, by partitioning the elements of $I$ according to the connected components in the subgraph of $T_{n,d}$ induced by $I$. The elements of $I$ are denoted by large vertices and the chain structure is denoted by blocks surrounding the vertices. In the figure, $n=90$, $d=2$ and $I=\{0,5,12,35,54,65,70,75\}$, and hence, $\cC(I)=\{(5,2),(12,1),(35,1),(54,1),(75,3)\}$.}
\label{fig:torus-chain-structure}
\end{figure}

\subsection{The structure of a homomorphism}

In this section, our goal is to a give a useful description of the structure of a homomorphism on the torus. Namely, that which is stated in Lemma \ref{lem:torus-bijection} and Lemma \ref{lem:torus-bijection2} below. To this end, we would like to decompose a homomorphism into two parts (see Figure \ref{fig:line-bijection} and Figure \ref{fig:torus-bijection}). The first part, which we shall denote by $X$, constitutes the changes in average height (the underlying bridge) of the homomorphism, while the second part, which we shall denote by $Y$, constitutes the fluctuations around the average height (the segments of constant average height).

We proceed first to define $X$. Given a subset $I \subset V$ satisfying \eqref{eq:torus-feasible-jump-structure1}, denote the set of {\em feasible sign vectors} for $I$ by
\[ \cB^*(I) := \left\{ \epsilon \in \{-1,1\}^{\cC(I)} ~|~ \sum_{(k,t) \in \cC(I)} \epsilon_{(k,t)} t = 0 \right\} .\]
When $I=\emptyset$, this set contains one element, the function with the empty domain.
Note that in order for a subset $I \subset V$ to be a feasible jump structure, it is necessary and sufficient for $I$ to satisfy \eqref{eq:torus-feasible-jump-structure1} and
\begin{equation}
\label{eq:torus-feasible-jump-structure2}
\cB^*(I) \neq \emptyset .
\end{equation}
This last condition is the manifestation of the topological constraint imposed by the torus. It says that the chain structure induced by the position of the jumps is such that it is possible to assign signs to each chain so that the average height ``returns'' to its initial value when completing an entire loop around the torus.

For a feasible jump structure $I$ and a feasible sign vector $\epsilon \in \cB^*(I)$, define the {\em signed chain structure} of $(I,\epsilon)$ by
\begin{equation}
\label{eq:torus-signed-chain-structure-def}
\cC^{*}(I,\epsilon) := \left\{ (k,\epsilon_{(k,t)} \cdot t) ~|~ (k,t) \in \cC(I) \right\} .
\end{equation}
Recall the definition of $S$ from \eqref{eq:S_torus_def}.
Define $X \in \{-1,1\}^{\cC(S)}$ by
\[ X_{(k,t)} := f(k) - f(k-1) , \quad (k,t) \in \cC(S) ,\]
and note that $X \in \cB^*(S)$. This defines for us the random
signed chain structure $\cC^*(S,X)$. This random variable contains
in a fairly simple manner all the necessary information for
determining the range of $f$. Namely, it gives us the positions,
lengths and signs of the chains in $f$.

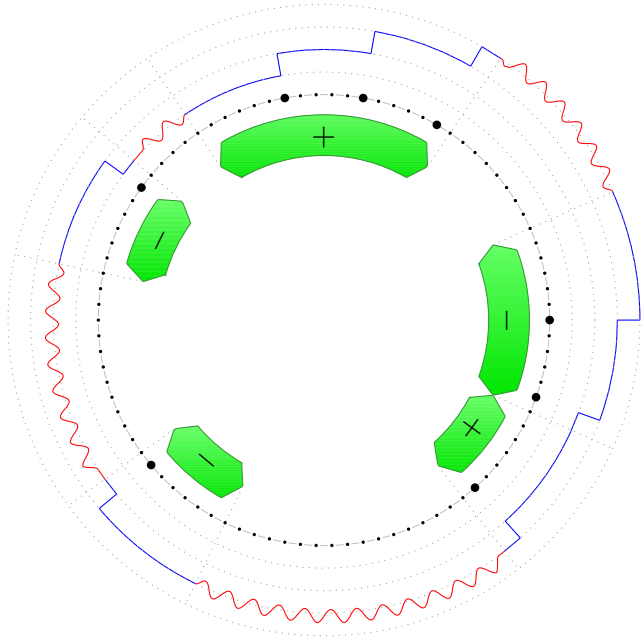
\begin{figure}[!t]
\centering
  \begin{tikzpicture}

    \def \d {5} //2d+1
    \def \n {90}
    \def \radius {3cm}
    \def \margin {1} 
    \def \tick {4}
    \def \r {3}
    \def \ra {0.73}
    \def \rb {0.91}
    \def \rc {0.82}

    \chain[\tick][\d][15][3][\r][\ra][\rb]
    \chain[\tick][\d][36][1][\r][\ra][\rb]
    \chain[\tick][\d][55][1][\r][\ra][\rb]
    \chain[\tick][\d][78][1][\r][\ra][\rb]
    \chain[\tick][\d][85][2][\r][\ra][\rb]

    \draw[dotted,thin,gray!75] ({\tick*14}:\radius*\rc) -- ({\tick*14}:{\radius*1.4});
    \draw[dotted,thin,gray!75] ({\tick*31}:\radius*\rc) -- ({\tick*31}:{\radius*1.4});
    \draw[dotted,thin,gray!75] ({\tick*35}:\radius*\rc) -- ({\tick*35}:{\radius*1.4});
    \draw[dotted,thin,gray!75] ({\tick*42}:\radius*\rc) -- ({\tick*42}:{\radius*1.4});
    \draw[dotted,thin,gray!75] ({\tick*54}:\radius*\rc) -- ({\tick*54}:{\radius*1.4});
    \draw[dotted,thin,gray!75] ({\tick*61}:\radius*\rc) -- ({\tick*61}:{\radius*1.4});
    \draw[dotted,thin,gray!75] ({\tick*77}:\radius*\rc) -- ({\tick*77}:{\radius*1.4});
    \draw[dotted,thin,gray!75] ({\tick*84}:\radius*\rc) -- ({\tick*84}:{\radius*1.4});
    \draw[dotted,thin,gray!75] ({\tick*6}:\radius*\rc) -- ({\tick*6}:{\radius*1.4});

    \fluc[\tick][31][4][\r][1.05][1.15]
    \fluc[\tick][42][12][\r][1.15][1.25]
    \fluc[\tick][61][16][\r][1.25][1.35]
    \fluc[\tick][6][8][\r][1.35][1.45]

    \jump[\tick][14][5][\r][1.4][1.3]
    \jump[\tick][19][5][\r][1.3][1.2]
    \jump[\tick][24][7][\r][1.2][1.1]
    \jump[\tick][35][7][\r][1.1][1.2]
    \jump[\tick][54][7][\r][1.2][1.3]
    \jump[\tick][77][7][\r][1.3][1.2]
    \jump[\tick][84][5][\r][1.2][1.3]
    \jump[\tick][89][7][\r][1.3][1.4]

    \tikzstyle{every node}=[draw, circle, fill,minimum size=0.03cm,inner sep=0]
    \foreach \s in {1,...,\n}
    {
      \node at ({360/\n * (\s - 1)}:\radius) {};
      \draw[-, >=latex,thin,gray,opacity=0.5] ({360/\n * (\s - 1)+\margin}:\radius)
        arc ({360/\n * (\s - 1)+\margin}:{360/\n * (\s)-\margin}:\radius);
    }

    \draw[dotted,thin,gray!75] (0:{\radius*1.1}) arc (0:360:{\radius*1.1});
    \draw[dotted,thin,gray!75] (0:{\radius*1.2}) arc (0:360:{\radius*1.2});
    \draw[dotted,thin,gray!75] (0:{\radius*1.3}) arc (0:360:{\radius*1.3});
    \draw[dotted,thin,gray!75] (0:{\radius*1.4}) arc (0:360:{\radius*1.4});

    \tikzstyle{every node}=[draw, circle, fill,minimum size=0.1cm,inner sep=0]
    \node[] at ({\tick * 15}:\radius) {};
    \node[] at ({\tick * 20}:\radius) {};
    \node[] at ({\tick * 25}:\radius) {};
    \node[] at ({\tick * 36}:\radius) {};
    \node[] at ({\tick * 55}:\radius) {};
    \node[] at ({\tick * 78}:\radius) {};
    \node[] at ({\tick * 85}:\radius) {};
    \node[] at ({\tick * 90}:\radius) {};

    \node[draw=none,fill=none,scale=1.1,rotate=\tick*22.5-90] at ({\tick * 22.5}:{\radius*0.81}) {$+$};
    \node[draw=none,fill=none,scale=1.1,rotate=\tick*38.5-90] at ({\tick * 38.5}:{\radius*0.81}) {$-$};
    \node[draw=none,fill=none,scale=1.1,rotate=\tick*57.5-90] at ({\tick * 57.5}:{\radius*0.81}) {$-$};
    \node[draw=none,fill=none,scale=1.1,rotate=\tick*81-90] at ({\tick * 81}:{\radius*0.81}) {$+$};
    \node[draw=none,fill=none,scale=1.1,rotate=\tick*0-90] at ({\tick * 0}:{\radius*0.81}) {$-$};

  \end{tikzpicture}

\caption{A homomorphism is broken up into sections of fluctuations and chains. The positions, lengths and associated signs of each chain (denoted by blocks with signs inside) make up the signed chain structure $\cC^*(S,X)$ of the homomorphism. This information, along with the independent fluctuation values between the chains (denoted by wavy lines), uniquely determines the homomorphism.}
\label{fig:torus-bijection}
\end{figure}

We now proceed to define $Y$. For a non-empty feasible jump structure $I$, define the {\em fluctuation points} of $I$ by
\[ FP(I) := \Big\{ y \in V ~|~ \rho^+(y,I) \in 2d+1 + 2\N \Big\} ,\]
where $\rho^+(y,I) := \min_{s \in I} \rho^+(y,s)$ and
$\N:=\{1,2,3,\ldots\}$. That is, a point is a fluctuation point if
its clockwise distance to the closest jump in the clockwise
direction is odd and at least $2d+3$. In particular, for any
homomorphism $f$ and any $y \in FP(S(f))$, $f$ is not at its average
height at $y$.
Now, for a homomorphism $f$ having at least one jump, define $Y \in \{-1,1\}^{FP(S)}$ by
\[ Y_y := f(y) - f(y-1) , \quad y \in FP(S) .\]

It will be useful to have the following formula for the number of
fluctuation points.

\begin{claim}
\label{cl:torus-bijection-fluc-size}
For any non-empty feasible jump structure $I$, we have
\[ |FP(I)| = n/2 - (d+1/2)|I| - |\cC(I)| . \]
\end{claim}
\begin{proof}
We have
\[ \left|\big\{ y \in V ~|~ \rho^+(y,I) < 2d+3 \big\}\right| = \sum_{(k,t) \in \cC(I) } ((2d+1)t + 2) = (2d+1)|I| + 2|\cC(I)| .\]
Furthermore, the set $\{ y \in V ~|~ \rho^+(y,I) \geq 2d+3 \}$ is a
disjoint union of intervals of even length, by
\eqref{eq:torus-feasible-jump-structure1}, so that
\begin{align*}
|FP(I)|
  &= \left|\big\{ y \in V ~|~ \rho^+(y,I) \geq 2d+3 \big\}\right| / 2 \\
  &= \left(n - \left|\big\{ y \in V ~|~ \rho^+(y,I) < 2d+3 \big\}\right|\right) / 2 \\
  &= n/2 - (d+1/2)|I| - |\cC(I)| . \qedhere
\end{align*}
\end{proof}

The final lemmas show that $X, Y$ and the jump structure $S$ exactly
encode the homomorphism.

\begin{lemma}
\label{lem:torus-bijection}
For any non-empty feasible jump structure $I$, the mapping $f \mapsto (X(f), Y(f))$ is a bijection between $\{S=I\}$ and $\cB^*(I) \times \{-1,1\}^{FP(I)}$. Also, the event $\{S=\emptyset\}$ is of size $2^{n/2+1}-2$.
\end{lemma}

This is an immediate consequence of the following lemma.

\begin{lemma}
\label{lem:torus-bijection2}
For any non-empty feasible jump structure $I$ and any feasible sign vector $\epsilon \in \cB^*(I)$, the mapping $f \mapsto Y(f)$ is a bijection between $\{ \cC^*(S,X)=\cC^*(I,\epsilon)\}$ and $\{-1,1\}^{FP(I)}$. Also, the event $\{\cC^*(S,X)=\emptyset\}$ is of size $2^{n/2+1}-2$.
\end{lemma}
\begin{proof}
It is not hard to verify that this is indeed a bijection (see Figure \ref{fig:torus-bijection} for a macroscopic picture and Figure \ref{fig:line-bijection} for a microscopic picture). We omit the proof as it is very similar to that of Lemma \ref{lem:line-bijection}.

For the second statement, we note that $\{\cC^*(S,X)=\emptyset\}=\{S=\emptyset\}=\{h \equiv const \}$, and hence by considering the events $\{h \equiv 0\}$ and $\{|h| \equiv 1\}$, and recalling that we set $h \equiv 0$ when $f$ takes on only two values, the statement readily follows.
\end{proof}

\subsection{The range}

In this section, our goal is to give a more explicit description of the distribution of the range of a homomorphism. Namely, that which is stated in Proposition \ref{prop:torus-range-distribution} below.

Recall the definition of the signed chain structure from
\eqref{eq:torus-signed-chain-structure-def}. Let
\begin{equation}
\label{eq:torus-def-bar-W}
\bar{W} := \left\{ w ~|~ (k,w) \in \cC^*(S,X) \right\} ,
\end{equation}
be the set of lengths and signs of the chains taken with multiplicities, i.e. $\bar{W}$ is a multi-set. For a vector of integers $w=(w_1,\dots,w_m)$, denote by
\begin{equation}
\label{eq:torus-def-range-of-vector}
|\Rng(w)| := 1 + \max_{0 \leq j \leq m} \sum_{i=1}^j w_i - \min_{0 \leq j \leq m} \sum_{i=1}^j w_i ,
\end{equation}
the size of the smallest interval in $\Z$ which contains all partial sums of $w$.

\begin{prop}
\label{prop:torus-range-distribution}
Let $m \geq 1$, let $\bar{w}=\{w_1,\dots,w_m\}$ be a multi-set such that $\Pr(\bar{W}=\bar{w})>0$ and let $\pi$ be a uniformly chosen permutation of $\{1,2,\dots,m\}$. Then,
\[ (|\Rng(f)| \text{ conditioned on } \bar{W} = \bar{w}) \eqd 2 + |\Rng(w_{\pi(1)},\dots,w_{\pi(m)})| .\]
\end{prop}

Proposition \ref{prop:torus-range-distribution} is a direct
consequence of the following two lemmas. The first of these, Lemma
\ref{lem:torus-range-equality}, relates the range of $f$ to a random
variable $W$ defined below. The second, Lemma
\ref{lem:torus-distribution-of-W}, describes the distribution of $W$
conditioned on $\bar{W}$.

\vspace{2pt}
Given a set $\mathcal{X}$ and a vector $x \in \mathcal{X}^m$, define the {\em period} of $x$ by
\[ \per(x) := \min \{ 1 \leq k \leq m ~|~ \sigma^k(x)=x \} ,\]
where $\sigma \colon \mathcal{X}^m \to \mathcal{X}^m$ is the mapping
$(x_1,\dots,x_m) \mapsto (x_2,\dots,x_m,x_1)$ and $\sigma^k$ is its
iteration $k$ times, so that $\sigma^m$ is the identity map. Define
an equivalence relation on $\mathcal{X}^m$ by $x \sim y$ if and only
if there exists a $k$ such that $\sigma^k(x)=y$. Denote by
$[x]:=\{x,\sigma(x), \ldots, \sigma^{m-1}(x)\}$ the equivalence
class of $x$. Observe that $|[x]|=\per(x)$. For $x,y \in
\mathcal{X}$, define
\[ x \vee y := \big((x_1,y_1),\dots,(x_m,y_m)\big) ,\]
and note that $\per(x \vee y) = \lcm(\per(x),\per(y))$.

Write $\cC^*(S,X) = \{(k_i,w_i)\}_{i=1}^{m}$, where $0 \leq k_1 < \cdots < k_m < n$. Let $k_0:=k_m$ and define
\[ W := \left[\left( w_1, \dots, w_m \right)\right] . \]
That is, $W$ forgets the absolute position of the chains and
remembers only their signed length and relative ordering. Note that
$\bar{W}$ is determined by $W$.

We begin by showing that the random variable $W$ governs the range
of the homomorphism. For a vector of integers $w$ whose sum is zero,
recalling \eqref{eq:torus-def-range-of-vector}, we define
$|\Rng([w])| := |\Rng(w)|$, and note that this is indeed
well-defined by the equivalence class of $w$.

\begin{lemma}
We have
\label{lem:torus-range-equality}
\[ |\Rng(f)| = 2 + |\Rng(W)| \quad \text{on the event } \{ \bar{W} \neq \emptyset \} .\]
\end{lemma}
\begin{proof}
The partial sums of $W$ correspond to differences in average height between two vertices. Therefore, $|\Rng(W)| = 1 + \max_{x \in V} h(x) - \min_{x \in V} h(x)$. By the definition of the average height, we have $|f(x)-h(x)| \leq 1$ and $\{h(x)-1,h(x),h(x)+1\} \subset \Rng(f)$ for any vertex $x \in V$. Therefore, by considering vertices at which the average height is maximal or minimal, we obtain the additional factor of $2$ in the above equation.
\end{proof}

\begin{remark}
On the event $\{ \bar{W} = \emptyset \}$, the size of the range of $f$ is either $2$ or $3$. However, Lemma \ref{lem:torus-bijection} implies that, conditioned on $\bar{W} = \emptyset$, the probability that the size of the range is $2$ is of order $2^{-n/2}$.
\end{remark}

The next lemma is the final ingredient in the proof of Proposition
\ref{prop:torus-range-distribution}. The remaining part of this
section is devoted to its proof.

\begin{lemma}
\label{lem:torus-distribution-of-W}
Let $m \geq 1$, let $\bar{w}=\{w_1,\dots,w_m\}$ be a multi-set such that $\Pr(\bar{W}=\bar{w})>0$ and let $\pi$ be a uniformly chosen permutation of $\{1,2,\dots,m\}$. Then,
\[ (W \text{conditioned on } \bar{W} = \bar{w}) \eqd [w_{\pi(1)}, \dots, w_{\pi(m)}] .\]
\end{lemma}

Write $\cC^*(S,X) = \{(k_i,w_i)\}_{i=1}^{m}$, where $0 \leq k_1 < \cdots < k_m < n$. Let $k_0:=k_m$ and define
\[ Z := \left[\big( \rho^+(k_{i-1}, k_i - (2d+1)|w_i| - 2)/2, w_i \big)_{i=1}^m \right] . \]
That is, $Z$ forgets the absolute positions of the chains in
$\cC^*(S,X)$ and remembers only their distances one to the other
(precisely, half the distance from the last vertex of one chain to one vertex before the beginning of the next chain). Note that the first coordinate of each element in $Z$ is necessarily a non-negative integer. Also note that $W$ is determined by $Z$. The next claim calculates the
distribution of $Z$.

\begin{claim}
\label{claim:torus-feasible-jump-structure-count} Let $m\ge 1$. Let
$w \in (\Z \setminus \{0\})^m$ be such that $w_1+\cdots+w_m=0$ and
let $x \in (\N\cup\{0\})^m$ be such that
\begin{equation}
\label{eq:torus-feasible-jump-structure-count}
(2d+1)(|w_1| + \cdots + |w_m|) + 2m + 2(x_1 + \cdots + x_m) = n .
\end{equation}
Then
\[ |\{ Z = [x \vee w] \}| = \frac{n \cdot \per(x \vee w)}{m} \cdot 2^{x_1 + \cdots + x_m} .\]
\end{claim}
\begin{proof}
Recall conditions \eqref{eq:torus-feasible-jump-structure1} and \eqref{eq:torus-feasible-jump-structure2}, and note that, together with the assumptions, they imply that the event $\{ Z=[x \vee w] \}$ is non-empty. We partition the event $\{ Z=[x \vee w] \}$ according to $\cC^*(S,X)$. Let $r \geq 1$ be the number of subsets in this partition, so that
\[ \{ Z=[x \vee w] \} = \bigcup_{i=1}^r \{ \cC^*(S,X)=\cC^*(I_i,\epsilon_i) \} ,\]
where the $(I_i,\epsilon_i)$ are distinct and feasible. By Lemma \ref{lem:torus-bijection2}, Claim \ref{cl:torus-bijection-fluc-size} and \eqref{eq:torus-feasible-jump-structure-count},
\[ |\{ \cC^*(S,X)=\cC^*(I_i,\epsilon_i) \}| = 2^{n/2 - (d+1/2)(|w_1|+\cdots+|w_m|) - m} = 2^{x_1 + \cdots + x_m} ,\]
for any $1 \leq i \leq r$, and therefore,
\[ |\{ Z=[x \vee w] \}| = r \cdot 2^{x_1 + \cdots + x_m} .\]
Recalling the definition of $Z$, we see that $\{ Z=[x \vee w] \}$
determines $\cC^*(S,X)=\{(k_i,w_i)\}_{i=1}^{m}$ up to a rotation of
the torus. It is not hard to see that if $l=\per(x \vee w)$ then
$r=nl/m$.
\end{proof}

\begin{proof}[Proof of Lemma \ref{lem:torus-distribution-of-W}]
We show something stronger. Define
\[ D := \left[\left( x_1, \dots, x_m \right)\right] \text{ where } Z = [((x_i,w_i))_{i=1}^{m}] .\]
Let $m \geq 1$, let $\bar{w}=\{w_1,\dots,w_m\}$ be a multi-set and let $x=(x_1,\dots,x_m)$ be such that $\Pr(\bar{W}=\bar{w}, D=[x])>0$. We shall show that
\[ (W \text{ conditioned on } \bar{W} = \bar{w} \text{ and } D = [x]) \eqd [w_{\pi(1)}, \dots, w_{\pi(m)}] ,\]
where $\pi$ is a uniformly chosen permutation of $\{1,2,\dots,m\}$.

Let $w$ be an ordering of $\bar{w}$. Define
\[ \mathcal{Z}(x,w) := \big\{ [x' \vee w'] ~|~ x' \in [x], ~ w' \in [w] \big\} , \]
and note that $|\mathcal{Z}(x,w)| = \gcd(\per(x),\per(w))$. We have
\[ \Pr(W=[w] ~|~ \bar{W} = \bar{w}, ~ D=[x]) = \frac{\Pr(W=[w], ~ D=[x])}{\Pr(\bar{W} = \bar{w}, ~ D=[x])} = \frac{\sum_{z \in \mathcal{Z}(x,w)} |\{Z=z\}|}{\big|\{ \bar{W}=\bar{w} \} \cap \{ D=[x] \}\big|} .\]
Let $w' \in [w]$ and $x' \in [x]$ be representatives of their equivalence classes. By Claim \ref{claim:torus-feasible-jump-structure-count}, we have
\[ |\{ Z=[x' \vee w'] \}| = \frac{n \cdot \per(x' \vee w')}{m} \cdot 2^{x_1 + \cdots + x_m} .\]
Since, $\per(x' \vee w') = \lcm(\per(x'),\per(w'))$, $\per(x')=\per(x)$ and $\per(w')=\per(w)$, we see that $\Pr(W=[w] ~|~ \bar{W} = \bar{w}, ~ D=[x])$ is proportional to
\[ |\mathcal{Z}(x,w)| \cdot \lcm(\per(x),\per(w)) = \per(x) \cdot \per(w) . \]
That is, conditioned on $\bar{W} = \bar{w}$ and $D=[x]$, the probability that $W$ equals $[w]$ is proportional to $\per(w)$. Finally, observe that the same is true for the probability that $[w_{\pi(1)}, \dots, w_{\pi(m)}]$ equals $[w]$. Indeed, one may check that
\[ \Pr([w_{\pi(1)}, \dots, w_{\pi(m)}] = [w]) = \frac{\per(w)}{C(\bar{w})} ,\]
where $C(\bar{w})$ is a multinomial coefficient depending on $\bar{w}$.
\end{proof}

\subsection{Proof of theorems}

In this section, we are primarily concerned with homomorphisms on the graph $T_{n,d}$. However, we will occasionally also refer to homomorphisms on the graph $P_{n,d}$. We note that in either case, such a homomorphism can be seen as an element of $\Z^{\{0,1,\dots,n\}}$, where $f \in \Hom(T_{n,d})$ is extended to $\{0,1,\dots,n\}$ by $f(n):=0$. Therefore, the uniform distributions on $\Hom(P_{n,d})$ and $\Hom(T_{n,d})$ can be seen as distributions on $\Z^{\{0,1,\dots,n\}}$. We shall denote the probability and expectation with respect to each of these distributions by $\Pr_P$ and $\E_P$ and $\Pr_T$ and $\E_T$, respectively. Throughout this section, we will frequently drop the subscript, in which case $\Pr$ and $\E$ will refer to $\Pr_T$ and $\E_T$.

We first state some technical lemmas and propositions whose proofs we defer to the next section. Our first proposition is one which will allow us to transfer some results from the line to the torus. This is an FKG-type inequality for the measure induced on non-negative homomorphisms by taking pointwise absolute value.

\begin{prop}
\label{prop:fkg-line-torus}
For any increasing function $\phi \colon \Z^{\{0,1,\dots,n\}} \to [0, \infty)$, we have
\[ \E_T[\phi(|f|)] \leq 9 \cdot \E_P[\phi(|f|)] .\]
\end{prop}

The next two lemmas are concerned with the probability of jumps occurring at given vertices. In the case of the line, we were able to obtain in Lemma \ref{lem:line-jumps-at-positions} a good upper bound on the probability of having $t$ jumps at any given vertices. In the case of the torus, we are not able to obtain such a general result. The main difficulty is due to the topological constraint imposed by the torus. In particular, if a jump occurs at a given vertex then a jump in the opposite direction must also occur at some other vertex. The next lemma shows that the probability of a chain of consecutive jumps is still unlikely.

\begin{lemma}
\label{lem:torus-chain-prob}
For any vertex $x \in V$ and any positive integer $t$, we have
\[ \Pr(C_{x,t}) \leq C 2^{-dt} .\]
\end{lemma}

The following lemma shows that having jumps in opposing directions
at given vertices is also unlikely.

\begin{lemma}
\label{lem:torus-jumps-at-multiple-positions}
For $I,J \subset V$ let $A_{I,J} := \bigcap_{x \in I, y \in J} A_{x,y}$ be the event that there are jumps in one direction at all vertices in $I$ and jumps in the opposite direction at all vertices in $J$. Then, for any subsets $I,J \subset V$ of size $m$ each, we have
\[ \Pr(A_{I,J}) \leq 2^{-(2d-1)m} .\]
\end{lemma}

The last lemma is the analog of Corollary \ref{cor:line-ratio-ineq-for-R} on the line. It will allow us to deduce the typical order of magnitude of $R$.

\begin{lemma}
\label{lem:torus-ratio-ineq-for-R} For any positive even integer $n$
and any positive integers $d$ and $r$ such that $Cdr \leq n$, we
have
\[ \Pr(R=r)\ge \frac{n^2}{r^2 2^{2d+5}} \cdot \left( 1 - \frac{Cdr}{n} \right)^2 \cdot \Pr(R=r-1).\]
\end{lemma}
As in the case of homomorphisms on the line, it is also possible to
prove an inequality in the opposite direction, showing that
$\Pr(R=r)\le C\frac{n^2}{r^2 2^{2d}} \Pr(R=r-1)$, but we neither use
nor prove this.

\subsubsection{The supercritical regime}
We prove Theorem \ref{thm:torus-supercritical} and Corollary \ref{cor:torus-supercritical}.
By Lemma \ref{lem:torus-jumps-at-multiple-positions} and by the union bound, we have
\[ \Pr(R \geq r) = \Pr\left(\bigcup_{\substack{I,J \subset V \\ |I|=|J|=r}} A_{I,J} \right) \leq \binom{n}{r}^2 2^{-(2d-1)r} , \quad r \geq 1 .\]
One may easily check that $|\Rng(f)| \leq R+3$, so that
\begin{equation}
\label{eq:torus-supercritical-bound}
\Pr\big(|\Rng(f)| \geq 3 + r\big) \leq \binom{n}{r}^2 2^{-(2d-1)r} , \quad r \geq 1 .
\end{equation}
Moreover, it is easy to describe all homomorphisms which take on at most $3$ values.
Let $\Omega_0$ be the set of homomorphisms which are constant on the even vertices (having the value $0$ on the even vertices and $1$ or $-1$ on the odd vertices), and let $\Omega_1$ be the set of homomorphisms which are constant on the odd vertices (having the value $\pm 1$ on the odd vertices, and $0$ or $\pm 2$, respectively, on the even vertices).
Then $\{ |\Rng(f)| \leq 3 \} = \Omega_0 \cup \Omega_1$, $|\Omega_0 \cap \Omega_1| = 2$, and $|\Omega_0| = |\Omega_1| = 2^{n/2}$. Therefore,
\[ \Pr\big(|\Rng(f)| < 3\big) = \Pr(\Omega_0 \cap \Omega_1) \leq \frac{|\Omega_0 \cap \Omega_1|}{|\Omega_0 \cup \Omega_1|} = \frac{2}{2^{n/2+1} - 2}  \leq 2^{1-n/2} ,\]
completing the proof of Theorem \ref{thm:torus-supercritical}. To
obtain Corollary \ref{cor:torus-supercritical}, recall that
$|\Omega_0| = |\Omega_1|$ and note that if $d - \log_2 n \to \infty$
as $n \to \infty$ then $\Pr(\Omega_0 \cup \Omega_1) = \Pr(|\Rng(f)|
\leq 3) = 1 - o(1)$, by Theorem \ref{thm:torus-supercritical}.

We remark that the bound \eqref{eq:torus-supercritical-bound}
obtained for the probability that the range is large constitutes
something of a compromise between two possibilities. With somewhat
less work we could have used the FKG-type inequality, Proposition
\ref{prop:fkg-line-torus}, to obtain a weaker bound. With somewhat
more work we could make a finer analysis of the possible cases in
the proof of Lemma \ref{lem:torus-jumps-at-multiple-positions} and
obtain a somewhat better bound, with $2d-1$ replaced by $2d+1$ or
even $2d+2$. The bound we chose to prove has the benefit that it is
already rather good and has a relatively simple proof.

\subsubsection{The subcritical regime}
We begin by proving the upper bound in Theorem \ref{thm:torus-subcritical-range}.
Since $\max_{0 \leq k \leq n} |f(k)|$ is an increasing function in $|f|$, we have by Proposition \ref{prop:fkg-line-torus},
\[ \E_T\left[\max_{0 \leq k \leq n} |f(k)| \right] \leq 9 \cdot \E_P\left[\max_{0 \leq k \leq n} |f(k)|\right] .\]
By our previous result on the line, Theorem \ref{thm:line-subcritical-range}, we have
\[ \E_P\left[\max_{0 \leq k \leq n} |f(k)|\right] \leq \E_P\big[|\Rng(f)| - 1\big] \leq 2 + C \sqrt{n 2^{-d}} ,\]
and then, using symmetry,
\[ \E_T\big[|\Rng(f)|\big] = \E_T\left[1 + \max_{0 \leq k \leq n} f(k) - \min_{0 \leq k \leq n} f(k)\right] \leq 1 + 2 \cdot\E_T\left[\max_{0 \leq k \leq n} |f(k)|\right] \leq 37 + C \sqrt{n 2^{-d}} .\]

We now prove the lower bound in Theorem \ref{thm:torus-subcritical-range}. The proof is very similar to the proof of the lower bound in Theorem \ref{thm:line-subcritical-range} in Section \ref{sec:line}, and so we only give an outline of the proof. First, we show that for any $\epsilon>0$ there exists a $\delta>0$ such that
\[ \Pr\left(|\Rng(f)| < 3 + \big\lfloor \delta \sqrt{n 2^{-d}} \big\rfloor \right) \leq \epsilon + 2^{1-n/2} .\]
Let $\epsilon > 0$. Note that, by Theorem \ref{thm:torus-supercritical}, the statement is trivial when $\delta \sqrt{n 2^{-d}} < 1$. Hence, we may assume that $n2^{-d} \geq C/\epsilon$.
Mimicking the proof of Lemma \ref{lem:line-R-large-whp} and its corollary, using Lemma \ref{lem:torus-ratio-ineq-for-R} in place of Lemma \ref{lem:line-ratio-ineq-for-R-cond}, we find that there exists a $\delta_1>0$ such that
\[ \Pr \left( R < \delta_1 n 2^{-d} \right) \leq \epsilon/4 .\]
Continuing as in \eqref{eq:line-M-R-C-relation} - \eqref{eq:line-many-chains-whp}, using Lemma \ref{lem:torus-chain-prob} in place of Lemma \ref{lem:line-jumps-at-positions}, we obtain
\begin{equation}
\label{eq:torus-many-chains-whp}
\Pr\left(|\cC(S)| \geq \delta_2 n 2^{-d}\right) \geq 1 - \epsilon/2 ,
\end{equation}
for some $\delta_2 > 0$.
Proposition \ref{prop:torus-range-distribution} and Proposition \ref{prop:sampling-without-replacement2} imply that
\begin{equation}
\label{eq:torus-range-prob-by-number-of-chains}
\Pr\big(|\Rng(f)| < r ~|~ |\cC(S)| \big) \leq \frac{C r}{\sqrt{|\cC(S)|}} , \quad r \in \N .
\end{equation}
Now, putting \eqref{eq:torus-many-chains-whp} and \eqref{eq:torus-range-prob-by-number-of-chains} together, we see that there exists a $\delta>0$ such that
\[ \Pr\left(|\Rng(f)| < 3 + \big\lfloor \delta \sqrt{n2^{-d}} \big\rfloor \right) \leq \Pr\left(|\Rng(f)| < \big\lfloor 4 \delta \sqrt{n2^{-d}} \big\rfloor \right) \leq \epsilon . \]
Finally, repeating the calculation in \eqref{eq:line-subcritical-range-expectation-lower-bound}, where we use Theorem \ref{thm:torus-supercritical} in place of Theorem \ref{thm:line-supercritical}, we obtain
\[ \E\big[|\Rng(f)|\big] \geq 3 + \big\lfloor c \sqrt{n 2^{-d}} \big\rfloor - 2^{1-n/2} . \qedhere \]

\subsubsection{The critical regime}
Here we prove Theorem \ref{thm:torus-critical-range}. Denote $\lambda := \lim n 2^{-d}$ which exists and is a positive number by assumption. We begin by showing that in the critical regime the jumps are sparse enough so that it is unlikely to have chains of length greater than one. Let
\[ B := \bigcap_{x \in V} (A^c_x \cup A^c_{x+2d+1}) \]
be the event that there are no two minimal-distance jumps (i.e. jumps at distance $2d+1$). We wish to show that $\Pr(B) = 1 - o(1)$. Indeed, by Lemma \ref{lem:torus-chain-prob}, we have
\[ \Pr(B^c) \leq \sum_{x \in V} \Pr(A_x \cap A_{x+2d+1}) \leq Cn 2^{-2d} = o(1) .\]

We now find the limiting distribution of $R$ as $n$ tends to infinity. By Lemma \ref{lem:torus-bijection} and Claim \ref{cl:torus-bijection-fluc-size}, we have that $\Pr(R=r ~|~ B)$ is proportional to
\begin{equation}
\label{eq:torus-critical-prob-R-prop}
\begin{cases} c(n,d,r) \cdot \binom{2r}{r} \cdot 2^{n/2 - r(2d+3)} &\text{if } r \geq 1 \\ 2^{n/2+1} - 2 &\text{if } r=0 \end{cases} ,
\end{equation}
where $c(n,d,r)$ is the number of feasible jump structures $I$ having $|I|=|\cC(I)|=2r$.

\begin{claim}
\label{cl:torus-critical-feasible-jump-structure-count}
For any $r \geq 1$, we have
\[ c(n,d,r) = \frac{n}{2r} \cdot \binom{n/2 - (2d+1)r - 1}{2r-1} .\]
\end{claim}
\begin{proof}
Denote by $\cI$ the set of all feasible jump structures $I$ having
$|I|=|\cC(I)|=2r$. For $v \in V$, let $\cI_v := \{ I \in \cI ~|~ v
\in I \}$. Then,
\[ n |\cI_0| = \sum_{v \in V} |\cI_v| = \sum_{v \in V} \sum_{I \in \cI} \mathbf{1}_{I}(v) = \sum_{I \in \cI} \sum_{v \in V} \mathbf{1}_{I}(v) = \sum_{I \in \cI} |I| = 2r |\cI| = 2r \cdot c(n,d,r) .\]
It remains to compute the size of $\cI_0$. By considering the distances between consecutive elements in any $I \in \cI_0$, and recalling \eqref{eq:torus-feasible-jump-structure1}, \eqref{eq:torus-chain-structure} and \eqref{eq:torus-feasible-jump-structure2}, we see that $|\cI_0|$ is given by the number of non-negative integer solutions to the equation
\[ x_1 + x_2 + \cdots + x_{2r} = n , \]
under the additional constraint that, for $1 \leq j \leq 2r$, $x_j$
is odd and at least $2d+3$. Therefore, after substituting $x_j=2y_j
+ 2d+3$ for $1 \leq j \leq 2r$, we see that $|\cI_0|$ is equal to
the number of non-negative integer solutions to the equation
\[ y_1 + \cdots y_{2r} = n/2 - (2d+3)r ,\]
from which the result now follows.
\end{proof}

By \eqref{eq:torus-critical-prob-R-prop} and Claim \ref{cl:torus-critical-feasible-jump-structure-count}, for any fixed $r \geq 1$, we have
\[ \frac{\Pr(R=r ~|~ B)}{\Pr(R=r-1 ~|~ B)} = \frac{n^2}{r^2 2^{2d+5}} \cdot (1 + o(1)) = \frac{(\lambda')^2}{r^2} \cdot (1 + o(1)) ,\]
where $\lambda' := \lambda / (4 \sqrt{2})$. By Lemma \ref{lem:torus-chain-prob} and since $\Pr(B)=1-o(1)$, we have
\[ \E[R ~|~ B] = \frac{1}{2} \sum_{x \in V} \Pr(A_x ~|~ B) \leq \frac{1}{2\Pr(B)} \sum_{x \in V} \Pr(A_x) \leq \frac{C n 2^{-d}}{\Pr(B)} = O(1) .\]
Therefore, conditioned on $B$, the expectation of $R$ is uniformly
bounded as $n \to \infty$. Hence, Markov's inequality implies that,
conditioned on $B$, $R$ is tight as $n \to \infty$. Recall the
definition of the distribution $\nu(\lambda')$ in
\eqref{eq:equality-biased-poisson}. Let $N(\lambda') \sim
\nu(\lambda')$ and note that
\[ \frac{\Pr(N(\lambda') = r)}{\Pr(N(\lambda') = r-1)} = \frac{(\lambda')^2}{r^2} , \quad r \geq 1 .\]
Thus, Lemma \ref{lem:distribution-ratio-convergence} implies that, conditioned on $B$, $R$ converges in distribution to $N(\lambda')$. Finally, since $\Pr(B)=1-o(1)$, we conclude that $R$ converges in distribution to $\nu(\lambda')$.

It remains to understand the range of a homomorphism. Recalling the definition of $\bar{W}$ given in \eqref{eq:torus-def-bar-W}, we observe that the event $B$ is the same as the event $\{ |\bar{W}|=|\cC(S)|=2R \}$, which is the same as the event that $\bar{W}$ consists of $R$ $1$'s and $R$ $(-1)$'s. Therefore, by Proposition \ref{prop:torus-range-distribution}, conditioned on $B$ and on $R$, on the event $\{R>0\}$, the range of a homomorphism is equal in distribution to two plus the range of a random walk bridge of length $2R$. By Theorem \ref{thm:torus-supercritical}, conditioned on the event $\{R=0\}$, the range of a homomorphism is $3$ with probability tending to one. This, together with our previous result on the convergence of $R$ in distribution, completes the proof of Theorem \ref{thm:torus-critical-range}.

\subsection{Proof of main lemmas}

\begin{proof}[Proof of Proposition \ref{prop:fkg-line-torus}]
Recall that a homomorphism on $P_{n,d}$ or $T_{n,d}$ can be seen as an element of $\Z^{\{0,1,\dots,n\}}$, where $f \in \Hom(T_{n,d})$ is extended to $\{0,1,\dots,n\}$ by $f(n):=0$. Therefore, the uniform distributions on $\Hom(P_{n,d})$ and $\Hom(T_{n,d})$ are distributions on $\Z^{\{0,1,\dots,n\}}$, which we denote by $\Pr_P$ and $\Pr_T$ respectively. We also denote by $Q := \{ f \in \Hom(T_{n,d}) \}$ the support of $\Pr_T$, so that the measure $\Pr_P(\cdot ~|~ Q)$ is just the measure $\Pr_T$.

Let $\phi \colon \Z^{\{0,1,\dots,n\}} \to [0, \infty)$ be an increasing function.
Note that the event $\{f(n)=0\}$ is a decreasing event in $|f|$. Therefore, we can apply Theorem \ref{thm:fkg} for the functions $\phi$ and $\psi(f) := \mathbf{1}_{\{f(n)=0\}}$, to obtain
\begin{equation}
\label{eq:fkg-condition-on-f-n-0}
\E_P\big[\phi(|f|) ~|~ f(n)=0\big] \leq \E_P\big[\phi(|f|)\big] .
\end{equation}
Notice that sampling a random homomorphism on $P_{n,d}$ conditioned on $\{f(n)=0\}$ is not equivalent to sampling a random homomorphism on $T_{n,d}$, which is just to say that $Q \neq \{ f(n)=0 \}$. However, it is equivalent to sampling a random homomorphism on another graph. Namely, the graph $P'_{n,d}$ obtained from $P_{n,d}$ by identifying the vertex $0$ with the vertex $n$. In order to obtain $T_{n,d}$ from $P'_{n,d}$, we must still add some edges which are missing, for example, the edge between $n-1$ and $2$, and the edge between $n-2$ and $1$. Nonetheless, this observation shows that the measure $\Pr_P(\cdot ~|~ f(n)=0)$ also satisfies the FKG inequality in Theorem \ref{thm:fkg}, since it is equivalent to sampling a random homomorphism on $P'_{n,d}$. Define the events
\begin{align*}
J &:= \{ |f(k)| \leq 1, ~ k=0,1,\dots,2d \} \quad \text{and} \\
J' &:= \{ |f(k)| \leq 1, ~ k=n,n-1,\dots,n-2d \} .
\end{align*}
Notice that
\[ J \cap J' \subset Q \subset \{ f(n) = 0 \} .\]
Thus, using \eqref{eq:fkg-condition-on-f-n-0} and the fact that $\phi$ is non-negative, we obtain
\[ \E_P\big[\phi(|f|) ~|~ Q\big] \leq \E_P\big[\phi(|f|) ~|~ f(n)=0\big] \cdot \frac{\Pr_P(f(n)=0)}{\Pr_P(Q)} \leq \frac{\E_P\big[\phi(|f|)\big]}{\Pr_P(J \cap J' ~|~ f(n)=0)} .\]
We now wish to bound $\Pr(J \cap J' ~|~ f(n)=0)$ from below. We first apply Theorem \ref{thm:fkg} to the graph $P'_{n,d}$ to get
\[ \Pr_P(J \cap J' ~|~ f(n)=0) \geq \Pr_P(J ~|~ f(n)=0) \cdot \Pr_P(J' ~|~ f(n)=0) = \Pr_P(J ~|~ f(n)=0)^2 ,\]
where we have used symmetry in the second step. Next, we apply Theorem \ref{thm:fkg} again to the graph $P_{n,d}$ to get
\[ \Pr_P(J ~|~ f(n)=0) \geq \Pr_P(J) .\]
Finally, since $J$ is just the event that no jump occurs at the first $2d+1$ vertices, we have by Lemma \ref{lem:line-jump-at-start} that $\Pr_P(J) \geq 1/3$. Therefore, we have shown that
\[ \E_T\big[\phi(|f|)\big] = \E_P\big[\phi(|f|) ~|~ Q\big] \leq 9 \cdot \E_P\big[\phi(|f|)\big] . \qedhere \]
\end{proof}

For the proofs of the remaining lemmas, it is convenient to denote by $[x,y]$ the vertices on the arc going from $x$ to $y$ in
the clockwise direction. That is, for $x,y \in V$, we define
\[ [x,y] := \{ z \in V ~|~ \rho^+(x,z) \leq \rho^+(x,y) \} . \]
Also, for a set $J\subset V$ and an integer $t$, we let
\begin{equation*}
J+t:=\{j+t~|~ j\in J\}
\end{equation*}
where, as always, addition for vertices on the torus is taken modulo
$n$.

\begin{proof}[Proof of Lemma \ref{lem:torus-chain-prob}]
First, we partition $C_{x,t}$ into two events $C^0_{x,t}$ and
$C^1_{x,t}$. Denote $x' := x-(2d+1)t-1$ and define
\begin{align*}
C^0_{x,t} &:= C_{x,t} \cap \big\{ | f(x) - f(x'-1) | = 2+t \big\} , \\
C^1_{x,t} &:= C_{x,t} \cap \big\{ | f(x) - f(x'-1) | = t \big\} .
\end{align*}
Note that since any $f \in C_{x,t}$ has $|f(x)-f(x')| = 1+t$, we indeed have $C_{x,t} = C^0_{x,t} \cup C^1_{x,t}$.
Also, observe that
\begin{equation}
\label{eq:torus-chain-prob-minimal-chain-segment}
C^0_{x,t} = \{ | f(x) - f(x'-1) | \geq 2+t \} ,
\end{equation}
since for any $f \in \Hom(T_{n,d})$, $|f(x)-f(x'-1)| \leq 2+t$ necessarily holds and $| f(x) - f(x'-1) | = 2+t$ holds only if there is a chain of length $t$ at $x$.

We now prove that
\begin{equation}\label{eq:C_1_C_0_inequality}
  \Pr(C^1_{x,t}) \leq \Pr(C^0_{x,t}).
\end{equation}
By rotating the torus if necessary, we note that it suffices to
prove this under the assumption that $x'\neq 1$. This simplifies
slightly the following discussion as it avoids issues stemming from
the fact that $f(0)$ is normalized to be $0$.

Consider the mapping $f \mapsto f_0$ from $C^1_{x,t}$ to $C^0_{x,t}$
defined by
\[ f_0(y) :=
    \begin{cases}
        f(y)                & \text{if } y \neq x'-1 \\
        2f(y+1) - f(y)      & \text{if } y = x'-1
    \end{cases} , \quad y \in V .
\]
Let $f \in C^1_{x,t}$. Note that, by the definition of the jumps and
the average height, there exists a vertex $x'' \in V$ such that
$|f(x)-f(x'')|=2+t$ and such that $1 \leq |f(y)-f(x'')| \leq 2$ for
any $y \in [x''+1,x'+2d+1]$. It is easy to see that the existence of such a $x''$ implies that $f_0$ is well-defined (see Figure
\ref{fig:torus-chain-prob}). Since the mapping is clearly injective,
\eqref{eq:C_1_C_0_inequality} follows, and so $\Pr(C_{x,t}) \leq 2\Pr(C^0_{x,t})$.

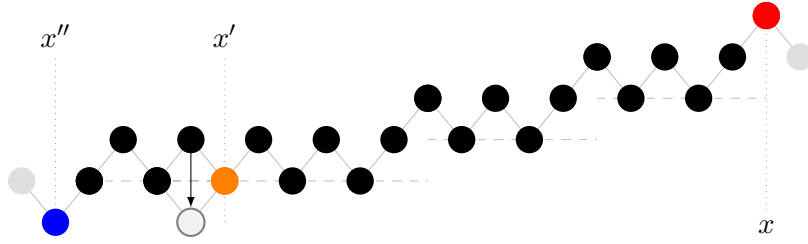
\begin{figure}[!t]
\centering
\begin{tikzpicture}
    \pgftransformcm{0.45}{0}{0}{0.55}{\pgfpoint{0cm}{0cm}}

    \draw[line width=0.1mm,gray,dotted](2,-1)--(2,3);
    \draw[line width=0.1mm,gray,dotted](7,-1)--(7,3);
    \draw[line width=0.1mm,gray,dotted](23,-0.7)--(23,4);
    \node at (2,3) [above] {$x''$};
    \node at (7,3) [above] {$x'$};
    \node at (23,-0.7) [below] {$x$};

    \drawRandomWalk{1}{{0,-1,0}}{99}{gray!25}{}{blue};
    \drawRandomWalk{3}{{0,1,0}}{0}{}{}{};
    \drawRandomWalk{5}{{0,1,0}}{0}{}{orange}{};
    \drawRandomWalk{5}{{0,-1,0}}{0}{}{orange}{};
    \node (A) at (6,-1) [draw=black!50,circle,fill=gray!10,thick] {};
    \path[->, >=latex] (6,1) edge (A);
    \drawRandomWalk{7}{{0,1,0,1,0,1,2}}{0}{orange}{}{};
    \drawRandomWalk{13}{{2,1,2,1,2,3}}{1}{}{}{};
    \drawRandomWalk{18}{{3,2,3,2,3,4}}{2}{}{red}{};
    \drawRandomWalk{23}{{4,3}}{99}{red}{gray!25}{};

\end{tikzpicture}
\caption{A homomorphism $f$ in $C^1_{x,t}$. Modifying the value at $x'-1$ to be $f(x'')$ injectively maps this homomorphism to $C^0_{x,t}$. Here $d=2$ and $t=3$.}
\label{fig:torus-chain-prob}
\end{figure}

It remains to bound the probability of the event $C^0_{x,t}$. Due to
rotation equivariance, the probability of this event is independent
of $x$. Thus, substituting $x=(2d+1)t+2$ so that $x'=1$, and
recalling \eqref{eq:torus-chain-prob-minimal-chain-segment}, we have
\[ \Pr(C^0_{x,t}) = \Pr(C^0_{(2d+1)t+2,t}) = \Pr(|f((2d+1)t+2)| \geq 2+t) . \]
Since this last event is clearly an increasing event in $|f|$, Proposition \ref{prop:fkg-line-torus} and Lemma \ref{lem:line-jumps-at-positions} now yield
\begin{align*}
\Pr_T(C^0_{x,t}) &= \Pr_T\left(|f((2d+1)t+2))| \geq 2 + t \right) \\
&\leq 9 \cdot \Pr_P\left(|f((2d+1)t+2))| \geq 2 + t\right) = 9 \cdot \Pr_P(C_{(2d+1)t+2,t+1}) \leq C 2^{-dt} .
\qedhere
\end{align*}

\end{proof}

\begin{proof}[Proof of Lemma \ref{lem:torus-jumps-at-multiple-positions}]
The idea of the proof is to remove jumps from the jump structure of
the given homomorphism and observe that this results in more
fluctuation points. We shall do so by removing the jumps two at a time. See Figure~\ref{fig:torus-T-x-y-mapping}.

We begin with some notation. For a feasible jump structure $I$ and a
vertex $x \in I$, denote by $C(I,x)$ the chain in $I$ containing
$x$, i.e., $C(I,x)$ is the unique element $(k,t) \in \cC(I)$
satisfying $x \in [k-(2d+1)(t-1),k]$. For a feasible jump structure
$I$ and two vertices $x,y \in I$ belonging to different chains, i.e.
$C(I,x) \neq C(I,y)$, denote
\[ I^{x,y} := (I \cap [y+1,x-1]) \cup ((I \cap [x+1,y-1]) - 1) .\]
Note that $I^{x,y}$ satisfies condition
\eqref{eq:torus-feasible-jump-structure1}. Moreover, it is easy to
see that the chain structure of $I^{x,y}$ satisfies
\[ \cC(I^{x,y}) = \cC(I \cap [y+1,x-1]) \cup \cC((I \cap [x+1,y-1])-1) . \]
In particular,
\begin{equation}\label{eq:chain_structure_I_x_y}
  |\cC(I)| - 2 \leq |\cC(I^{x,y})| \leq |\cC(I)| + 2
\end{equation}
and, denoting $D_{x,y} := \{ x-2d-1, x+2d+1, y-2d-1, y+2d+1 \}$,
\begin{equation}\label{eq:chain_structure_I_x_y_2}
|\cC(I^{x,y})| = |\cC(I)| - 2,\quad\text{if } I \cap D_{x,y} =
\emptyset.
\end{equation}
Now, for a feasible sign vector $\epsilon \in \cB^*(I)$, define $\epsilon^{x,y} \in \{-1,1\}^{\cC(I^{x,y})}$ by
\[ \epsilon^{x,y}(k,t) := \begin{cases}
 \epsilon(C(I,k)) &\text{if } k \in I \\
 \epsilon(C(I,k+1)) &\text{if } k+1 \in I
 \end{cases}, \quad (k,t) \in \cC(I^{x,y}).\]
That is, the sign of a chain in $\cC(I^{x,y})$ is inherited from its
corresponding chain in $\cC(I)$. Note that,
\[ \sum_{(k,t) \in \cC(I^{x,y})} \epsilon^{x,y}(k,t) \cdot t = - \epsilon(C(I,x)) - \epsilon(C(I,y)) .\]
Hence, if $\epsilon(C(I,x)) \neq \epsilon(C(I,y))$ then $\epsilon^{x,y} \in \cB^*(I^{x,y})$ and, by \eqref{eq:torus-feasible-jump-structure2}, $I^{x,y}$ is a
feasible jump structure.

\begin{figure}[!t]
\centering
\begin{tikzpicture}
    \pgftransformcm{0.4}{0}{0}{0.5}{\pgfpoint{0cm}{0cm}}

    \draw[line width=0.1mm,gray,dotted](4,-2)--(4,2);
    \draw[line width=0.1mm,gray,dotted](12,-2)--(12,2);
    \draw[line width=0.1mm,gray,dotted](19,-2)--(19,2);
    \draw[line width=0.1mm,gray,dotted](28,-2)--(28,2);
    \draw[line width=0.1mm,gray,dotted](15,1)--(18,1);

    \drawRandomWalk{0}{{0,1,0}}{99}{gray!25}{gray!50}{gray!37};
    \drawRandomWalk{2}{{0,-1,0}}{99}{gray!50}{orange}{gray!75};
    \drawRandomWalk{4}{{0,1,0,1,0,1,0,1,2}}{0}{orange}{red}{};
    \drawRandomWalk{12}{{2,1,0}}{1}{red}{gray!50}{};
    \drawRandomWalk{14}{{0,1}}{1}{gray!50}{gray!25}{};
    \drawRandomWalk{18}{{1,2}}{1}{gray!25}{gray!50}{};
    \drawRandomWalk{19}{{2,1,0}}{1}{gray!50}{}{orange};
    \drawRandomWalk{21}{{0,1,0,1,0,1,0,-1}}{1}{}{red}{};
    \drawRandomWalk{28}{{-1,0}}{0}{red}{}{};
    \drawRandomWalk{29}{{0,1,0}}{0}{}{gray!50}{gray!75};
    \drawRandomWalk{31}{{0,-1,0}}{0}{gray!50}{gray!12}{gray!25};
\end{tikzpicture}
\\ [-0.7ex]
\begin{tikzpicture}
    \pgftransformcm{0.4}{0}{0}{0.5}{\pgfpoint{0cm}{0cm}}

    \draw[line width=0.1mm,gray,dotted](4,-2)--(4,2);
    \draw[line width=0.1mm,gray,dotted](12,-2)--(12,2);
    \draw[line width=0.1mm,gray,dotted](19,-2)--(19,2);
    \draw[line width=0.1mm,gray,dotted](28,-2)--(28,2);
    \draw[line width=0.1mm,gray,dotted](14,0)--(17,0);
    \node at (12,-2) [below] {$x$};
    \node at (28,-2) [below] {$y$};

    \drawRandomWalk{0}{{0,1,0}}{99}{gray!12}{gray!50}{gray!25};
    \drawRandomWalk{2}{{0,-1,0}}{99}{gray!50}{}{gray!75};

    \drawRandomWalk{4}{{0,1,0,1,0,1,0,1,0}}{0}{}{}{white!0};
    \drawRandomWalk{4}{{0,-1,0,-1,0,-1,0,-1,0}}{0}{}{}{white!0};
    \node at (5,1) [draw=black!50,circle,fill=gray!10,thick] {};
    \node at (5,-1) [draw=black!50,circle,fill=gray!10,thick] {};
    \node at (6,0) [draw=black,circle,fill=black] {};
    \node at (7,1) [draw=black!50,circle,fill=gray!10,thick] {};
    \node at (7,-1) [draw=black!50,circle,fill=gray!10,thick] {};
    \node at (8,0) [draw=black,circle,fill=black] {};
    \node at (9,1) [draw=black!50,circle,fill=gray!10,thick] {};
    \node at (9,-1) [draw=black!50,circle,fill=gray!10,thick] {};
    \node at (10,0) [draw=black,circle,fill=black] {};
    \node at (11,1) [draw=black!50,circle,fill=gray!10,thick] {};
    \node at (11,-1) [draw=black!50,circle,fill=gray!10,thick] {};

    \drawRandomWalk{12}{{0,-1}}{0}{}{gray!50}{};
    \drawRandomWalk{13}{{-1,0}}{0}{gray!50}{gray!25}{};
    \drawRandomWalk{17}{{0,1}}{0}{gray!25}{gray!50}{};
    \drawRandomWalk{18}{{1,0}}{0}{gray!50}{}{};

    \drawRandomWalk{19}{{0,1,0,1,0,1,0,1,0,1,0}}{0}{}{}{white!0};
    \drawRandomWalk{19}{{0,-1,0,-1,0,-1,0,-1,0,-1,0}}{0}{}{}{white!0};
    \node at (20,1) [draw=black!50,circle,fill=gray!10,thick] {};
    \node at (20,-1) [draw=black!50,circle,fill=gray!10,thick] {};
    \node at (21,0) [draw=black,circle,fill=black] {};
    \node at (22,1) [draw=black!50,circle,fill=gray!10,thick] {};
    \node at (22,-1) [draw=black!50,circle,fill=gray!10,thick] {};
    \node at (23,0) [draw=black,circle,fill=black] {};
    \node at (24,1) [draw=black!50,circle,fill=gray!10,thick] {};
    \node at (24,-1) [draw=black!50,circle,fill=gray!10,thick] {};
    \node at (25,0) [draw=black,circle,fill=black] {};
    \node at (26,1) [draw=black!50,circle,fill=gray!10,thick] {};
    \node at (26,-1) [draw=black!50,circle,fill=gray!10,thick] {};
    \node at (27,0) [draw=black,circle,fill=black] {};
    \node at (28,1) [draw=black!50,circle,fill=gray!10,thick] {};
    \node at (28,-1) [draw=black!50,circle,fill=gray!10,thick] {};

    \drawRandomWalk{29}{{0,1,0}}{0}{}{gray!50}{gray!75};
    \drawRandomWalk{31}{{0,-1,0}}{0}{gray!50}{gray!12}{gray!25};
\end{tikzpicture}
\caption{An illustration of the operation of ``removing jumps'' from
a homomorphism. Given a homomorphism having jumps at $x$ and $y$ in
opposite directions, we may remove these jumps and gain entropy in
the newly formed fluctuation points. Here $d=3$.}
\label{fig:torus-T-x-y-mapping}
\end{figure}
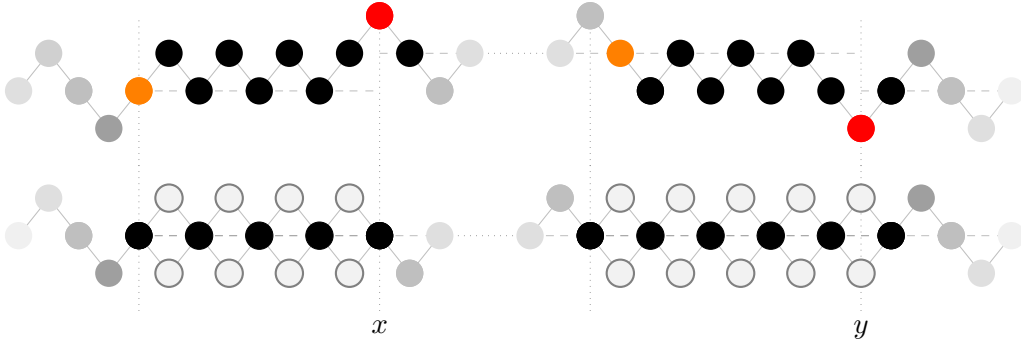

Denote by $\cI$ the set of all feasible jump structures.
For subsets $J,J' \subset V$, denote by $B(J,J')$ the set of all feasible signed jump structures containing $J \cup J'$ and having different signs on $J$ and $J'$, i.e.,
\[ B(J,J') := \big\{ (I,\epsilon) ~|~ I \in \cI, ~ \epsilon \in \cB^*(I), ~ J \cup J' \subset I, ~ \epsilon(C(I,x)) \neq \epsilon(C(I,y)) ~ \forall x \in J, y \in J' \big\}. \]
Then, we have the equality of events,
\[ A_{J,J'} = \{ (S,X) \in B(J,J') \} .\]
Therefore, the lemma is equivalent to
\begin{equation}\label{eq:removing_jumps_induction}
\Pr\big((S,X) \in B(J,J')\big) \leq 2^{-(2d-1)m} , \quad m \geq 0, ~
J,J' \subset V, ~ |J|=|J'|=m .
\end{equation} We prove this by
induction on $m$. The induction base, $m=0$, and the case when
$B(J,J')=\emptyset$ are trivial. Suppose that $m \geq 1$ and let
$J,J' \subset V$ be such that $B(J,J')\neq\emptyset$. We choose two
vertices $x \in J$ and $y \in J'$ such that $[x,y] \cap (J \cup J')
= \{x,y\}$ and define a mapping
\[ T \colon B(J,J') \to B(J \setminus \{x\}, J' \setminus \{y\}) \]
by
\[ T(I,\epsilon) := (I^{x,y}, \epsilon^{x,y}) .\]
Note that the mapping $I \mapsto I^{x,y}$ is injective on $\{ I \in \cI ~|~ x,y \in I\}$. Thus, recalling that jumps belonging to the same chain must have the same sign, it is not hard to see that, for any $(I',\epsilon') \in B(J \setminus \{x\}, J' \setminus \{y\})$, we have
\[ |T^{-1}(I',\epsilon')| \leq \begin{cases} 1 &\text{if } I' \cap D'_{x,y} \neq \emptyset \\ 2 &\text{if } I' \cap D'_{x,y} = \emptyset \end{cases} ,\]
where $D'_{x,y} := \{ x-2d-1, x+2d, y-2d-2, y+2d+1 \}$. Moreover, by
Lemma \ref{lem:torus-bijection2} and Claim
\ref{cl:torus-bijection-fluc-size}, for any $(I,\epsilon) \in
B(J,J')$,
\[ \Pr\big((S,X)=T(I,\epsilon)\big) = \Pr\big((S,X)=(I,\epsilon)\big) \cdot
 \begin{cases}
   2^{2d+1 + |\cC(I)| - |\cC(I^{x,y})|} &\text{if } I^{x,y} \neq \emptyset \\
   2^{2d+3} (2 - 2^{1-n/2}) &\text{if } I^{x,y} = \emptyset
 \end{cases} . \]
Thus, by \eqref{eq:chain_structure_I_x_y} and
\eqref{eq:chain_structure_I_x_y_2}, for any $(I,\epsilon) \in
B(J,J')$,
\[ \Pr\big((S,X)=T(I,\epsilon)\big) \ge \Pr\big((S,X)=(I,\epsilon)\big) \cdot 2^{2d-1}\cdot
 \begin{cases} 1 &\text{if } I \cap D_{x,y} \neq \emptyset \\ 16 &\text{if } I \cap D_{x,y} = \emptyset \end{cases} .\]
Denote
\begin{align*}
B &:= \big\{ (I,\epsilon) \in B(J,J') ~|~ I \cap D_{x,y} = \emptyset \big\} , \\
B' &:= \big\{ (I',\epsilon') \in B(J \setminus \{x\},J' \setminus
\{y\}) ~|~ I' \cap D'_{x,y} = \emptyset \big\} .
\end{align*}
Note that $T$ maps $B$ into $B'$ and $B(J,J') \setminus B$ into $B(J
\setminus \{x\},J' \setminus \{y\}) \setminus B'$. Therefore,
applying Lemma \ref{lem:main-tool} to the restriction of $T$ to $B$
and separately to its restriction to $B(J,J') \setminus B$, we
obtain
\begin{align*}
\Pr\big((S,X) \in B(J,J')\big) &= \Pr\big((S,X) \in B\big) +
\Pr\big((S,X) \in B(J,J') \setminus B \big) \\
&\leq 2^{-(2d+2)} \cdot \Pr\big((S,X) \in B' \big) +
2^{-(2d-1)} \cdot \Pr\big((S,X) \in B(J \setminus \{x\},J' \setminus \{y\}) \setminus B' \big) \\
&\leq 2^{-(2d-1)} \cdot \Pr\big((S,X) \in B(J \setminus \{x\},J'
\setminus \{y\}) \big) .
\end{align*} Thus,
\eqref{eq:removing_jumps_induction} follows by induction, proving
the lemma.
\end{proof}

\begin{proof}[Proof of Lemma \ref{lem:torus-ratio-ineq-for-R}]
The proof utilizes a similar technique as the proof of
Lemma~\ref{lem:torus-jumps-at-multiple-positions}, where this time
we aim to add jumps to the jump structure of the homomorphism rather
than remove jumps.

For a feasible jump structure $I$, denote
\[ U(I) := \begin{cases} \left\{ (x,y) \in V^2 ~|~ \rho(x,y) \geq 2d+3, ~ \substack{\rho^+(I,x),~\rho^+(x,I)+1,\\ \rho^+(y,I),~\rho^+(I,y)-1} \in D \right\}
&\text{if } I \neq \emptyset \\[5pt]
\left\{ (x,y) \in V^2 ~|~ \rho(x,y) \in D, ~ \rho(0,x) \text{ is
even} \right\} &\text{if } I = \emptyset \end{cases}
 , \]
where $\rho^+(x,I) := \min_{s \in I} \rho^+(x,s)$, $\rho^+(I,x) := \min_{s \in I} \rho^+(s,x)$ and
\[ D := \{ 2d+3, 2d+5, \dots \} .\]
For a feasible jump structure $I$ and a pair $(x,y) \in U(I)$, define
\[ I_{x,y} := (I \cap [y,x]) \cup ((I \cap [x,y]) + 1) \cup \{x,y\} . \]
Note that $I_{x,y}$ satisfies condition \eqref{eq:torus-feasible-jump-structure1}. Moreover, it is easy to see that the chain structure of $I_{x,y}$ satisfies
\[ \cC(I_{x,y}) = \cC(I \cap [y,x]) \cup \cC((I \cap [x,y]) + 1) \cup \{ (x,1), (y,1) \} . \]
Now, for a feasible sign vector $\epsilon \in \cB^*(I)$ and a sign
$i \in \{-1,1\}$, define $\epsilon_{x,y,i} \in \cB^*(I_{x,y})$ by
\[ \epsilon_{x,y,i}(k,t) := \begin{cases}
 \epsilon(k,t) &\text{if } k \in I \\
 \epsilon(k-1,t) &\text{if } k-1 \in I \\
 i &\text{if } k=x \\
 -i &\text{if } k=y
 \end{cases}, \quad (k,t) \in \cC(I_{x,y}).\]
That is, the sign of a chain in $\cC(I_{x,y}) \setminus \{(x,1),(y,1)\}$ is inherited from its corresponding chain in $\cC(I)$ and the sign of the chain at $x$, which is opposite of that of $y$, is determined independently.
Note that, by \eqref{eq:torus-feasible-jump-structure2}, $I_{x,y}$ is a feasible jump structure.
Moreover, since $|I_{x,y}|=|I|+2$ and $|\cC(I_{x,y})|=|\cC(I)|+2$, Lemma \ref{lem:torus-bijection2} and Claim \ref{cl:torus-bijection-fluc-size} imply that
\begin{equation}
\label{eq:torus-ratio-ineq-for-R-prob-I}
\Pr\big((S,X)=(I_{x,y},\epsilon_{x,y,i})\big) = \Pr\big((S,X)=(I,\epsilon)\big) \cdot
2^{-2d-3} \cdot \begin{cases} 1 &\text{if } I \neq \emptyset \\ (2-2^{1-n/2})^{-1} &\text{if } I = \emptyset \end{cases} .
\end{equation}

Denote by $\cI$ the set of all feasible jump structures. For $r \geq 0$, let $B_r$ denote the set of feasible signed jump structures having $2r$ jumps, i.e.,
\[ B_r := \{ (I,\epsilon) ~|~ I \in \cI, ~ \epsilon \in \cB^*(I), ~ |I|=2r \}. \]
For $r \geq 1$, define the mapping
\[ T_r \colon B_{r-1} \to \cP(B_r) \]
by
\[ T_r(I,\epsilon) := \big\{ (I_{x,y}, \epsilon_{x,y,i}) ~|~ (x,y) \in U(I), ~ i \in \{-1,1\} \big\} .\]
Assume henceforth that $n \geq Cdr$.
Then, since the mapping $((x,y),i) \mapsto (I_{x,y}, \epsilon_{x,y,i})$ is injective on $U(I) \times \{-1,1\}$, we have
\[ |T_r(I,\epsilon)| = 2 |U(I)| \geq 2 (n/2 - Cdr)^2.\]
Therefore, by \eqref{eq:torus-ratio-ineq-for-R-prob-I},
\[ \Pr\big((S,X) \in T_r(I,\epsilon)\big) \geq \Pr\big((S,X) = (I,\epsilon)\big) \cdot \frac{(n-Cdr)^2}{2^{2d+5}} \cdot \begin{cases} 2 &\text{if } r \geq 2 \\ 1 &\text{if } r=1 \end{cases} . \]
For $(I',\epsilon') \in B_r$, denote
\[ N_r(I',\epsilon') := \big\{ (I,\epsilon) \in B_{r-1} ~|~ (I',\epsilon') \in T_r(I,\epsilon) \big\} .\]
We have
\[ |N_r(I',\epsilon')| \leq r^2 \cdot \begin{cases} 2 &\text{if } r \geq 2 \\ 1 &\text{if } r=1 \end{cases} .\]
Thus, considering separately the case $r=1$, Lemma \ref{lem:main-tool} implies that for any $r \geq 1$,
\[ \Pr(R=r-1) = \Pr\big((S,X) \in B_{r-1}\big) \leq \Pr\big((S,X) \in B_r\big) \cdot \frac{r^2 2^{2d+5}}{(n-Cdr)^2} = \Pr(R=r) \cdot \frac{r^2 2^{2d+5}}{(n-Cdr)^2} . \qedhere \]
\end{proof}

\section{Local limits on the line}
\label{sec:local-limits}

In this section we prove the theorems which were stated in Section
\ref{sec:main-results-local-limit}. Throughout this section, the
parameter $d \geq 1$ is fixed, and so we drop the $d$ from the
notation when convenient. On the other hand, the parameter $n \geq
1$ is allowed to vary, and our main goal is to understand
$\Hom(P_{n,d}) := \Hom(P_{n,d},0)$ as $n$ grows larger. At first, in
Section \ref{sec:local-limit-counting}, we investigate the
asymptotic size of $\Hom(P_{n,d})$ as $n$ tends to infinity.
Subsequently, in Sections \ref{sec:local-limit-infinite-hom} and
\ref{sec:local-limit-markov-chain}, we describe the local limit of
such homomorphisms as a probability measure on infinite
homomorphisms defined through a Markov chain (see Figure
\ref{fig:local-limit-markov-chain}).

\subsection{Definitions}

\newcommand{\concat}{\ensuremath{\,\circ\,}}

Given a finite set $\Pi$, called an {\em alphabet}, we denote by $\Pi^*$ the set of all finite words on $\Pi$. That is,
\[ \Pi^* := \big\{ (a_1,a_2,\dots,a_t) ~|~ a_i \in \Pi, ~ t \geq 0 \big\} . \]
For $u,v \in \Pi^*$, we denote the length of $u$ by $|u|$ and the {\em concatenation} of $u$ and $v$ by $u \concat v$, i.e.,
\[ u \concat v := (u_1, \dots, u_{|u|}, v_1, \dots, v_{|v|}) .\]
It is clear that concatenation is associative. For $u \in \Pi^*$ with $|u| \geq 1$, let $u^-$ be the word obtained from $u$ by dropping the last element, i.e.,
\[ u^- := (u_1, \dots, u_{|u|-1}) .\]

Define the derivative operator $D_n \colon \Hom(P_{n,d}) \to
\{-1,1\}^n$ by
\begin{equation}
\label{eq:local-limit-def-D}
(D_n(f))_k := f(k) - f(k-1) , \quad 1 \leq k \leq n .
\end{equation}
Denote by $\cD \subset \{-1,1\}^*$ the set of words on $\{-1,1\}$ which do not contain $(-1,-1,-1)$ or $(1,1,1)$ as a subsequence, and note that $D_n(\Hom(P_{n,d})) \subset \cD$. Let $\Sigma$ be the four letter alphabet
\[ \Sigma := \{a,b,A,B\} , \]
where
\begin{align*}
a &:= (1,-1), \\
b &:= (-1,1), \\
A &:= (1,1,-1), \\
B &:= (-1,-1,1).
\end{align*}
These basic sequences will serve as a means to encode homomorphisms into words (see Figure \ref{fig:local-limit-encoding-building-blocks}). Define $T' \colon \Sigma^* \to \{-1,1\}^*$ by
\[ T'(x) := x_1 \concat x_2 \concat \cdots \concat x_{|x|} .\]
For $x \in \Sigma^*$, define the {\em weight} of $x$ by
\begin{equation}
\label{eq:local-limit-def-weight}
w(x) := |T'(x)| = \sum_{k=1}^{|x|} |x_k| .
\end{equation}
Now, we define a mapping $T \colon \cD \to \Sigma^*$ recursively by the relations
\begin{equation}
\label{eq:local-limit-recursion-for-T}
\begin{aligned}
T(()) &:=() , \\
T((1)) &:= (a) , \\
T((-1)) &:= (b) , \\
T((1,1)) &:= (A) , \\
T((-1,-1)) &:= (B) , \\
T(u \concat v) &:= (u) \concat T(v) \quad \text{ for } u \in \Sigma \text{ and } v \in \cD .
\end{aligned}
\end{equation}
It is not hard to see that $T$ is indeed well-defined (see Figure
\ref{fig:local-limit-encoding-to-words}), and that it maps a word $u
\in \cD$ to the unique word $x \in \Sigma^*$ satisfying $T'(x)=u$ or
$T'(x)^{-}=u$ (in which case $w(x)=|u|$ or $w(x)=|u|+1$,
respectively). Also, one should note that $T^{-1}(x)=\emptyset$ if
$x \in \Sigma^*$ contains $(a,B)$ or $(b,A)$ as a sub-word or if
$x=\emptyset$, and that $T^{-1}(x) = \{ T'(x), T'(x)^- \}$
otherwise.

\begin{figure}[!t]
\[ \begin{array}{ccc}
\begin{tikzpicture}\pgftransformcm{0.45}{0}{0}{0.55}{\pgfpoint{0cm}{0cm}} \drawRandomWalk{0}{{0,1,0}}{99}{}{}{};\end{tikzpicture}
& &
\begin{tikzpicture}\pgftransformcm{0.45}{0}{0}{0.55}{\pgfpoint{0cm}{0cm}} \drawRandomWalk{0}{{0,1,2,1}}{99}{}{}{};\end{tikzpicture}
\\
a := (1, -1) & &
A := (1, 1, -1) \\
\\
\begin{tikzpicture}\pgftransformcm{0.45}{0}{0}{0.55}{\pgfpoint{0cm}{0cm}} \drawRandomWalk{0}{{0,-1,0}}{99}{}{}{};\end{tikzpicture}
& &
\begin{tikzpicture}\pgftransformcm{0.45}{0}{0}{0.55}{\pgfpoint{0cm}{0cm}} \drawRandomWalk{0}{{0,-1,-2,-1}}{99}{}{}{};\end{tikzpicture}
\\
b := (-1, 1) & &
B := (-1, -1, 1)
\end{array} \]
\caption{The basic building blocks for encoding a homomorphism into a word on the alphabet $\Sigma := \{a,b,A,B\}$.}
\label{fig:local-limit-encoding-building-blocks}
\end{figure}
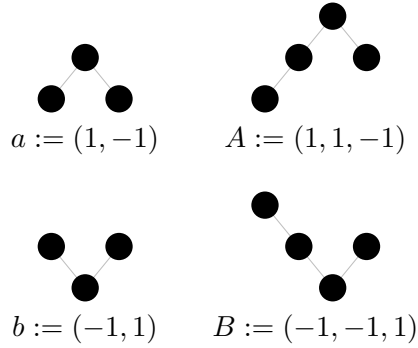

Another observation which will be useful later on is that the
recursive relation in the last line of \eqref{eq:local-limit-recursion-for-T} may be
generalized to hold for certain $u \in \cD$.

\begin{claim}\label{cl:local-limit-recursion-for-T-extended}
We have
\begin{equation}
\label{eq:local-limit-recursion-for-T-extended}
T(u \concat v) = T(u) \concat T(v) \quad \text{ for } u, v \in \cD \text{ such that } |u|=w(T(u)) .
\end{equation}
\end{claim}
\begin{proof}
We prove the claim by induction on $|u|$. If $|u|=0$ then there is
nothing to prove. Otherwise, $|u| \geq 1$. By the assumption, we
have $|u|=w(T(u))$, which implies that $u$ may be decomposed as
$u=u' \concat u''$, where $u' \in \Sigma$. Note that this now
implies that $|u''|=w(T(u''))$, since $T(u)=T(u') \concat T(u'')$,
by \eqref{eq:local-limit-recursion-for-T}, and since $|u'|=w(T(u'))$
trivially. Therefore, by induction,
\begin{align*}
T(u \concat v) &= T(u' \concat u'' \concat v) = T(u') \concat T(u'' \concat v) \\
&= T(u') \concat T(u'') \concat T(v) = T(u' \concat u'') \concat T(v) = T(u) \concat T(v) . \qedhere
\end{align*}
\end{proof}

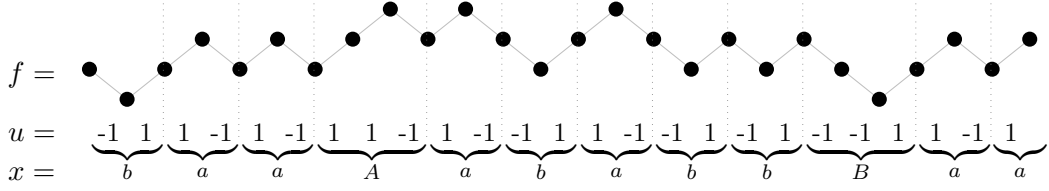
\begin{figure}[!t]
\newlength{\dlen}
\newlength{\dspc}
\dlen=0.42cm
\dspc=0.08cm

\newsavebox{\myline}
\savebox{\myline}{\tikz{\draw[line width=0.1mm,gray,dotted](0,0)--(0,1.85);}}

\newcommand{\wordline}{\makebox[0cm]{\raisebox{0.05cm}{\hspace{0.04cm}\smash{{\usebox{\myline}}}}}}

\newsavebox{\worda}
\newsavebox{\wordb}
\newsavebox{\wordA}
\newsavebox{\wordB}
\savebox{\worda}{$\hspace{\dspc}\underbrace{
    \makebox[\dlen]{\small 1}\hspace{\dspc}
    \makebox[\dlen]{\small -1}
    }_{a}$\wordline}
\savebox{\wordb}{$\hspace{\dspc}\underbrace{
    \makebox[\dlen]{\small -1}\hspace{\dspc}
    \makebox[\dlen]{\small 1}
    }_{b}$\wordline}
\savebox{\wordA}{$\hspace{\dspc}\underbrace{
    \makebox[\dlen]{\small 1}\hspace{\dspc}
    \makebox[\dlen]{\small 1}\hspace{\dspc}
    \makebox[\dlen]{\small -1}
    }_{A}$\wordline}
\savebox{\wordB}{$\hspace{\dspc}\underbrace{
    \makebox[\dlen]{\small -1}\hspace{\dspc}
    \makebox[\dlen]{\small -1}\hspace{\dspc}
    \makebox[\dlen]{\small 1}
    }_{B}$\wordline}

\[
\begin{array}{rl}
\raisebox{2ex}{$f =$}
&
\begin{tikzpicture}
    \tikzstyle{every node}=[minimum size=0.2cm,inner sep=0]
    \pgftransformcm{0.5}{0}{0}{0.4}{\pgfpoint{0cm}{0cm}}
    \drawRandomWalk{0}{{0,-1,0,1,0,1,0,1,2,1,2,1,0,1,2,1,0,1,0,1,0,-1,0,1,0,1}}{99}{}{}{}
\end{tikzpicture}

\\
\underset{\raisebox{-1.75ex}{$x =$}}{u =}
&
\hspace{0.065cm}
\usebox{\wordb}
\usebox{\worda}
\usebox{\worda}
\usebox{\wordA}
\usebox{\worda}
\usebox{\wordb}
\usebox{\worda}
\usebox{\wordb}
\usebox{\wordb}
\usebox{\wordB}
\usebox{\worda}
\underbrace{
    \makebox[\dlen]{\small 1 \hfill}
    }_{a}

\end{array}
\]

\caption{A homomorphism $f \in \Hom(P_{n,d})$ is first viewed as a word $u := D_n(f) $ of length $n$ on the alphabet $\{-1,1\}$. Then, $u$ is encoded into a word $x := T(u)$ on the alphabet $\Sigma$ by sequentially reading off the letters from left to right, as defined in the recursive formula in \eqref{eq:local-limit-recursion-for-T}. If this process exhausts $u$ completely then we end up with a word $x$ of weight exactly $n$. Otherwise, we remain with a tail of $u$ of length one or two (as is the case in this figure), which is a prefix of at least one element in $\Sigma$. In this case, the last letter is chosen in such a way that the weight of the resulting word $x$ is $n+1$, as defined by the base cases in \eqref{eq:local-limit-recursion-for-T}. }
\label{fig:local-limit-encoding-to-words}
\end{figure}

We say a word $x \in \Sigma^*$ is {\em $d$-legal} if it satisfies the conditions
\begin{equation}
\label{eq:local-limit-d-legal-word}
\begin{array}{lclr}
x_m = A & \Rightarrow & x_{m-i} = a , &  \forall i \in \{1,\dots,d-1\} \text{ such that } i<m , \\
x_m = B & \Rightarrow & x_{m-i} = b , &  \forall i \in \{1,\dots,d-1\} \text{ such that } i<m , \\
x_m = A , ~ m > d & \Rightarrow & x_{m - d} \in \{a,A\} , \\
x_m = B , ~ m > d & \Rightarrow & x_{m - d} \in \{b,B\}.
\end{array}
\end{equation}
Denote by $\Omega_{n,d}$ the set of $d$-legal words on $\Sigma$ of
weight $n$ or $n+1$. That is,
\[ \Omega_{n,d} := \{ x \in \Sigma^* ~|~ x \text{ is $d$-legal}, ~ w(x)\in\{n,n+1\} \} .\]
Define
\[ L_n := (T \circ D_n) |_{\Hom(P_{n,d})} .\]

\begin{claim}
\label{cl:local-limit-bijection}
The mapping $L_n$ is a bijection between $\Hom(P_{n,d})$ and $\Omega_{n,d}$.
\end{claim}
\begin{proof}
It is clear from \eqref{eq:local-limit-def-weight} and \eqref{eq:local-limit-recursion-for-T} that $L_n$ injectively maps $\Hom(P_{n,d})$ to words on $\Sigma$ of weight $n$ or $n+1$. It remains to show that the image of $L_n$ is precisely $\Omega_{n,d}$.

One may easily see that a homomorphism $f \in \Hom(P_{n,1})$ is a homomorphism in $\Hom(P_{n,d})$ if and only if $D_n(f)$ does not contain a sequence of the form
\[ u^l_{\pm} := \pm (1,\underbrace{1,-1,\dots,1,-1}_{2l},1,1) , \]
with $1 \leq l \leq d-1$. Indeed, $D_n(f)$ contains $u^l_{\pm}$ if and only if there exist $0 \leq i,j \leq n$ such that $|i-j|=2l+3$ and $|f(i)-f(j)|=3$. Now, it is also not hard to check that $u^l_+$ appears in $D_n(f)$ if and only if $L_n(f)$ contains a subword of the form
\[ (A,\underbrace{a,\dots,a}_{l-1},A) \quad\text{or}\quad (B,\underbrace{a,\dots,a}_{l},A) \quad\text{or}\quad (b,\underbrace{a,\dots,a}_{l},A) , \]
depending on the position of $u^l_+$ in $D_n(f)$. The same is true for $u^l_{-}$ with $\{a,A\}$ and $\{b,B\}$ interchanged. Therefore, by \eqref{eq:local-limit-d-legal-word}, we see that $u^l_{\pm}$ appears in $D_n(f)$, for some $1 \leq l \leq d-1$, if and only if $L_n(f)$ is not $d$-legal.

We have shown that for any $f \in \Hom(P_{n,1})$, $f \in \Hom(P_{n,d})$ if and only if $L_n(f) \in \Omega_{n,d}$. In particular, since $\Hom(P_{n,d}) \subset \Hom(P_{n,1})$, we have $L_n(\Hom(P_{n,d})) \subset \Omega_{n,d}$. For the other direction, let $x \in \Omega_{n,d}$. Either $T'(x)$ or $T'(x)^-$ is of length $n$. Let $u \in \cD$ be this sequence and let $f := D_n^{-1}(u) \in \Hom(P_{n,1})$. Since $L_n(f)=T(u)=x$ is $d$-legal, we see that $f \in \Hom(P_{n,d})$. Hence, $\Omega_{n,d} \subset L_n(\Hom(P_{n,d}))$, completing the proof.
\end{proof}

\subsection{Counting the homomorphisms}
\label{sec:local-limit-counting}

In this section we prove Theorem \ref{thm:local-limit-hom-count}. This is done by deriving a recursion formula and investigating its characteristic polynomial.

For $0 \leq k,m \leq d$, define
\begin{align*}
\Omega_{n,d}(k,m) &:= \left\{ x \in \Omega_{n,d} ~|~ x_1,\dots,x_k \neq A ~\text{and}~ x_1,\dots,x_m \neq B \right\} .
\end{align*}
By symmetry we have $|\Omega_{n,d}(k,m)| = |\Omega_{n,d}(m,k)|$, so we can define
\[ c_n(k) := |\Omega_{n,d}(k,d)| = |\Omega_{n,d}(d,k)| , \quad 0 \leq k \leq d-1 .\]
This definition is motivated by the following two lemmas, which show that the $c_n(k)$ satisfy some explicit recursion formulas and that they have a simple relation to $|\Hom(P_{n,d})|$.

\begin{lemma}
\label{lem:local-limit-hom-count}
For any $n \geq 3$, we have
\[ |\Hom(P_{n,d})| = 2c_{n-2}(0) + 2c_{n-3}(d-1) .\]
\end{lemma}
\begin{proof}
Note that by \eqref{eq:local-limit-d-legal-word}, we have
\begin{equation}\label{eq:Omega_n_d_relations}
\begin{aligned}
\{ x \in \Omega_{n,d} ~|~ x_1=a \} &= \{ (a) \concat x ~|~ x \in \Omega_{n-2,d}(0,d) \} , \\
\{ x \in \Omega_{n,d} ~|~ x_1=A \} &= \{ (A) \concat x ~|~ x \in \Omega_{n-3,d}(d-1,d) \} .
\end{aligned}
\end{equation}
Therefore, by partitioning according to the first element and by the symmetry between $\{a,A\}$ and $\{b,B\}$, we obtain
\begin{align*}
|\Omega_{n,d}| &= 2 \big|\{ x \in \Omega_{n,d} ~|~ x_1=a \}\big| + 2 \big|\{ x \in \Omega_{n,d} ~|~ x_1=A \}\big| \\
&= 2c_{n-2}(0) + 2c_{n-3}(d-1) .
\end{align*}
The result now follows as $|\Omega_{n,d}|=|\Hom(P_{n,d})|$, by Claim \ref{cl:local-limit-bijection}.
\end{proof}

\begin{lemma}
For any $n \geq 3$, we have
\begin{align}
\label{eq:recursion1}
c_n(0) &= c_{n-2}(0) + c_{n-2}(d-1) + c_{n-3}(d-1) , \\
\label{eq:recursion2}
c_n(k) & = c_{n-2}(k-1) + c_{n-2}(d-1) , \quad 1 \leq k \leq d-1 .
\end{align}
\end{lemma}
\begin{proof}
Note that by \eqref{eq:local-limit-d-legal-word}, similarly to
\eqref{eq:Omega_n_d_relations}, we have
\begin{align*}
\{ x \in \Omega_{n,d}(0,d) ~|~ x_1=a \} &= \{ (a) \concat x ~|~ x \in \Omega_{n-2,d}(0,d) \} , \\
\{ x \in \Omega_{n,d}(0,d) ~|~ x_1=b \} &= \{ (b) \concat x ~|~ x \in \Omega_{n-2,d}(d,d-1) \} , \\
\{ x \in \Omega_{n,d}(0,d) ~|~ x_1=A \} &= \{ (A) \concat x ~|~ x \in \Omega_{n-3,d}(d-1,d) \} .
\end{align*}
Therefore, by partitioning according to the first element, we obtain
\begin{align*}
c_n(0)
&= |\Omega_{n,d}(0,d)| \\
&= \big|\{ x \in \Omega_{n,d}(0,d) ~|~ x_1=a \}\big| + \big|\{ x \in \Omega_{n,d}(0,d) ~|~ x_1=b \}\big| + \big|\{ x \in \Omega_{n,d}(0,d) ~|~ x_1=A \}\big| \\
&= |\Omega_{n-2,d}(0,d)| + |\Omega_{n-2,d}(d,d-1)| + |\Omega_{n-3,d}(d-1,d)| \\
&= c_{n-2}(0) + c_{n-2}(d-1) + c_{n-3}(d-1) .
\end{align*}
In a similar manner, for $1 \leq k \leq d-1$, we have
\begin{align*}
c_n(k)
    &= |\Omega_{n,d}(k,d)| \\
    &= \big|\{ x \in \Omega_{n,d}(k,d) ~|~ x_1=a \}\big| + \big|\{ x \in \Omega_{n,d}(k,d) ~|~ x_1=b \}\big| \\
    &= \big|\{ (a) \concat x ~|~ x \in \Omega_{n-2,d}(k-1,d) \}\big| + \big|\{ (b) \concat x ~|~ x \in \Omega_{n-2,d}(d,d-1) \}\big| \\
    &= c_{n-2}(k-1) + c_{n-2}(d-1) . \qedhere
\end{align*}
\end{proof}

We express all quantities $c_n(k)$ in terms of $c_n(d-1)$. Substituting $k=d-1$ in \eqref{eq:recursion2} yields
\[ c_n(d-2) = c_{n+2}(d-1) - c_n(d-1) . \]
Now substituting $k=d-2$ in \eqref{eq:recursion2} gives
\[ c_n(d-3) = c_{n+2}(d-2) - c_n(d-1) = c_{n+4}(d-1) - c_{n+2}(d-1) - c_n(d-1) . \]
Continuing in this manner (by induction), we get for $1 \leq m < d$,
\begin{equation}
\label{eq:recursion3}
c_n(d-m-1) = c_{n+2m}(d-1) - c_{n+2m-2}(d-1) - \dots - c_{n+2}(d-1) - c_n(d-1) .
\end{equation}
In particular, for $m = d-1$ this gives,
\[ c_n(0) = c_{n+2d-2}(d-1) - c_{n+2d-4}(d-1) - \dots - c_{n+2}(d-1) - c_n(d-1) .\]
Substituting this in \eqref{eq:recursion1} gives
\[ c_{n+2d-2}(d-1) = 2 c_{n+2d-4}(d-1) + c_{n-3}(d-1), \quad n\ge 3.\]
The characteristic polynomial for this equation is
\[ q(\mu) := \mu^{2d - 1} (\mu^2 - 2) - 1 .\]

\begin{claim}
The polynomial $q$ has $2d+1$ distinct (complex) roots. Exactly one of these, which we denote by $\mu_0$, is positive. Moreover, $\mu_0 > \sqrt{2}$, while all other roots have modulus less than $\sqrt{2}$.
\end{claim}
\begin{proof}
Assume that $d \geq 2$ (the case $d=1$ can be verified directly). It
is easy to verify that the derivative of $q$ does not vanish at any
zero, so that the roots are simple, and hence there are $2d+1$
distinct roots. Since $q(\pm \sqrt{2})=-1$, $q(\sqrt{3})>0$ and
$q(-2/\sqrt{3})>0$, the intermediate value theorem implies that
there are roots $\sqrt{2} < \mu_0 < \sqrt{3}$ and $-\sqrt{2} < \mu_1
< -2/\sqrt{3}$. Considering $q$ as a real function, by
differentiating, one finds that $q$ has a single minimum and a
single maximum, and hence at most $3$ real roots. Since $q(-1)=0$,
we see that $\mu_0$ is indeed the unique positive root. For the last
part, it suffices to show $q$ has $2d-1$ roots of modulus at most
$1$. This is a consequence of Rouch\'{e}'s theorem applied to $q$
and $g(z):=2z^{2d-1}$ on the disc $D:=\{ z \in \mathbb{C} ~|~
|z|\leq r\}$ for any sufficiently small $r>1$. Indeed, on $\partial
D$, we have $|g(z)|=2r^{2d-1}$ and $|q(z)+g(z)| \leq r^{2d+1}+1$,
and since $r^{2d+1}+1 < 2r^{2d-1}$ (using our assumption that $d\ge
2$), Rouch\'{e}'s theorem implies that $g$ and $q$ have the same
number of zeros in $D$. As $g$ clearly has $2d-1$ zeros in $D$, this
completes the proof.
\end{proof}

Let $\mu_0$ be the unique positive root of $q$. We denote $\lambda := \mu_0^2$. That is, $\lambda$ is the unique positive solution of the equation
\begin{equation}
\label{eq:lambda}
\lambda^{d - 1/2} (\lambda - 2) = 1 .
\end{equation}

\begin{claim}
\label{cl:local-limit-asymptotic-c_n}
For any fixed $0 \leq k \leq d-1$, there exists a constant $r_k>0$, such that
\[ c_n(k) \sim r_k \lambda^{n/2} \quad \text{ as } n \to \infty . \]
\end{claim}
\begin{proof}
Denote by $\mu_0:=\sqrt{\lambda}, \mu_1, \dots, \mu_{2d}$ the roots of $q$. The roots $\mu_i$ are distinct, and therefore,
\[ c_n(k) = r_k^0 \mu_0^n + \cdots + r_k^{2d} \mu_{2d}^n ,\]
for some coefficients $r_k^i$. Now, since any word $x \in \{a,b\}^*$ is $d$-legal and has $w(x)=2|x|$, we see that
\[ c_n(k) = |\Omega_{n,d}(k,d)| \geq \left|\{a,b\}^{\lceil n/2 \rceil}\right| = 2^{\lceil n/2 \rceil} \geq \sqrt{2}^n .\]
Therefore, since $|\mu_i|<\sqrt{2}$ for $1 \leq i \leq 2d$, we must have $r_k^0 > 0$ for all $0 \leq k \leq d-1$, and then
\[ c_n(k) \sim r_k^0 \mu_0^n = r_k^0 \lambda^{n/2} . \qedhere \]
\end{proof}

We now have all the ingredients to prove Theorem \ref{thm:local-limit-hom-count}.
\begin{proof}[Proof of Theorem \ref{thm:local-limit-hom-count}]
By Lemma \ref{lem:local-limit-hom-count} and Claim \ref{cl:local-limit-asymptotic-c_n}, we have
\begin{align*}
|\Hom(P_{n,d})| &= 2c_{n-2}(0) + 2c_{n-3}(d-1) \\
    &\sim 2r_0 \lambda^{(n-2)/2} + 2r_{d-1} \lambda^{(n-3)/2} = C(d)\lambda^{n/2}(1 + o(1)) \quad \text{ as } n \to \infty . \qedhere
\end{align*}
\end{proof}

Claim \ref{cl:local-limit-asymptotic-c_n} gives the asymptotic behavior of $c_n(k)$ as $n \to \infty$ for fixed $d \geq 1$. Specifically, it says that the order of magnitude of $c_n(k)$ is $\lambda^n$, where $\lambda=\lambda(d)$ depends on $d$ and is given implicitly by \eqref{eq:lambda}. The next claim describes the dependence of the constant $\lambda(d)$ on $d$ as $d \to \infty$.

\begin{claim}\label{cl:lambda_asymptotics}
The unique positive solution $\lambda$ of \eqref{eq:lambda} satisfies
\[ \lambda = \lambda(d) = 2 + 2^{-d+1/2} (1 - o(1)) \quad \text{ as } d \to \infty .\]
\end{claim}
\begin{proof}
Writing $\lambda = 2 + \delta$, we have by \eqref{eq:lambda} that
\[ \delta = (2+\delta)^{-d+1/2} .\]
Thus, since $\delta>0$, we have
\[ \delta \leq 2^{-d+1/2} ,\]
and therefore,
\[ 2^{d-1/2} \delta = (1+\delta/2)^{-d+1/2} \to 1 \quad \text{ as } d \to \infty . \qedhere \]
\end{proof}

\subsection{Infinite homomorphisms}
\label{sec:local-limit-infinite-hom}

Denote by $P_{\infty,d}$ the graph on the vertex set $\{0,1,2,...\}$
with the edge set $\{ (i,j) ~|~ |i - j| = 1,3,...,2d+1 \}$. Note
that $\Hom(P_{\infty,d}):=\Hom(P_{\infty,d},0)$ is an infinite set
of homomorphisms. For a homomorphism $f \in \Hom(P_{n,d})$ (where
possibly $n=\infty$) and an integer $r \geq 0$, we denote by
$B_r(f)$ the restriction of $f$ to the first $r+1$ vertices, so that
\[ B_r(f) := f|_{\{0,1,\dots,\min\{r,n\}\}} \in \Hom(P_{\min\{r,n\},d}) .\]

An infinite word $x$ on $\Sigma$ is $d$-legal if it satisfies \eqref{eq:local-limit-d-legal-word}, as for finite words. Denote by $\Omega_{\infty,d}$ the set of infinite $d$-legal words on $\Sigma$. That is,
\[ \Omega_{\infty,d} := \{ x \in \Sigma^\N ~|~ x \text{ is $d$-legal} \} .\]
The mapping $D_n$ defined in \eqref{eq:local-limit-def-D} extends to
the case $n=\infty$ in an obvious way. The mapping $T$ defined in
\eqref{eq:local-limit-recursion-for-T} can also be extended to map
the infinite words $D_{\infty}(\Hom(P_{\infty,d}))$ to
$\Omega_{\infty,d}$ by the same recursion formula. Then, following
the proof of Claim \ref{cl:local-limit-bijection}, we see that
\[ L_{\infty} := (T \circ D_{\infty}) |_{\Hom(P_{\infty,d})} \]
is a bijection between $\Hom(P_{\infty,d})$ and $\Omega_{\infty,d}$.

\subsection{The local limit as a Markov chain}
\label{sec:local-limit-markov-chain}

The main goal of this section is to prove Theorem \ref{thm:local-limit}. To this end, we will describe a Markov chain (see Figure \ref{fig:local-limit-markov-chain}) on the state space
\[ \tilde{\Sigma} := \{a_1,\dots,a_d,b_1,\dots,b_d,A,B\}, \]
which will allow us to generate words in $\Omega_{\infty,d}$, and hence also homomorphisms in $\Hom(P_{\infty,d})$ through the bijection $L_{\infty}$. Loosely speaking, the idea of this Markov chain is that the state $a_k$ ($b_k$) represents the fact that a streak of $k$ consecutive $a$'s ($b$'s) has been accumulated. Likewise, the state $A$ ($B$) represents the fact that a jump has occurred in the positive (negative) direction.

\begin{figure}[!t]
\begin{center}

\begin{tabular}{lcr}

\begin{tikzpicture}[->,>=stealth',shorten >=1pt,auto,node distance=2.2cm,semithick]
\tikzstyle{every state}=[fill=white,draw=black,thick,text=black,scale=0.9]

\node[state]    (a0)    {$a_1$};
\node[state]    (a1)    [right of=a0] {$a_2$};
\node[state,draw=white!0]   (aa)    [right of=a1] {$\cdots$};
\node[state]    (ad)    [right of=aa] {$a_d$};
\node[state]    (A)     [left=2cm, above=0.75cm] at (a0) {$A$};

\node[state]    (b0)    [below of=a0] {$b_1$};
\node[state]    (b1)    [right of=b0] {$b_2$};
\node[state,draw=white!0]   (bb)    [right of=b1] {$\cdots$};
\node[state]    (bd)    [right of=bb] {$b_d$};
\node[state]    (B)     [left=2cm, below=0.75cm] at (b0) {$B$};

\path (a0) edge  node[above] {} (a1);
\path (a1) edge  node[above] {} (aa);
\path (aa) edge  node[above] {} (ad);
\path (ad) edge  [bend right] node[above] {} (A);
\path (ad) edge  [loop right] node[above] {} (ad);

\path (b0) edge  node[above] {} (b1);
\path (b1) edge  node[above] {} (bb);
\path (bb) edge  node[above] {} (bd);
\path (bd) edge  [bend left] node[above] {} (B);
\path (bd) edge  [loop right] node[above] {} (bd);

\path (A) [bend right] edge  (b0);
\path (A) [bend left] edge  (a1);
\path (B) [bend left] edge  (a0);
\path (B) [bend right] edge  (b1);
\path[<->] (b0) edge (a0);

\tikzstyle{every edge}=[draw=gray]
\path [dotted] (a1) edge  (b0);
\path [dotted] (ad) edge  (b0);
\path [dotted] (b1) edge  (a0);
\path [dotted] (bd) edge  (a0);

\end{tikzpicture}

&
\quad\quad
&

\raisebox{4.8em}{
\begin{tikzpicture}[->,>=stealth',shorten >=1pt,auto,node distance=2.2cm,semithick]
\tikzstyle{every state}=[fill=white,draw=black,thick,text=black,scale=0.9]

\node[state]    (a0)    {$a_1$};
\node[state]    (A)     [left of=a0] {$A$};

\node[state]    (b0)    [below of=a0] {$b_1$};
\node[state]    (B)     [left of=b0] {$B$};

\path[<->] (a0) edge node[above] {} (A);
\path (a0) edge  [loop right] node[above] {} (a0);

\path[<->] (b0) edge node[above] {} (B);
\path (b0) edge  [loop right] node[above] {} (b0);

\path (A) edge  (b0);
\path (B) edge  (a0);
\path[<->] (b0) edge (a0);
\path (A) edge [loop left] node[above] {} (A);
\path (B) edge [loop left] node[above] {} (B);

\end{tikzpicture}
}

\end{tabular}

\end{center}
\caption{The Markov chains describing the local limit when $d \geq
2$ (on the left) and when $d=1$ (on the right). The allowed transitions are those determined by \eqref{eq:local-limit-d-legal-word}.}
\label{fig:local-limit-markov-chain}
\end{figure}
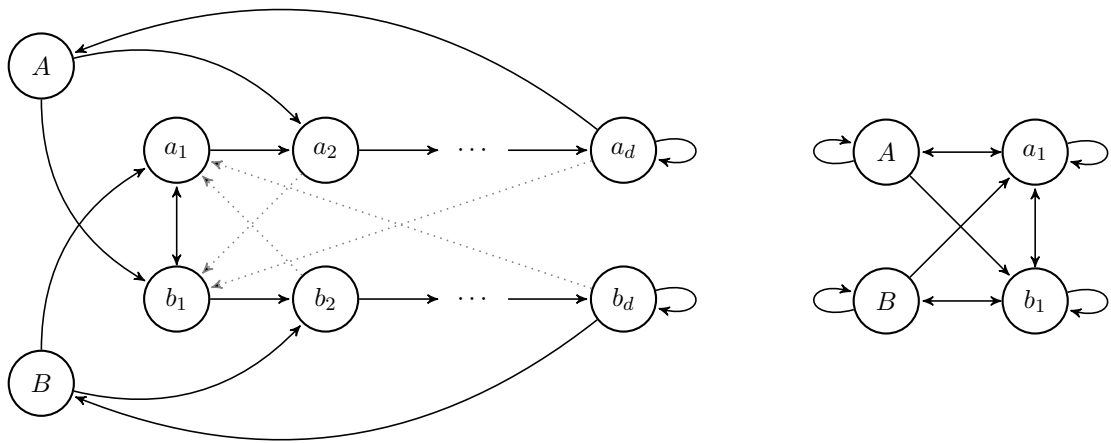

Consider the above Markov chain (see Figure
\ref{fig:local-limit-markov-chain}) on the state space
$\tilde{\Sigma}$ with the transition probabilities $p$ and the
initial state distribution $\pi$ as described below.
\begin{equation}
\label{eq:local-limit-transition-prob}
\begin{aligned}
&p(A, b_1) = p(a_d, a_d) := \lambda^{-1} , \\
&p(A, a_1) := \begin{cases} 0 &\text{if } d \geq 2 \\ \lambda^{-1} &\text{if } d=1 \end{cases} , \\
&p(a_k, b_1) := \frac{\lambda - 1}{\lambda^k (\lambda-2) + \lambda} , \quad 1 \leq k \leq d , \\
\end{aligned}
\end{equation}
where $\lambda$ is the unique positive solution to \eqref{eq:lambda}. The analogous relations hold with the roles of $\{a,A\}$ and $\{b,B\}$ interchanged. Figure \ref{fig:local-limit-markov-chain} shows the legal transitions (i.e., transitions having positive probability). The probability of unspecified legal transitions are determined by the condition $\sum_s p(s',s) = 1$. The initial state distribution $\pi$ is given by
\[ \pi(a_d) = \pi(b_d) :=  \frac{\lambda^{-1/2}}{2}  \quad\text{and}\quad
\pi(A) = \pi(B) :=  \frac{1 - \lambda^{-1/2}}{2} .\] It is
interesting to note that, since $\lambda > 2$ and using
\eqref{eq:lambda}, we have
\[ \frac{1}{2} > \lambda^{-1} = p(a_1,b_1) > p(a_2,b_1) > \cdots > p(a_d,b_1) = \frac{\lambda - 1}{\lambda+\sqrt{\lambda}} > \frac{1}{2 + \sqrt{2}} ,\]
which expresses the fact that there is a small but growing tendency to continue in the same direction.

Running this chain for an infinite amount of time and considering
its trajectory as an infinite word on $\tilde{\Sigma}$, we may
obtain an infinite word $W_{\infty}$ on $\Sigma$ by dropping the
subscripts of the letters in $\tilde{\Sigma}$. More precisely, let
$\tilde{W}(1),\tilde{W}(2),\dots$ be a Markov chain on
$\tilde{\Sigma}$ with transition probabilities as in
\eqref{eq:local-limit-transition-prob} and such that $\tilde{W}(1)
\sim \pi$. Define $\phi \colon \tilde{\Sigma} \to \Sigma$ by
$\phi(a_i):=a$, $\phi(b_i):=b$, $\phi(A):=A$ and $\phi(B):=B$. Then
$W_{\infty}$ is defined by $W_{\infty}(k) := \phi(\tilde{W}(k))$ for
$k \geq 1$. Recalling \eqref{eq:local-limit-d-legal-word}, it is
clear that this process generates a $d$-legal word, i.e. that
$W_{\infty} \in \Omega_{\infty,d}$. Denote by
\[ f_{\infty} := L_{\infty}^{-1}(W_{\infty}) \]
the infinite homomorphism corresponding to this word. Let $f_n$ be a
uniformly chosen homomorphism in $\Hom(P_{n,d})$.
Theorem~\ref{thm:local-limit} will follow when we show that
\begin{equation}
\label{eq:local-limit-convergence-of-F} \Pr(B_r(f_n)=f)
\xrightarrow[n \to \infty]{} \Pr(B_r(f_{\infty})=f) \quad \text{ for
any } r \geq 1 \text{ and } f \in \Hom(P_{r,d}).
\end{equation}

\begin{figure}[!t]
\centering
\vspace{8pt}
\[
\begin{array}{ccl}
\begin{array}{c}
    \begin{tikzpicture}
        \pgftransformcm{0.24}{0}{0}{0.3}{\pgfpoint{0cm}{0cm}}
        \tikzstyle{every node}=[minimum size=0.13cm,inner sep=0]
        \drawRandomWalk{0}{{0,-1,0,1,0,1,0,1,0}}{0}{}{red}{};
    \end{tikzpicture} \\
    b\underbrace{a\dots a}_{< d}
\end{array}
&
\longrightarrow
&
\begin{array}{c}
   \begin{array}{c}
    \begin{tikzpicture}
        \pgftransformcm{0.24}{0}{0}{0.3}{\pgfpoint{0cm}{0cm}}
        \tikzstyle{every node}=[minimum size=0.13cm,inner sep=0]
        \drawRandomWalk{0}{{0,-1,0,1,0,1,0,1,0}}{0}{}{red}{};
        \drawRandomWalk{8}{{0,1,0}}{99}{red}{red}{red};
    \end{tikzpicture} \\
    ba\dots a a
   \end{array}
    \text{ or }
   \begin{array}{c}
    \begin{tikzpicture}
        \pgftransformcm{0.24}{0}{0}{0.3}{\pgfpoint{0cm}{0cm}}
        \tikzstyle{every node}=[minimum size=0.13cm,inner sep=0]
        \drawRandomWalk{0}{{0,-1,0,1,0,1,0,1,0}}{0}{}{red}{};
        \drawRandomWalk{8}{{0,-1,0}}{99}{red}{red}{red};
    \end{tikzpicture} \\
    ba\dots a b
   \end{array}
\end{array}

\\ \\

\begin{array}{c}
    \begin{tikzpicture}
        \pgftransformcm{0.24}{0}{0}{0.3}{\pgfpoint{0cm}{0cm}}
        \tikzstyle{every node}=[minimum size=0.13cm,inner sep=0]
        \drawRandomWalk{0}{{0,-1,0,1,0}}{0}{}{}{};
        \drawRandomWalk{4}{{0,1}}{0}{}{}{};
        \drawRandomWalk{5}{{1,0,1,0}}{0}{}{red}{};
    \end{tikzpicture} \\
    b\underbrace{a\dots a}_{\geq d}
\end{array}
&
\longrightarrow
&
\begin{array}{c}
   \begin{array}{c}
    \begin{tikzpicture}
        \pgftransformcm{0.24}{0}{0}{0.3}{\pgfpoint{0cm}{0cm}}
        \tikzstyle{every node}=[minimum size=0.13cm,inner sep=0]
        \drawRandomWalk{0}{{0,-1,0,1,0,1,0,1,0}}{0}{}{red}{};
        \drawRandomWalk{8}{{0,1,0}}{99}{red}{red}{red};
    \end{tikzpicture} \\
    ba\dots aa
   \end{array}
    \text{ or }
   \begin{array}{c}
    \begin{tikzpicture}
        \pgftransformcm{0.24}{0}{0}{0.3}{\pgfpoint{0cm}{0cm}}
        \tikzstyle{every node}=[minimum size=0.13cm,inner sep=0]
        \drawRandomWalk{0}{{0,-1,0,1,0,1,0,1,0}}{0}{}{red}{};
        \drawRandomWalk{8}{{0,-1,0}}{99}{red}{red}{red};
    \end{tikzpicture} \\
    ba\dots ab
   \end{array}
    \text{ or }
   \begin{array}{c}
    \begin{tikzpicture}
        \pgftransformcm{0.24}{0}{0}{0.3}{\pgfpoint{0cm}{0cm}}
        \tikzstyle{every node}=[minimum size=0.13cm,inner sep=0]
        \drawRandomWalk{0}{{0,-1,0,1,0,1,0,1,0}}{0}{}{red}{};
        \drawRandomWalk{8}{{0,1,2,1}}{99}{red}{red}{red};
    \end{tikzpicture} \\
    ba\dots aA
   \end{array}
\end{array}
\end{array}
\]
\caption{The possible transitions from state $a_k$ for $1 \leq k < d$ (on the top) and for $k=d$ (on the bottom), as determined by \eqref{eq:local-limit-d-legal-word}. The transitions from state $A$ are analogous, and the transitions from states $b_k$ and $B$ are symmetric.}
\label{fig:transitions1}
\end{figure}
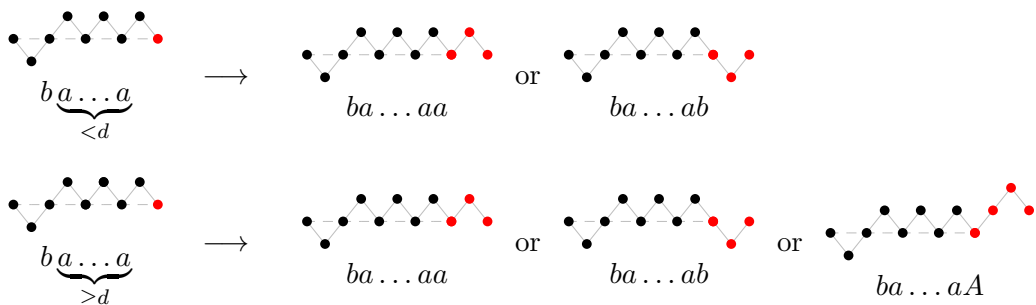

For $n\ge 1$, define
\begin{equation*}
  W_n := L_n(f_n).
\end{equation*}
The following lemma links the uniform distribution on homomorphisms
to the above Markov chain. For any $1 \leq n \leq \infty$ and any
word $x \in \Omega_{n,d}$, there exists a unique trajectory
$(s_1,s_2,\dots,s_{|x|})$ such that $s_i \in \tilde{\Sigma}$, $s_1
\in \{a_d,b_d,A,B\}$ and $p(s_i,s_{i+1})>0$, which generates the
word $x$ by the process of dropping the subscripts of the symbols in
$\tilde{\Sigma}$. For a finite word $x$, define $\text{State}(x) :=
s_{|x|}$ to be the final state of this trajectory. Let $P_k \colon \Sigma^* \to \Sigma^k$ denote the truncation to length $k$.
\begin{lemma}\label{lem:Markov_limit}
For any $u \in \Sigma$ and $x \in \Sigma^*$ such that $x$ and $x
\concat (u)$ are $d$-legal, we have
  \begin{equation}\label{eq:local-limit-convergence-of-trans-prob}
    \begin{aligned}
    \Pr(W_n(1)=u) &\xrightarrow[n \to \infty]{} \pi(\text{State}(u)) \\
    \Pr\big(W_n(|x|+1)=u ~|~ P_{|x|}(W_n)=x\big) &\xrightarrow[n \to \infty]{}
    p(\text{State}(x),\text{State}(x \concat (u)) .
    \end{aligned}
\end{equation}
\end{lemma}
\begin{proof}
For a $d$-legal word $x \in \Sigma^*$, define $M(x):=i-1$ if
$\text{State}(x) \in \{a_i,b_i\}$ for $1 \leq i \leq d$ and
$M(x):=0$ if $\text{State}(x) \in \{A,B\}$. Then, using
Claim~\ref{cl:local-limit-bijection}, we have for any $n \geq w(x)$
that
\[ |\{ P_{|x|}(W_n) = x \}| = c_{n - w(x)}(d - M(x) - 1) .\]
By Claim \ref{cl:local-limit-asymptotic-c_n}, we have $c_n(d-1) \sim
a \lambda^{n/2}$, for some constant $a>0$, and then
\eqref{eq:recursion3} gives
\begin{equation*}
c_n(d-m-1) \sim \displaystyle{a\lambda^{n/2+m} - a\lambda^{n/2 +
m-1} - \dots - a\lambda^{n/2 + 1} - a\lambda^{n/2}},\quad 1\le m<d.
\end{equation*}
Thus,
\begin{equation}\label{eq:local-limit-asymptotic-c_n-explicit}
  c_n(d-m-1)\sim a\lambda^{n/2} \frac{\lambda^{m}(\lambda - 2) +
1}{\lambda - 1}, \quad 0\le m<d.
\end{equation}
Therefore, if $\text{State}(x)=a_i$ with $1 \leq i \leq d$, then
$\text{State}(x \concat (b))=b_1$ and
\begin{align*}
\Pr\big(W_n(|x|+1)=b ~|~ P_{|x|}(W_n)=x\big)
&= \frac{|\{ P_{|x|+1}(W_n) = x \concat (b) \}|}{|\{ P_{|x|}(W_n) = x \}|} \\
&= \frac{c_{n-w(x)-2}(d-1)}{c_{n-w(x)}(d-i)} \sim \frac{\lambda -
1}{\lambda^i (\lambda-2) + \lambda} = p(a_i,b_1) .
\end{align*}
The remaining cases are handled similarly by taking the relevant
ratios. This proves the second part of
\eqref{eq:local-limit-convergence-of-trans-prob}. For the first
part, we will also need to compute the size of $\Hom(P_{n,d})$. By
Lemma \ref{lem:local-limit-hom-count},
\eqref{eq:local-limit-asymptotic-c_n-explicit} and
\eqref{eq:lambda}, we have
\begin{equation}\label{eq:local-limit-asymptotic-hom-P-n-d-explicit}
|\Hom(P_{n,d})| = 2c_{n-2}(0) + 2c_{n-3}(d-1) \sim 2a\lambda^{n/2-1} \frac{\lambda^{1/2} + 1}{\lambda - 1} .
\end{equation}
Therefore, since $\text{State}((a))=a_d$,
\[ \Pr(W_n(1) = a) =
\frac{|\{ P_{1}(W_n) = (a) \}|}{|\Hom(P_{n,d})|} =
\frac{c_{n-2}(0)}{|\Hom(P_{n,d})|} \sim \lambda^{-1/2}/2 = \pi(a_d)
.\] The remaining cases are again handled similarly. This proves the
first part of \eqref{eq:local-limit-convergence-of-trans-prob}.
\end{proof}
We continue by observing, using
Claim~\ref{cl:local-limit-recursion-for-T-extended}, that for any
$r\ge 1$ and $f\in\Hom(P_{r,d})$ there exists a $k \geq 1$ and a set
$X(f)\subset \Sigma^k$ such that we have the equality of events,
\begin{equation}\label{eq:X_f_def}
  \{B_r(f_n) = f\} = \{P_k(W_n)\in X(f)\},\quad r\le n\le \infty.
\end{equation}
For instance, one may take $X((0,-1,0,1,0)) = \{(b,a)\}$ and
$X((0,-1,0,1,0,1)) = \{(b,a,a),(b,a,A)\}$. In addition,
Lemma~\ref{lem:Markov_limit} implies that
\begin{equation}\label{eq:W_n_W_infty_limit}
  \Pr(P_k(W_n)=x) \xrightarrow[n \to \infty]{} \Pr(P_k(W_{\infty})=x) \text{ for any } k \geq 0 \text{ and } x \in
  \Sigma^k.
\end{equation}
This follows directly from
\eqref{eq:local-limit-convergence-of-trans-prob} when $x$ is
$d$-legal, and it follows trivially when $x$ is not $d$-legal since
the probabilities involved are zero.

Finally, putting together \eqref{eq:X_f_def} and
\eqref{eq:W_n_W_infty_limit}, we conclude that for any $r \geq 1$ and
$f \in \Hom(P_{r,d})$, we have
\begin{equation*}
  \lim_{n\to\infty} \Pr(B_r(f_n)=f) = \lim_{n\to\infty} \sum_{x\in
  X(f)} \Pr(P_{|x|}(W_n) = x) = \sum_{x\in
  X(f)} \Pr(P_{|x|}(W_\infty) = x) = \Pr(B_r(f_{\infty})=f),
\end{equation*}
proving \eqref{eq:local-limit-convergence-of-F}, as required.

We remark that it is now simple to derive an exact formula for the probability that $B_{r}(f_{\infty})=f$ for certain homomorphisms $f \in \Hom(P_{r,d})$. Specifically, let $f \in \Hom(P_{r,d})$ satisfy $w(L_r(f))=r$. For such $f$, one may take $X(f)=\{L_r(f)\}$. Hence, denoting $x:=L_r(f)$ and $m:=M(x)$ (defined in the proof of Lemma \ref{lem:Markov_limit}), we have using \eqref{eq:local-limit-asymptotic-c_n-explicit} and \eqref{eq:local-limit-asymptotic-hom-P-n-d-explicit} that
\begin{align*}
\Pr(B_r(f_{\infty})=f) &=
\Pr(P_{|x|}(W_{\infty})=x) =
\lim_{n \to \infty} \Pr(P_{|x|}(W_n) = x) \\
&= \lim_{n \to \infty}
\frac{c_{n-r}(d-m-1)}{|\Hom(P_{n,d})|} = \frac{1}{2} \lambda^{1-r/2}
\frac{\lambda^{m}(\lambda - 2) + 1}{\lambda^{1/2} + 1} .
\end{align*}

\section{Discussion and Open Problems}

\subsection{A continuous model} One may consider a continuous
variant of the graph homomorphisms considered here. Given a finite
connected graph $G=(V,E)$ and a vertex $v_0\in V$, let
\begin{equation*}
\text{Lip}(G,v_0):=\{f\colon V\to\R ~|~ f(v_0)=0, ~ |f(u)-f(v)|\le 1\text{ when }(u,v)\in E\}.
\end{equation*}
Thus, elements of $\text{Lip}(G,v_0)$ may be regarded as real-valued
Lipschitz functions on the graph, normalized to equal $0$ at $v_0$.
There is a natural uniform measure on $\text{Lip}(G,v_0)$ obtained
by regarding a function $f\in \text{Lip}(G,v_0)$ as a vector in
$\R^{V\setminus\{v_0\}}$ and using normalized Lebesgue measure
there. Hence, one may speak of a uniformly sampled function from
$\text{Lip}(G,v_0)$. In statistical physics terminology, this models
a random surface whose energy is defined via the Hammock potential
(see, e.g., \cite{BrascampLiebLebowitz}).

Naively, one may expect the behavior of a uniformly chosen function
$f$ from $\text{Lip}(P_{n,d},0)$ to be rather similar, perhaps up to
constants, to that of a uniformly chosen function from
$\Hom(P_{n,d},0)$. In particular, one may expect that
$\Var(f(n))\approx n 2^{-d}$ when $n 2^{-d}\ge 1$, say. However, a
different intuition comes from the following consideration. A
standard heuristic in statistical physics is that (continuous)
models of random surfaces should behave similarly to the Gaussian
free field. The Gaussian free field is again a real-valued function
$g:V\to\R$, satisfying $g(v_0)=0$, and sampled from a distribution
whose density is proportional to
\begin{equation*}
  \exp\left(-\beta \sum_{(u,v)\in E} (g(u)-g(v))^2\right),
\end{equation*}
with $\beta\in (0,\infty)$ a parameter. Analysis of the variance of
the Gaussian free field on a graph is made simple by the observation
that its distribution is a multivariate Gaussian. When $G=P_{n,d}$
and $v_0=0$ one obtains that $\Var(g(n))\approx nd^{-3}/\beta$. Thus
it is not clear whether one should expect a function $f$ sampled
uniformly from $\text{Lip}(G,v_0)$ to satisfy $\Var(f(n))\approx n
2^{-d}$ or $\Var(f(n))\approx n d^{-\alpha}$. We conjecture the
latter to be the truth. Thus, we expect a significant difference in
behavior between the homomorphism model considered in this paper and
its continuous counterpart. Consideration of the complete graph suggests that, when comparing the Gaussian free field to the continuous Lipschitz model on a regular graph, one should take $\beta$ to be one over the degree. As $P_{n,d}$ is nearly a $(2d+2)$-regular graph, this leads to the following conjecture.
\begin{conj}
There exist absolute constants $C,c>0$ such that the
following holds for any positive integers $n$ and $d$. If $f$ is uniformly sampled
from $\text{Lip}(P_{n,d},0)$ then
\begin{equation*}
  c(nd^{-2}+1)\le\Var(f(n))\le C (nd^{-2} + 1).
\end{equation*}
\end{conj}
In particular, the threshold function $d(n)$ separating the regime
of localization from the regime of delocalization is polynomial in $n$, rather than logarithmic in $n$ as is the case for the
homomorphism model. Figure \ref{fig:lipschitz-model-sample} shows a uniformly sampled function in $\text{Lip}(P_{n,d},0)$.
We remark that when considering this model it is natural to consider the non-bipartite graph $\tilde{P}_{n,d}$, which is the discrete segment $\{0,1,\dots,n\}$ with edges between vertices at distance at most $d+1$, regardless of their parity.

\newcommand{\drawRandomWalker}[6]
{
	\begin{scope}

	\tikzstyle{every node}+=[circle, fill=black]
	\setcounter{c}{0}
	\setcounter{x}{#1}

        \foreach \y in #2 {

		\ifnum\value{c} > 0
			\draw[gray,randomWalkPathStyle] (\value{x} - 1,\value{lastY}/100) -- (\value{x},\y);
			\ifnum #3 < 99 \draw[gray,dashed,randomWalkPathStyle] (\value{x} - 1,#3) -- (\value{x},#3); \fi
			\ifnum \value{c} > 1
				\node at (\value{x} - 1,\value{lastY}/100) [fill=#6] {};
			\else
				\node at (\value{x} - 1,\value{lastY}/100) [fill=#4] {};
			\fi
		\fi

		\setcounter{lastY}{100*\real{\y}}
		\stepcounter{x}
		\stepcounter{c}
        }

	\node at (\value{x} - 1,\value{lastY}/100) [fill=#5] {};

    \end{scope}
}

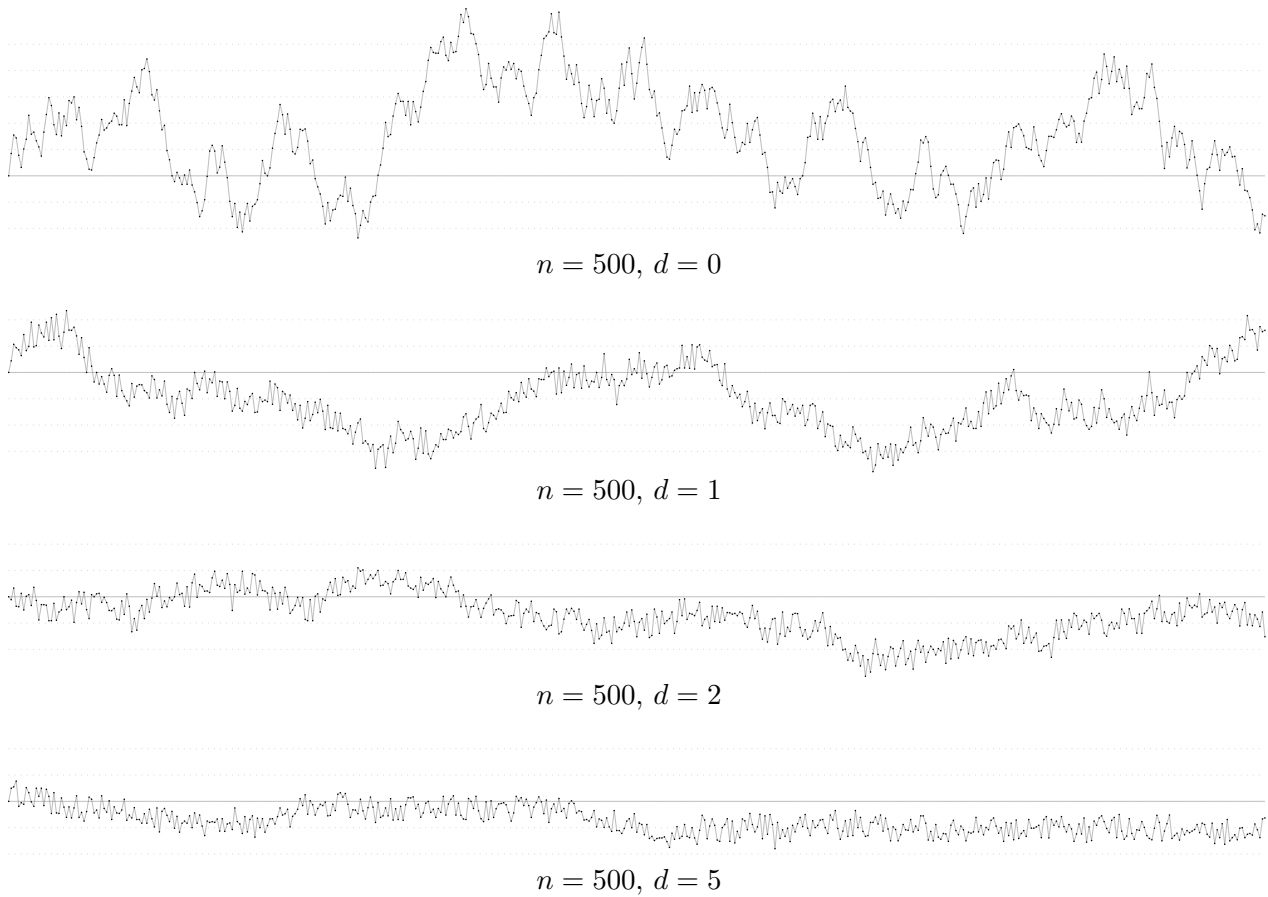
\begin{figure}[!t]
\centering
\begin{tikzpicture}

	\pgftransformcm{0.0334}{0}{0}{0.35}{\pgfpoint{0cm}{0cm}}
	\tikzstyle{every node}=[minimum size=0.02cm,inner sep=0]
	\draw[line width=0.1mm,gray!40,dotted](0,5)--(500,5);
	\draw[line width=0.1mm,gray!40,dotted](0,4)--(500,4);
	\draw[line width=0.1mm,gray!40,dotted](0,3)--(500,3);
	\draw[line width=0.1mm,gray!40,dotted](0,2)--(500,2);
	\draw[line width=0.1mm,gray!40,dotted](0,1)--(500,1);
	\draw[line width=0.1mm,gray!40,dotted](0,-1)--(500,-1);
	\draw[line width=0.1mm,gray!40,dotted](0,-2)--(500,-2);
	\drawRandomWalker{0}{{0.00,0.85,1.55,1.44,0.79,0.32,1.02,1.40,2.31,1.58,1.67,1.33,1.10,0.75,1.67,2.33,2.97,2.68,1.95,1.57,2.40,1.53,2.28,1.92,2.87,2.78,3.00,2.14,2.60,1.88,0.91,0.65,0.25,0.22,0.69,1.25,1.56,2.15,1.74,1.82,2.00,2.09,2.39,2.27,1.95,1.95,2.90,1.92,2.75,3.23,3.74,3.38,3.17,4.04,4.08,4.45,3.96,3.09,2.89,3.28,2.48,1.76,1.93,0.96,0.61,0.01,-0.22,0.14,-0.08,-0.34,0.04,-0.32,0.22,-0.35,-0.62,-1.02,-1.56,-1.34,-0.92,-0.02,0.95,1.17,0.95,0.12,0.34,1.15,0.53,-0.07,-0.97,-1.56,-1.05,-1.96,-1.38,-2.13,-1.46,-1.06,-1.72,-1.17,-1.09,-0.91,-0.31,0.59,0.10,-0.01,0.30,1.05,1.61,2.07,2.71,2.32,1.59,2.35,1.77,0.93,0.83,1.08,1.78,1.72,1.79,1.32,0.46,0.60,-0.11,-0.43,-0.70,-1.17,-1.84,-1.02,-1.71,-1.31,-1.27,-0.89,-0.74,-0.79,-0.06,-0.98,-0.47,-1.14,-1.45,-2.36,-1.88,-1.34,-1.60,-1.75,-1.00,-0.79,-0.76,0.03,0.42,1.06,2.04,1.49,1.57,2.27,2.74,3.14,2.83,2.84,2.30,2.83,1.94,2.82,2.45,2.06,2.61,3.21,3.58,4.40,4.89,4.70,4.66,4.66,5.10,5.28,4.59,4.38,5.07,4.69,4.72,5.30,6.13,5.83,6.35,6.05,5.42,5.39,5.02,4.61,3.81,3.29,3.47,4.27,3.71,3.38,3.38,2.80,3.72,4.15,4.02,4.33,4.05,4.28,3.69,4.05,3.98,3.42,3.04,2.76,2.30,2.97,3.15,4.04,4.59,5.23,5.32,5.48,6.16,5.45,5.39,6.22,5.29,4.55,3.68,4.21,3.89,3.47,4.23,3.56,2.76,2.21,2.64,3.45,2.92,2.25,2.87,2.92,3.70,3.34,2.39,2.90,2.14,2.01,2.47,3.34,4.26,3.45,4.32,4.86,3.88,2.93,3.52,4.31,4.89,5.24,4.24,3.29,2.98,3.23,2.29,2.40,1.82,1.24,0.71,0.64,1.17,1.77,1.58,1.76,2.47,2.50,2.66,3.42,2.42,3.17,2.48,3.16,3.48,3.26,2.55,3.31,3.37,3.23,2.77,2.35,1.82,1.44,1.76,2.72,2.03,1.46,0.88,0.94,1.28,1.19,1.99,1.28,1.75,2.09,2.23,1.55,0.82,0.90,0.32,-0.64,-0.62,-1.22,-0.26,-0.58,-0.67,-0.05,-0.48,-0.28,-0.12,-0.41,-0.71,-0.12,0.00,0.50,1.46,1.50,2.38,2.00,1.40,2.25,1.38,2.00,2.28,2.80,2.92,2.45,2.75,2.88,2.51,3.42,2.60,2.46,2.39,1.81,1.31,0.55,1.28,1.45,0.99,0.32,-0.32,-0.26,-0.86,-0.83,-0.59,-1.49,-1.07,-0.70,-1.08,-1.43,-1.25,-1.61,-0.96,-1.32,-1.00,-0.53,-0.54,0.09,0.53,1.36,1.25,1.49,1.29,0.36,-0.50,-0.83,-0.73,0.04,0.20,0.37,0.02,-0.32,-0.32,-0.70,-1.10,-1.91,-2.20,-1.55,-1.12,-0.58,-0.46,-1.19,-0.31,-0.97,-0.71,-0.08,-0.89,-0.08,0.51,0.60,0.59,-0.15,0.26,1.14,1.68,1.08,1.84,1.88,1.98,1.75,1.53,1.08,1.04,0.99,1.86,1.36,0.79,0.57,0.35,0.96,1.51,1.51,1.46,1.76,2.15,2.37,2.01,1.87,1.93,2.27,2.09,1.27,1.61,1.29,2.09,2.38,2.93,2.98,3.71,3.95,3.14,3.66,4.63,4.16,3.51,3.97,4.53,3.58,4.05,3.74,3.33,4.18,3.32,2.34,2.63,2.81,2.56,2.56,2.98,3.90,3.76,4.26,3.36,2.94,2.07,1.13,0.29,1.27,0.73,1.44,1.88,1.13,1.69,1.62,1.20,0.31,0.61,1.34,0.71,0.02,-0.57,-1.28,-0.30,0.26,0.32,0.82,1.34,0.80,0.21,1.00,0.79,0.90,1.10,0.73,0.74,0.19,-0.35,0.27,-0.55,-0.59,-0.83,-1.31,-2.05,-1.84,-2.18,-1.46,-1.52}}{0}{}{}{};

\end{tikzpicture}
$n=500$, $d=0$

\makebox[0cm]{\vspace{5pt}}

\begin{tikzpicture}

	\pgftransformcm{0.0334}{0}{0}{0.35}{\pgfpoint{0cm}{0cm}}
	\tikzstyle{every node}=[minimum size=0.02cm,inner sep=0]
	\draw[line width=0.1mm,gray!40,dotted](0,2)--(500,2);
	\draw[line width=0.1mm,gray!40,dotted](0,1)--(500,1);
	\draw[line width=0.1mm,gray!40,dotted](0,-1)--(500,-1);
	\draw[line width=0.1mm,gray!40,dotted](0,-2)--(500,-2);
	\draw[line width=0.1mm,gray!40,dotted](0,-3)--(500,-3);
	\drawRandomWalker{0}{{0.00,0.45,1.07,0.94,0.86,0.65,1.45,0.83,0.99,1.92,0.97,1.02,1.80,1.51,1.35,1.91,1.23,2.09,1.24,2.21,1.39,0.91,1.54,2.35,1.61,1.60,1.73,1.40,0.69,1.32,0.57,0.01,0.95,0.24,-0.37,-0.00,-0.44,-0.16,-0.45,-0.76,-0.11,0.15,-0.11,-0.91,-0.10,-0.70,-0.92,-1.09,-1.28,-0.43,-0.95,-0.52,-0.46,-0.48,-0.70,-0.77,-0.98,-0.14,-0.28,-0.98,-0.67,-0.90,-0.34,-1.21,-1.53,-0.92,-1.74,-1.01,-0.73,-1.18,-1.64,-0.64,-0.70,-0.99,-0.01,-0.43,-0.66,-0.22,0.05,-0.77,-0.26,-0.40,0.03,-0.91,-0.33,-0.36,-1.01,-0.37,-0.72,-1.48,-1.25,-0.62,-1.36,-1.47,-1.10,-1.23,-0.93,-0.81,-1.52,-1.51,-0.94,-1.11,-1.05,-1.11,-0.34,-1.21,-0.45,-0.59,-1.44,-0.72,-0.82,-1.17,-0.87,-1.45,-1.20,-1.81,-1.26,-2.12,-1.74,-1.11,-1.84,-1.60,-1.13,-1.54,-1.03,-1.47,-2.18,-1.25,-2.10,-2.14,-1.52,-2.22,-1.56,-2.13,-2.18,-2.34,-2.04,-2.73,-2.27,-1.93,-1.78,-2.69,-2.60,-2.73,-2.99,-2.74,-3.65,-2.93,-2.84,-2.75,-3.62,-2.88,-2.43,-3.05,-2.69,-1.87,-2.18,-2.31,-2.58,-2.53,-2.24,-2.52,-3.31,-2.88,-2.12,-3.07,-2.22,-3.10,-3.29,-3.00,-2.76,-2.83,-2.33,-2.56,-2.56,-2.17,-2.21,-2.49,-2.27,-2.37,-2.26,-1.70,-1.57,-2.05,-2.54,-1.93,-1.78,-1.69,-1.22,-2.13,-2.26,-1.91,-1.56,-1.63,-1.77,-1.47,-1.18,-1.54,-1.52,-0.73,-0.96,-1.05,-0.77,-1.12,-0.38,-0.79,-0.94,-0.24,-0.64,-0.71,-0.93,-0.24,-0.28,-0.27,0.18,-0.16,0.01,0.05,-0.82,0.18,-0.78,0.05,-0.49,-0.19,0.32,-0.58,-0.09,-0.03,-0.60,0.25,-0.31,-0.56,-0.34,-0.37,0.01,-0.57,-0.20,0.71,-0.28,-0.49,-0.08,-0.28,-1.22,-0.51,-0.32,-0.50,-0.25,0.50,-0.08,0.15,-0.03,0.75,-0.22,0.44,-0.06,-0.56,-0.19,0.21,0.36,-0.44,-0.29,0.21,0.28,-0.18,-0.13,0.08,0.15,0.18,1.04,0.18,0.60,0.17,1.05,0.14,0.98,1.05,0.59,0.47,0.43,0.84,0.23,0.29,0.31,-0.31,-0.51,0.33,-0.64,-0.91,-0.51,-0.78,-0.97,-0.83,-1.31,-1.12,-0.72,-1.30,-1.93,-1.18,-1.84,-1.54,-1.39,-2.15,-1.54,-1.17,-1.65,-1.64,-1.90,-1.94,-1.00,-1.12,-1.58,-1.42,-1.23,-1.67,-1.37,-1.55,-1.25,-2.03,-1.80,-1.25,-1.19,-1.27,-2.14,-1.87,-1.50,-2.44,-1.96,-2.02,-2.06,-2.80,-2.97,-2.40,-2.84,-2.65,-2.77,-2.86,-2.07,-2.35,-2.62,-2.50,-3.06,-2.99,-3.29,-2.85,-3.78,-3.53,-2.71,-3.13,-3.28,-2.84,-3.23,-2.73,-3.54,-2.76,-3.34,-2.91,-3.06,-2.85,-2.64,-2.08,-2.77,-2.68,-3.06,-2.60,-1.73,-2.11,-2.43,-1.97,-2.08,-2.46,-2.94,-2.37,-2.17,-2.32,-2.31,-2.02,-1.74,-2.61,-1.91,-1.96,-1.77,-1.35,-2.06,-1.87,-2.13,-2.13,-1.49,-2.05,-1.16,-0.74,-1.57,-0.69,-1.39,-1.28,-0.42,-1.20,-0.78,-0.24,-0.66,-0.13,0.12,-0.69,-1.06,-0.73,-0.89,-0.86,-0.90,-1.84,-1.35,-1.57,-2.14,-1.57,-1.76,-1.91,-2.01,-2.31,-1.97,-1.38,-1.99,-1.41,-0.63,-1.04,-1.46,-1.15,-1.78,-1.73,-1.06,-1.21,-1.81,-2.03,-2.17,-2.05,-1.33,-1.93,-1.36,-0.66,-1.01,-1.71,-1.39,-1.79,-1.91,-1.85,-2.17,-1.72,-1.59,-2.39,-1.73,-1.30,-1.65,-1.83,-1.37,-0.92,-1.47,-0.88,0.02,-0.77,-1.64,-0.93,-1.26,-1.21,-0.68,-1.39,-1.13,-1.14,-1.73,-1.04,-0.90,-1.02,-0.90,-0.01,-0.36,-0.05,0.05,0.25,-0.24,0.60,0.46,0.28,1.02,0.91,0.60,0.91,-0.06,0.53,0.59,0.41,1.04,0.84,0.45,0.51,1.32,1.36,1.26,2.17,1.61,1.63,1.47,0.85,1.74,1.56,1.60}}{0}{}{}{};

\end{tikzpicture}
\raisebox{0.6ex}{$n=500$, $d=1$}

\makebox[0cm]{\vspace{5pt}}

\begin{tikzpicture}

	\pgftransformcm{0.0334}{0}{0}{0.35}{\pgfpoint{0cm}{0cm}}
	\tikzstyle{every node}=[minimum size=0.02cm,inner sep=0]
	\draw[line width=0.1mm,gray!40,dotted](0,2)--(500,2);
	\draw[line width=0.1mm,gray!40,dotted](0,1)--(500,1);
	\draw[line width=0.1mm,gray!40,dotted](0,-1)--(500,-1);
	\draw[line width=0.1mm,gray!40,dotted](0,-2)--(500,-2);
	\drawRandomWalker{0}{{0.00,-0.12,0.34,-0.36,-0.38,0.13,-0.50,0.03,0.08,-0.33,0.37,-0.15,-0.89,-0.31,-0.31,-0.32,-0.92,-0.91,-0.25,-0.90,-0.59,-0.51,0.01,-0.49,-0.39,-0.14,-0.98,-0.23,0.19,-0.27,0.18,-0.49,-0.83,0.08,0.01,0.11,0.14,-0.21,-0.02,-0.23,-0.53,-0.54,-0.44,-0.85,-0.88,-0.16,-0.68,0.12,-0.36,-1.34,-0.84,-1.28,-0.54,-0.33,-0.84,-0.10,-0.58,-0.06,0.51,0.02,0.39,-0.41,-0.17,-0.38,-0.04,-0.15,-0.01,0.53,0.08,0.14,-0.37,0.38,-0.40,0.54,0.87,-0.03,0.46,0.40,0.21,0.23,0.16,0.72,0.98,0.44,0.40,0.64,0.32,0.88,0.47,-0.52,0.46,0.18,0.24,0.97,0.10,0.34,0.16,0.89,0.83,0.07,0.60,0.24,0.22,0.15,-0.51,0.37,0.14,-0.48,0.16,-0.28,0.37,0.01,-0.64,-0.50,-0.34,-0.03,-0.35,-0.04,-0.86,-0.92,-0.11,-0.91,0.02,0.12,-0.59,-0.08,-0.15,0.28,0.56,0.17,0.70,0.05,0.09,0.74,0.56,0.61,0.19,0.42,0.32,1.10,0.97,1.03,0.49,0.67,0.84,0.69,0.83,1.01,0.43,0.45,0.59,0.51,0.12,0.57,0.68,1.00,0.66,0.67,0.91,0.42,0.34,0.29,0.41,-0.02,0.74,0.26,0.52,0.48,0.55,0.69,-0.04,0.59,0.30,0.17,0.01,0.42,0.76,0.44,0.19,0.22,-0.38,0.07,-0.37,-0.77,-0.38,-0.36,-0.11,-0.46,-0.79,-0.45,-0.24,-0.08,0.07,-0.69,-0.49,-0.46,-0.50,-0.63,-1.02,-0.74,-0.72,-0.71,-0.15,-0.23,-1.17,-0.64,-0.48,-0.58,-0.18,-0.95,-0.60,-0.65,-0.92,-1.12,-1.25,-1.12,-1.08,-0.48,-0.57,-1.42,-1.33,-0.85,-0.47,-0.65,-0.45,-0.55,-0.87,-0.47,-1.04,-1.27,-1.06,-0.86,-1.29,-1.77,-1.04,-1.60,-1.46,-0.81,-1.35,-1.31,-1.77,-0.77,-1.33,-1.16,-1.26,-1.01,-0.58,-1.11,-0.50,-0.95,-1.52,-0.97,-0.84,-1.28,-0.66,-1.56,-1.47,-0.48,-0.77,-1.40,-0.56,-1.02,-0.91,-0.83,-0.59,-1.40,-0.95,-0.12,-0.73,-0.09,-0.86,-0.64,-0.67,-0.72,-0.53,-0.21,-1.10,-0.86,-0.75,-1.20,-0.37,-0.95,-0.37,-0.98,-0.43,-0.90,-0.60,-0.62,-0.67,-0.64,-0.72,-0.58,-0.65,-1.03,-1.16,-0.80,-0.33,-0.84,-0.56,-1.23,-1.62,-1.19,-1.72,-1.23,-0.85,-0.82,-1.81,-1.41,-0.58,-1.50,-1.17,-1.32,-0.68,-0.63,-0.65,-1.18,-1.49,-1.25,-1.62,-1.54,-1.30,-0.91,-1.30,-0.98,-0.61,-1.19,-1.87,-1.14,-1.82,-1.43,-1.59,-1.55,-2.02,-1.99,-2.41,-2.16,-2.45,-1.85,-2.14,-2.64,-2.24,-3.02,-2.40,-2.87,-2.17,-1.62,-2.20,-2.62,-2.21,-2.78,-2.28,-1.92,-1.98,-2.63,-2.27,-1.96,-1.68,-2.59,-2.35,-2.07,-2.76,-2.01,-1.66,-2.12,-1.88,-2.49,-2.11,-1.89,-2.00,-2.21,-1.97,-2.06,-1.71,-2.54,-1.59,-2.37,-1.67,-1.80,-2.33,-2.02,-1.66,-1.52,-2.18,-1.60,-2.24,-1.74,-1.66,-1.65,-2.14,-1.73,-2.26,-1.89,-1.85,-1.99,-1.65,-1.62,-1.77,-1.98,-1.12,-1.22,-1.12,-1.60,-1.43,-0.89,-1.78,-1.57,-1.16,-1.75,-1.56,-1.38,-2.07,-2.06,-1.88,-1.86,-1.79,-2.30,-1.55,-0.90,-1.61,-0.89,-1.33,-1.59,-0.89,-1.16,-0.64,-0.73,-1.24,-1.11,-0.90,-0.97,-1.43,-0.64,-0.83,-0.73,-0.63,-0.90,-1.38,-1.50,-1.08,-0.97,-1.22,-1.01,-1.19,-0.48,-0.63,-1.06,-0.46,-0.83,-0.84,-1.39,-0.67,-1.40,-0.64,-0.54,-0.43,-0.18,-0.84,-0.46,0.05,-0.94,-0.58,-0.94,-1.01,-0.68,-0.60,-1.10,-0.24,-0.92,-0.75,-0.05,-0.34,-0.43,-0.32,-0.39,0.12,-0.78,-0.66,-0.62,-0.11,-0.59,-0.40,-1.01,-0.25,-0.41,-0.21,-0.68,-0.84,-0.55,-0.30,-0.47,-1.12,-0.49,-0.23,-1.07,-0.67,-1.10,-0.69,-0.85,-1.14,-0.58,-1.53}}{0}{}{}{};

\end{tikzpicture}
\raisebox{0.6ex}{$n=500$, $d=2$}

\makebox[0cm]{\vspace{5pt}}

\begin{tikzpicture}

	\pgftransformcm{0.0334}{0}{0}{0.35}{\pgfpoint{0cm}{0cm}}
	\tikzstyle{every node}=[minimum size=0.02cm,inner sep=0]
	\draw[line width=0.1mm,gray!40,dotted](0,2)--(500,2);
	\draw[line width=0.1mm,gray!40,dotted](0,1)--(500,1);
	\draw[line width=0.1mm,gray!40,dotted](0,-1)--(500,-1);
	\draw[line width=0.1mm,gray!40,dotted](0,-2)--(500,-2);
	\drawRandomWalker{0}{{0.00,0.50,0.55,0.78,-0.20,-0.06,-0.12,0.53,0.33,-0.03,-0.19,0.51,0.19,0.51,-0.16,0.46,0.20,-0.46,0.33,-0.45,-0.46,0.14,-0.22,-0.61,-0.21,-0.63,-0.25,0.21,-0.61,-0.43,-0.75,-0.40,0.16,0.08,-0.45,-0.33,-0.64,0.21,-0.27,-0.42,0.03,-0.69,-0.21,-0.44,-0.46,-0.34,0.09,-0.72,-0.53,-0.64,-0.24,-0.70,-0.69,-0.84,-0.47,-0.83,-0.31,-0.67,-0.97,-0.26,-1.17,-0.38,-0.27,-0.54,-0.77,-0.39,-1.02,-0.57,-1.18,-0.83,-0.76,-0.84,-0.55,-1.01,-0.44,-0.84,-0.52,-0.95,-1.31,-0.76,-0.81,-0.78,-0.80,-0.68,-1.14,-0.67,-1.15,-0.86,-0.82,-1.07,-0.25,-0.78,-0.86,-1.23,-0.59,-1.16,-0.69,-0.56,-1.36,-0.65,-1.01,-0.66,-1.18,-0.93,-0.68,-0.66,-0.87,-0.58,-0.25,-0.52,-0.48,-0.30,-0.97,-0.57,-0.47,-0.30,0.10,-0.17,0.25,-0.49,0.03,-0.13,-0.52,-0.28,-0.12,-0.16,-0.04,-0.56,-0.54,-0.11,-0.35,0.27,0.33,0.17,0.31,-0.12,-0.62,-0.32,-0.34,-0.47,0.09,-0.37,-0.68,-0.24,0.17,-0.23,-0.81,-0.06,-0.61,-0.56,-0.25,-0.55,0.10,-0.71,-0.29,-0.69,-0.20,-0.72,-0.14,0.15,0.13,0.18,-0.29,-0.12,-0.60,-0.44,-0.07,0.00,-0.43,-0.31,-0.62,-0.27,-0.43,0.20,-0.53,-0.10,0.08,0.21,-0.53,-0.07,-0.54,-0.68,-0.70,0.11,0.00,-0.28,0.20,-0.36,-0.67,-0.78,-0.03,-0.69,0.03,-0.61,-0.29,-0.10,-0.09,-0.06,-0.14,-0.43,-0.76,-0.44,0.19,0.03,0.15,-0.17,0.08,-0.17,-0.61,-0.23,-0.05,0.18,-0.03,-0.47,0.05,-0.80,-0.29,-0.10,-0.43,0.12,-0.39,-0.28,-0.63,-0.08,0.03,-0.08,-0.45,-0.37,-0.40,-0.71,-0.69,-0.75,-0.62,-0.35,-0.97,-0.68,-0.73,-0.99,-0.64,-0.97,-1.15,-0.35,-1.11,-1.11,-1.22,-1.04,-0.75,-0.66,-1.01,-0.48,-1.18,-0.83,-0.88,-1.28,-0.89,-1.41,-1.34,-1.49,-1.42,-1.40,-1.40,-1.35,-1.54,-1.76,-1.21,-1.45,-0.96,-1.41,-1.33,-0.89,-1.33,-1.38,-1.00,-1.40,-1.64,-1.10,-0.54,-0.71,-1.43,-0.94,-0.68,-0.92,-1.47,-1.53,-1.00,-0.89,-0.66,-1.52,-0.76,-1.58,-1.15,-1.39,-0.68,-0.74,-1.09,-0.46,-0.88,-1.17,-0.98,-0.49,-1.43,-0.55,-1.26,-0.50,-0.88,-1.80,-0.81,-1.41,-1.09,-1.33,-1.03,-1.54,-0.91,-1.16,-0.85,-0.61,-0.96,-0.50,-1.03,-1.30,-0.74,-1.01,-0.92,-0.71,-0.70,-1.00,-1.22,-0.52,-0.71,-0.35,-1.04,-1.44,-1.14,-0.74,-1.34,-0.70,-1.31,-0.35,-1.00,-0.61,-1.15,-1.33,-1.31,-0.81,-0.81,-0.60,-0.87,-0.53,-1.13,-0.65,-1.05,-1.50,-0.97,-1.07,-1.50,-1.28,-0.88,-1.09,-1.10,-0.69,-1.17,-0.94,-1.10,-0.55,-0.56,-0.83,-1.47,-1.15,-1.51,-0.89,-1.37,-1.44,-1.23,-1.24,-1.52,-0.92,-0.73,-1.11,-1.27,-0.94,-1.08,-1.50,-1.18,-1.21,-1.16,-0.61,-1.17,-0.93,-1.28,-0.99,-1.08,-0.73,-0.63,-0.94,-1.21,-1.01,-1.52,-0.80,-1.31,-0.83,-1.46,-0.71,-1.41,-1.22,-1.12,-1.34,-0.69,-1.02,-1.03,-0.98,-0.91,-1.38,-0.69,-0.58,-1.40,-0.70,-0.92,-0.75,-0.98,-0.59,-0.96,-1.17,-0.88,-1.05,-1.53,-1.24,-0.86,-1.24,-0.63,-1.55,-0.98,-1.49,-0.76,-0.92,-0.64,-0.55,-0.77,-1.04,-1.19,-0.73,-0.62,-1.39,-0.51,-1.14,-1.44,-1.32,-0.67,-1.22,-1.44,-1.09,-1.10,-0.78,-0.99,-1.30,-1.41,-1.46,-1.03,-0.54,-0.80,-1.51,-1.14,-0.85,-0.53,-1.42,-1.30,-1.21,-1.48,-1.06,-1.27,-1.11,-1.19,-1.13,-1.07,-0.89,-1.19,-0.83,-0.89,-1.25,-1.37,-0.83,-1.60,-1.09,-1.39,-0.85,-1.64,-0.70,-1.50,-1.38,-1.28,-1.46,-1.02,-1.00,-1.39,-1.25,-1.22,-1.29,-0.86,-1.09,-1.39,-0.68,-0.63}}{0}{}{}{};

\end{tikzpicture}
$n=500$, $d=5$

\caption{Uniformly sampled functions in $\text{Lip}(P_{n,d},0)$. The case $d=0$ is just a random walk with independent uniform increments in $[-1,1]$. The simulation uses a Metropolis
algorithm (see, e.g., \cite[Chapter~3]{Peres2009markov}) and coupling from the past \cite{ProppWilson}.}
\label{fig:lipschitz-model-sample}
\end{figure}

\subsection{The scaling limit}

In this paper we explored the properties of a random homomorphism
for given $n$ and $d$, and also the local limit of the homomorphism
when $d$ is fixed and $n$ tends to infinity. Another limit of
interest is the scaling limit. As in many models of random walk, one
may expect that in the subcritical regime, when the range of a
homomorphism in $\Hom(P_{n,d},0)$ tends to infinity as $n$ tends to
infinity, the homomorphism has a Brownian motion scaling limit. This
is the content of the next conjecture.

\begin{conj}
There exists a function $\sigma \colon \N \to (0,\infty)$ such that
the following holds. Let $f_{n,d}$ be a uniformly chosen
homomorphism in $\Hom(P_{n,d},0)$. Define $B_{n,d} \colon [0,1] \to
\mathbb{R}$ to be the continuous function defined by
\[ B_{n,d}\left(\frac{i}{n}\right) := \frac{f_{n,d}(i)}{\sigma(d)\sqrt{n}} \]
and interpolated linearly between these points. If $d(n) - \log_2 n
\to -\infty$ as $n \to \infty$, then $B_{n,d(n)}$ converges in
distribution as $n \to \infty$ to a standard Brownian motion on
$[0,1]$.
\end{conj}

An educated guess for the function $\sigma$ may be obtained as
follows. Recall the local limit $f_{\infty,d}$ from Section
\ref{sec:local-limits}. The fact that $f_{\infty, d}$ may be
described via a Markov chain simplifies the analysis of its scaling
limit. Define
\begin{equation*}
  \sigma'(d)^2 :=
  \frac{(\lambda(d)-2)(\lambda(d)-1)}{4+(2d+1)(\lambda(d)-2)},
\end{equation*}
where $\lambda(d)$ is defined in
Theorem~\ref{thm:local-limit-hom-count}. Observe that, by
Claim~\ref{cl:lambda_asymptotics},
\[ \sigma'(d)^2 = 2^{-d-3/2}(1 - o(1)) \quad \text{as } d \to \infty . \]
Then, defining the continuous function
\[ B'_{n,d}\left(\frac{i}{n}\right) := \frac{f_{\infty,d}(i)}{\sigma'(d)\sqrt{n}}, \]
interpolated linearly between these points, it may be shown that for
any fixed $d$ the process $B'_{n,d}$ converges in distribution as
$n\to\infty$ to a standard Brownian motion on $[0,1]$. Thus, it seems
plausible that the $\sigma(d)$ of the above conjecture equals
$\sigma'(d)$.

\bibliographystyle{amsplain}
\nocite{*}
\bibliography{article}

\end{document}